\documentclass[11pt]{article}
\usepackage{amsmath,amsthm,amssymb,amsfonts}
\usepackage{hyperref}
\usepackage{authblk}
\usepackage{xcolor}
\hypersetup{
	colorlinks,
	linkcolor={red!50!black},
	citecolor={blue!50!black},
	urlcolor={blue!80!black}
}
\usepackage{geometry}
\usepackage{multirow}
\newtheorem{definition}{Definition}[section]
\newtheorem{theorem}{Theorem}[section]
\newtheorem{lemma}{Lemma}[section]
\theoremstyle{definition}
\newtheorem{assumption}{Assumption}[section]
\newtheorem{remark}{Remark}[section]
\newtheorem{corollary}{Corollary}[section]
\newtheorem{proposition}{Proposition}[section]
\newtheorem{example}{Example}[section]

\title{On the Non-Asymptotic Properties of Regularized M-estimators}
\author{Demian Pouzo\thanks{Contact: 530 Evans Hall \# 3880, Berkeley CA 94720. E-mail: dpouzo[at]econ.berkeley.edu. I would like to thank Noureddine El-Karoui, Nacho Esponda,  Bryan Graham, Michael Jansson, Pat Kline, Elena Manresa, Anna Mikusheva, Jim Powell, Andres Rodriguez-Clare and Andres Santos for their comments. All errors are mine.}}
\affil{Dept. of Economics, UC Berkeley}

\begin{document}
	\maketitle
	
	\begin{abstract}
		We propose a general framework for regularization in M-estimation problems under time dependent (absolutely regular-mixing) data which encompasses many of the existing estimators. We derive non-asymptotic concentration bounds for the regularized M-estimator. Our results exhibit a variance-bias trade-off, with the variance term being governed by a novel measure of the complexity of the parameter set. We also show that the mixing structure affect the variance term by scaling the number of observations; depending on the decay rate of the mixing coefficients, this scaling can even affect the asymptotic behavior. Finally, we propose a data-driven method for choosing the tuning parameters of the regularized estimator which yield the same (up to constants) concentration bound as one that optimally balances the (squared) bias and variance terms. We illustrate the results with several canonical examples. 
	\end{abstract}
	
	\section{Introduction} 
	
Regularized M-estimators are ubiquitous in econometrics and statistics. They appear in several models, ranging from semi-/non-parametric models such as semi-/non-parametric regressions and semi-/non-parametric likelihood models, to high-dimensional regression models, just to name a few.\footnote{See \cite{Bickel-Li-TEST06} for references and discussion.} In these models, the standard M-estimator might be ill-defined or ill-behaved so a regularized version of the M-estimator is needed. A few examples of these estimators which are widely used by practitioners are: series/sieves-based estimators (e.g. \cite{Grenander1981}, \cite{GemanHwang1982}, \cite{GallantNychka1987} and \cite{shen1994}), Kernel-based estimators (e.g. Nadaraya and Watson estimators) and penalized estimators (e.g. \cite{shen1997}, and \cite{EggLaRic2001} and references therein) --- which include LASSO and Ridge regressions in high-dimensional models (e.g. \cite{VdG-Buhlmann11} and references therein).

Even though regularization methods are a powerful tool to study otherwise ill-posed problems, they rely on tuning parameters that ought to be chosen by the researcher. E.g., the number of terms in the series/sieves estimators, the bandwidth in the Kernel-based estimators or the scale of the penalization in penalized-based estimators.

 
 Our goal is to propose an unifying framework that encompasses, among others, the aforementioned models and allow us to study the effect of these tuning parameters on the behavior of the estimator, how this effect changes with the number of observations and also with the time-dependence structure in the data. Moreover, we provide a data-driven method to choose the tuning parameters. Our framework allows for (potentially) infinite-dimensional parameters which are identified by minimization of a criterion function, and also for dependence in the data --- quantified using $\beta$-mixing (or regular mixing) conditions.\footnote{See \cite{Bradley2005} for references and discussions of mixing processes.} By focusing on time-dependent data we extend the scope of high-dimensional models to many applications in fields like economics and finance, where time-dependent data is ubiquitous (e.g. \cite{SongBickel2011, FanLvQi2011}). Also, from a conceptual point of view, time-dependent data provides a good laboratory to study departures from the standard setup of i.i.d. observations.

 The standard approach for handling tuning parameters, especially in semi-/non-parametric problems, relies on asymptotic theory. In this paper we take another route and we, instead, establish \emph{non-asymptotic} results. An appealing feature of this approach is that it allows us to understand the incidence of tuning parameters on the properties of the estimator for a given sample size, and also, how this depends on the underlying data structure, namely its dependence structure. Overall, our results provide a more accurate description of the estimator's behavior than the standard asymptotic ones, and even extend them in some settings.
 
 
 In order to assess the behavior of our estimator, we focus on the so-called concentration properties. That is, we provide non-asymptotic bounds for the probability that our estimator is within a certain ball of the identified parameter. We characterize how the sample size, the tuning parameters and the dependence structure affect the radius of this ball. This radius, called the concentration rate, should be thought as analogous to the convergence rate in asymptotic theory. We think that studying concentration properties are a necessary and important first step for assessing the finite sample behavior of regularized estimators, and also constitute the foundation for establishing the asymptotic results of consistency and convergence rates.
 
 The concentration rate is comprised of two familiar terms: a \textquotedblleft variance" term and a \textquotedblleft (squared) bias" term.  The behavior of the former term depends on (i) how \textquotedblleft large/complex'' the underlying parameter set is, and on (ii) the so-called effective sample size. For i.i.d. data, the quantity in (ii) is simply the number of observations; for dependent data, however, it can be strictly (even asymptotically) smaller than the number of observations. The effective sample size depends on the dependence structure, quantified by the $\beta$-mixing condition, and captures the intuition that high dependence in the data reduces the information of each observation. Regarding the quantity in (i), one contribution of this paper is to propose a, to our knowledge, novel measure of complexity of the parameter set.

 More precisely, we first show that the behavior of the \textquotedblleft variance" term is governed almost entirely by the behavior of the supremum of the scaled (regularized) criterion function over the parameter set. The behavior of this quantity in turn depends on how \textquotedblleft large or complex'' the parameter set is; we quantify this by defining a new measure of complexity 
 inspired by Talagrand's Generic Chaining approach (\cite{talagrand1996}). 
 In fact, if the loss function is Lipschitz, our measure of complexity is proportional to Talagrand's Generic Chaining bound. Therefore, by Talagrand's results (\cite{talagrand2005},\cite{talagrand2014}) our measure of complexity is bounded above by the one based on metric-entropies (\cite{Dudley-1967}); i.e., one can do no worse with our measure than with the more standard one used in the literature. Moreover, in some cases, our measure is also bounded above by the expectation of the supremum of a Gaussian process, which is fairly ``easy" to bound object.\footnote{In \cite{2015arXivDemian}, we show that our measure is also a  lower bound for bracketing-entropies (\cite{ossiander1987}).}

 Second, we show that the effect of the dependence structure of the data on the concentration rate can be summarized by a re-scaling of the number of observations. That is, we show that ``variance" term (or rather a bound of it) is analogous to the one obtained in the i.i.d. data case, but with a modified (smaller) sample size, instead of the actual sample size. 
 We call this the \emph{effective number of observations}.  In order to fix ideas, consider a very simple linear regression model with $k$ covariates and with $m$-dependent data.\footnote{That is, an observation at time $t$ is independent for all observations at time $t-m$, $t-m-1$, etc. However, it can be correlated with observations at time $t-1$,...,$t-m+1$.} For the case of i.i.d. data, it is not hard to show 
 that the \textquotedblleft variance" term is of order $\sqrt{k /n }$ where $n$ is the number of observations in the sample. With $m$-dependent data, however, we show that the \textquotedblleft variance" term is of order $\sqrt{m k /n }$. I.e., is as if we only had $n/m$ observations for computing the estimator, as opposed to the original $n$; $n/m$ is our effective number of observations for this case. Intuitively, this loss reflects the fact that, without any further restrictions on the correlation within the ``window" of m-observations, is as if we can only use $n/m$ observations.

 This example serves to illustrate another point regarding the usefulness of non-asymptotic results. For $m$ fixed, the previous discussion implies that, asymptotically, the effective and actual number of observations coincide,  and thus the \textquotedblleft variance" term (and ultimately the concentration rate) is asymptotically the same as the one for the i.i.d. case. However, in finite samples, when $m$ is comparable to the number of observations, the effective number of observations can be small, ultimately yielding a concentration rate that can be slower than the one for the i.i.d. case and the one predicted by the asymptotic theory. We provide numerical simulations that allow us to quantify this point in a regression setting. 
 
 The aforementioned observation holds more generally, for cases where the $\beta$-mixing coefficients decay \textquotedblleft fast enough'' to zero (the exact rate is established in the paper). This result is consistent with the existing asymptotic results which show that the convergence rates for this case coincide with those for the i.i.d. case (e.g. \cite{CS-1998} and \cite{ChenLiao2013}). Our non-asymptotic approach thus offers a sharper characterization of the role of dependency than the one provided by the asymptotic approach. On the other hand, to our knowledge, there are no general asymptotic results for the case of ``slow-decaying" $\beta$-mixing coefficients. Here, we provide non-asymptotic concentration results, and show that in this case, the concentration rates are slower than in the i.i.d. case, even asymptotically. This result thus extends existing asymptotic results to situations where the dependency in the data decays ``slowly''.

%
%
%
%
%
 
  Finally, on a more technical point, our derivations 
  show that under general $\beta$-mixing structure, the natural topology for determining the size/complexity of the parameter set is given by a \emph{family} of norms, each depending on the mixing structure. This is in contrast to the cases of i.i.d data or ``fast-decaying" $\beta$-mixing coefficients where a single norm is used (cf. \cite{DMR1995} and \cite{CS-1998}).

  We conclude with an application of our concentration result to the problem of selection of the tuning parameter. We provide a data-driven method for choosing the tuning parameter, which relates to Lepskii's method (\cite{Lepskii1991}) and is an adaptation to M-estimation of the one first proposed by \cite{PereverzevSchock2006}. The motivation for this method is as follows. In many cases choosing the tuning parameter to  balance the \textquotedblleft (squared) bias" and the \textquotedblleft variance" terms yields good convergence rates, even achieving min-max convergence rates (see \cite{Barron1999} and refences therein). This method, however, requires knowing the ``bias" term which is typically unknown to the practitioner. By exploiting our concentration results, we show that our method yields (up to constants) the same concentration rate as the aforementioned choice, but with the salient feature that it does not rely on any knowledge of the \textquotedblleft bias" term. We thus view our method as a viable alternative to Cross Validation methods which require data splitting and restrictions on the shape of the loss function (\cite{shao1997},\cite{arlot2010}) both of which can be delicate in general frameworks with time-dependent data.


%
\medskip 

\textbf{Related Literature.}  Conceptually our general setup is based on the excellent reviews by \cite{Bickel-Li-TEST06} and \cite{massart2007}. The construction of our estimator is based on \cite{CP-2012} PSMD estimator (adapted to M-estimation) and \cite{Chen2013}.  

We build on and contribute to several strands of literature. First, we contribute to the extensive branch of  high dimensional models; we refer the reader to \cite{VdG-Buhlmann11} book for a thorough review of the literature and an extensive treatment of LASSO (\cite{Tibshirani1996}) and similar models. In econometrics, we refer the reader to \cite{BCH2013} review article for discussions and applications. Our paper is aligned with many papers in this literature in terms of the non-asymptotic nature of our results, however, the vast majority of these papers only derive results under i.i.d. data. In this framework, closest to ours is the paper by \cite{NRWY2012} who derives non-asymptotic results for M-estimators under (a stronger) set of assumptions (the so-called decomposability assumption and strong convexity-type assumptions). While we view our results as complementary to theirs, we try to derive the non-asymptotic results under ``minimal" assumptions; in this sense, our approach throughout the paper relates to the work by \cite{Chatterjee2013} about ``assumptionless" consistency of the LASSO. 

Second, we contribute to the vast literature of non-/semi-parametric models. We refer the reader to \cite{Powell1994} and \cite{Chen2013}  for excellent reviews; the  latter focusing on time dependent data. As opposed to high-dimensional models, in this branch there are many paper studying the behavior of estimators under time dependent data; however, most of these results are asymptotic. \cite{massart2007} offers a nice review of non-asymptotic results but for i.i.d data. From these papers, closest to ours are \cite{CS-1998} and \cite{ChenLiao2013} who derive asymptotic results (both convergence rate and inference) for sieve and penalized M estimators respectively, for $\beta$-mixing decay of the order $O(q^{-\varpi})$ with $\varpi > 2$. They rely on $L^{2}$ bracketing entropy to measure the complexity or size of the parameter set. \cite{YU1994} establishes rates of convergence empirical processes for time dependent data; we employ the coupling technique in the paper to derive exponential bounds. 

Third, our results about the complexity measure build on Talagrand's Generic Chaining approach; see \cite{talagrand2014} for a recent book treatment. To our knowledge this approach has not been used before in M-estimation problems, with the notable exception of  \cite{vandeGeer2013}. 
The idea that the notion of distance used to construct the complexity measure depends on the mixing structure was first pointed out by \cite{DMR1995} in the context of central limit results and under the assumption of sufficiently fast decaying $\beta$-mixing coefficients.\footnote{See their paper and our Section \ref{sec:discussion} for a precise statement regarding the rate of decay.} 
This last assumption ensures that a single norm suffices to construct a Ossiander's bracket entropy integral (see \cite{ossiander1987}), which the authors used as their measure of complexity. However, if the assumption is dropped --- i.e., the $\beta$-mixing coefficients decay  slowly --- then their norm, and thus their approach, does not work as stated. Our results, propose a natural extension of their insight  by allowing a family of norms to describe the relevant notion of distance used to construct the measure of complexity. 

Fourth, the approach for choosing the tuning parameter is an adaptation of the method proposed in \cite{PereverzevSchock2006}, developed for an ill-posed inverse problem with known linear operator and measurement error in the data and thus being not directly amenable for our purposes.  In statistical/econometrics applications, the paper by \cite{ChenTim2015} uses a similar method for choosing the tuning parameter under the $L^{\infty}$ norm for the non-parametric IV model under i.i.d. data. \cite{Horowitz2014} proposes an alternative way of choosing the tuning parameter for the same model but under the $L^{2}$ norm and also under i.i.d. assumption; his procedure attains the optimal $L^{2}$-norm rate up to a $\sqrt{\log(n)}$. Finally, we refer the reader to \cite{massart2007} for a review of model selection in this setup.

\medskip 

\textbf{Roadmap.} Section \ref{sec:examples} presents some illustrative examples. Section \ref{sec:prem} defines the data structure, the M-estimation model (subsection \ref{sec:model}) and the regularized M-estimator (subsection \ref{sec:reg}). Section \ref{sec:main} presents the concentration result for the regularized M-estimator and a discussion; it introduces the notion of effective number of observations and establishes bounds for our measure of complexity (subsection \ref{sec:MoC-bound}). Section \ref{sec:choice} proposes the method for choosing the tuning parameters. Section \ref{sec:simul} presents some numerical simulations. All proofs are gathered in the Appendix.

\medskip 

\textbf{Notation.} For a generic Polish space $\mathbb{X}$ the associated $\sigma$-algebra is the Borel $\sigma$-algebra. All functions that we define from $\mathbb{X}$ are taken to be Borel measurable. For any set $A \subseteq \mathbb{X}$ and distance function $d$, $d(x,A) \equiv \inf_{x' \in A} d(x,x')$. All statements involving measurable functions are taken to hold almost surely with respect to the true probability (which we denote as $\mathbf{P}$ below). For any probability measure $P$, $E_{P}[.]$ denotes the expectation with respect to $P$; sometimes we omit the dependence, in this cases the expectation is taken with respect to the true probability measure $\mathbf{P}$. 
For any function $f$ from $\mathbb{X}$ to $\mathbb{Y}$, $x \mapsto f(x)$ denotes the function and $x_{1} \mapsto f(x_{1},x_{2})$ is used when viewing $f$ as a function of $x_{1}$, keeping $x_{2}$ fixed. For two real-valued sequences  $(x_{n})_{n}$ and $(y_{n})_{n}$, $x_{n} \precsim y_{n}$ ( $x_{n} \succsim y_{n}$) means that there exists a finite universal constant, $C \geq 1$, such that $x_{n} \leq C y_{n}$ ( $C x_{n} \geq y_{n}$);  $x_{n} \asymp y_{n}$ means that both $x_{n} \precsim y_{n}$ and $x_{n} \succsim y_{n}$ hold.  For any given real-valued sequence, $(x_{k})_{k}$, $x_{k} \downarrow a$ denotes that (a) the sequence is decreasing and (b) its limit is $a$; $\uparrow$ is defined analogously. Finally, $\mathbf{N}$ denotes $\{...,-1,0,1,...\}$, $\mathbb{N} = \{ 1,2,3,...\}$ and $\mathbb{N}_{0} = \mathbb{N} \cup \{ 0 \}$.

\section{Illustrative Examples}
\label{sec:examples}

The following example presents a high-dimensional quantile regression (HD-QR) model. Quantile regression (\cite{koenker1978regression}) is a widely used statistical method that provides an alternative to linear regressions models by capturing the heterogeneous impact of regressors on different parts of the distribution. Recently, many papers (\cite{BelloniChern2011}, \cite{koenker2004quantile}, \cite{wang2012quantile}, \cite{wu2009variable}) extend the original framework to high-dimensional settings. Here, we extend it further to allow for dependent data, and also derive the results without imposing sparsity (akin to \cite{Chatterjee2013}). By doing this, we hope our results extend the scope of HD-QR models to cover applications in finance and economics, where time-series data is ubiquitious.\footnote{Related models are also \cite{CW1999} who studied a non-parametric version for dependent data; \cite{Wu2008} who studied kernel estimation of conditional quantiles for dependent data; and \cite{KoenkerXiao2006} who studied a parametric quantile auto-regression model.}

%

%

\medskip

\begin{example}[HD-QR] \label{exa:HD-QR} 
	Let $Y_{i} = X^{T}_{i}\theta_{\ast} + U_{i}$ with $E[1\{U_{i} \leq 0 \}|X_{i} ] = \tau $ for some $\tau \in (0,1)$ and $i \in \{1,...,n\}$, where $\theta_{\ast} \in \Theta = \mathbb{R}^{d}$ and $d \in \mathbb{N}$ is fixed (albeit might be much greater than $n$). The regressors $X$ could include previous lags of $Y$ as well as other (lagged) variables. 
	We assume that $F(\cdot|X)$ (the true conditional cdf) admits a continuously differentiable pdf (with respect to Lebesgue) $f(\cdot|X)$, 
	and $E[|X|] <\infty$, $E[|Y|] <\infty$, $e_{min}(E[XX^{T}])>0$ and $\mathbb{E}_{\pi_{0}} \equiv E[|e_{max}(XX^{T})|^{\pi_{0}}] < \infty$ for some $\pi_{0} > 2$.\footnote{$e_{max}(A)$ and $e_{min}(A)$ are the maximum and minimum eigenvalue of matrix $A$.}
	
	The parameter of interest can be viewed as solving $\theta_{\ast} = \arg\min_{\theta \in \Theta} E_{\mathbf{P}}[\phi(Z,\theta)]$ with $Z=(X,Y)$ and \emph{loss function} $(z,\theta) \mapsto \phi(z,\theta) \equiv (y-x^{T}\theta)(\tau - 1\{ y - x^{T}\theta \leq 0  \})$; now and throughout the paper, $\mathbf{P}$ denotes the true probability distribution. In this model, since $d$ could be much larger than $n$, the standard plug-in estimator may fail to be well-defined and thus the problem needs to be regularized. A widely used regularized M-estimator is given by 
	\begin{align*}
		\nu_{k}(P_{n}) = \arg\min_{\theta \in \Theta} E_{P_{n}}[\phi(Z,\theta)] + \lambda_{k} Pen(\theta)
	\end{align*}
	where $P_{n}$ is the empirical distribution, $n^{-1} \sum_{i=1}^{n} \delta_{Z_{i}}$; $\lambda_{k} \geq 0 $ is the \emph{tuning parameter}; and $Pen: \Theta \rightarrow \mathbb{R}_{+}$ is a \emph{penalization function}. We present two widely used examples of penalization function: The $\ell^{1}$-norm and the weighted $\ell^{2}$-norm.  
	
	\smallskip
	
	\textbf{The $Pen=||.||_{\ell^{1}}$ Case.} By applying our result to this setting, we obtain the following \emph{concentration property} (proved in Proposition \ref{pro:HD-QR-L1-new} in the Appendix \ref{app:examples}): For any $(k,n)$ and any $u >0$, with probability higher than $1-\mathbb{G}_{0}/u$
	\begin{align} \label{eqn:HD-QR-l1-CR}
		||\sqrt{W_{k}}(\nu_{k}(P_{n})  - \theta_{\ast})   ||_{\ell^{2}} \leq & u \mathbb{K} \max\{ 1, 2^{1/\pi_{0}} \mathbb{E}_{\pi_{0}}\}  \\ \notag
		& \times \left(  \min \left\{ \sqrt{\frac{tr\{ W^{-1}_{k} \}}{n(\beta)}}  ,  \left( \frac{\log (2 d)} {n(\beta)}\right)^{1/4} \sqrt{M_{n,k}/\lambda_{k}}  \right\}  + \sqrt{\lambda_{k}||\theta_{\ast}||_{\ell^{1}}}    \right),
	\end{align}
	where $M_{n,k} = E_{P_{n}}[\phi(Z,\theta_{\ast})] + \lambda_{k} ||\theta_{\ast}||_{\ell^{1}}$;  $\mathbb{G}_{0}$ and $\mathbb{K}$ are universal constants that do not depend on $\theta_{\ast}$ or $\mathbf{P}$ and are specified below; and $W_{k} \equiv E[\underline{d}_{k}(X) XX^{T}]$ with $ \underline{d}_{k}(X) \equiv \inf_{\theta \in \{\theta \colon ||\theta||_{\ell^{1}} \leq \max\{M_{n,k},E[M_{n,k}]\}/\lambda_{k} \}} f(X^{T}\theta \mid X) $. Finally, the quantity $n(\beta) \in \mathbb{N}$ is the so-called \emph{effective number of observations} and is one of the key quantities in our results; it is defined in Theorem \ref{thm:effe-n} below and discussed further in Section \ref{sec:effe-n}. Informally, the effective number of observations summarizes the effect of the $\beta$-mixing structure on the concentration property of the regularized estimator. For instance, for i.i.d. data, $n(\beta) = n$, but for richer dependence structures, it follows that $n(\beta) < n$; moreover for $\beta$-mixing coefficients decaying ``slowly" to zero, it follows that $n(\beta)/n \rightarrow 0$.\footnote{The exact characterization of ``slow" decay is presented in Section \ref{sec:effe-n}.} 
	
	The notion of distance we use, $ ||\sqrt{W_{k}} (\cdot) ||_{\ell^{2}}$, is akin to the MSE, $||\sqrt{E[XX^{T}]}(\cdot)   ||_{\ell^{2}}$, which is routinely used in high-dimensional linear regression models. The discrepancy --- the term $\underline{d}_{k}$ --- arises  because for quantile regression the loss function $\phi$ is different to the quadratic one. We thus view our notion of distance as a natural generalization to the one typically used in high-dimensional \emph{linear} regression models.

	In expression \ref{eqn:HD-QR-l1-CR}, the right hand side term inside the parenthesis consists of the sum of a ``$\sqrt{variance}$" and ``bias" term.  The latter term is quite intuitive and vanishes as $\lambda_{k} \rightarrow 0$. We now present some remarks for the, more involved, ``$\sqrt{variance}$" term. 
	
	 First consider the case where $d << n(\beta)$. In this case no regularization is needed so we set $\lambda_{k}=0$, and the LHS of expression \ref{eqn:HD-QR-l1-CR} becomes $u \mathbb{K} \max\{ 1, 2^{1/\pi_{0}} \mathbb{E}_{\pi_{0}}\} \sqrt{tr\{ W^{-1}_{k}  \}/n(\beta)}$. Our results deliver the standard bound for finite (low) dimensional regression problems with $n(\beta)$ playing the role of $n$.\footnote{For comparison, in the linear regression case, the term  $\sqrt{tr\{ W^{-1}_{k}  \}/n(\beta)}$, becomes the standard $\sqrt{tr\{ E[XX^{T}]^{-1}  \}/n(\beta)}$. 
	 	}  As pointed out before, for i.i.d. data $n(\beta)=n$, but for other dependence structures, $n(\beta)<n$ and thus the concentration rates are slower than that for the i.i.d. case.  This remark illustrates the role that the effective number of observations plays in summarizing the effect of dependence structure on the concentration rates.
		
		Second, for the $d >> n(\beta)$ case, the concentration property reduces to: with probability higher than $1-\mathbb{G}_{0}/u$,
		\begin{align} \notag
		||\sqrt{W_{k}}(\nu_{k}(P_{n})  - \theta_{\ast})   ||_{\ell^{2}} \leq & u \mathbb{K} \max\{ 1, 2^{1/\pi_{0}} \mathbb{E}_{\pi_{0}}\}  \\ \label{eqn:HD-QR-LASSO-1}
		& \times \left(  \left( \frac{\log (2 d)} {n(\beta)}\right)^{1/4} \sqrt{M_{n,k}/\lambda_{k}}  + \sqrt{\lambda_{k}||\theta_{\ast}||_{\ell^{1}}}    \right).
		\end{align}

			\begin{remark}\label{rem:rate-LASSO}
				We did not impose that $\Theta \subseteq \{ \theta \mid ||\theta||_{\ell^{1}} \leq K  \}$ for some $K>0$. By Remark \ref{rem:Gamma-bound} in the Online Appendix \ref{app:sup-Gau}, imposing this additional condition --- which is quite standard, see e.g. \cite{Chatterjee2013}, \cite{VdG-Buhlmann11} and references therein --- delivers:
				\begin{align*}
				||\sqrt{W_{k}} ( \nu_{k}(P_{n}) - \theta_{\ast}) ||_{\ell^{2}} \leq u 2\mathbb{K} \max\{1, 2^{1/\pi_{0}} \mathbb{E}_{\pi_{0}}  \}  \left\{ \left(  \frac{\log (2 d) }{n(\beta)}      \right)^{1/4} \sqrt{K} \right\}
				\end{align*}
				with probability higher than $1-\mathbb{G}_{0}/u$, and for  $\lambda_{k} \leq \left(  \frac{\log (2 d) }{n(\beta)}      \right)^{1/2}$.
				$\triangle$	
			\end{remark}
			
			For the case of i.i.d data, the result in Remark \ref{rem:rate-LASSO} agrees (possibly up to constants) with those obtained by \cite{Chatterjee2013}, \cite{Bartlett-Mendelson-Neeman-2012} and \cite{VdG-Buhlmann11} among others, for the Linear Regression case with $\ell^{1}$ penalty. The reason for the power $1/4$ in $\log (2d)/n(\beta)$ (as opposed to, say, power of $1/2$) follows from the fact that neither sparsity nor compatibility-type conditions is assumed. In fact, a notable feature of these results is that were obtained under almost no assumptions. If $d = o (\exp\{ n(\beta) \})$, display \ref{eqn:HD-QR-LASSO-1} and Remark \ref{rem:rate-LASSO} extend the ``assumptionless" (nomenclature in \cite{Chatterjee2013}) consistency results in the aforementioned papers to a wider class of regression models and more general data structure. In particular, we allow for Quantile regression models and general $\beta$-mixing processes.			
			
			Moreover, as opposed to \cite{Chatterjee2013}, our results do not impose Gaussianity of the residual of the regression. The reason for this stems from our complexity measure, defined in Definition \ref{def:MoC}. By exploiting the results by \cite{talagrand2014}, Proposition \ref{pro:M-rate1} in Section \ref{sec:Dudley-bound} shows that, essentially, our measure of complexity is bounded \emph{above} by the expectation of suprema of Gaussian processes; i.e.,  $E \left[ \sup_{ \theta \in \{ \theta \in \Theta \colon ||\theta||_{\ell^{1}} \leq M_{n,k}/\lambda_{k}  \}}  |\sum_{j=1}^{d} \zeta_{j} \theta_{j}| \right]$ with $\zeta_{j} \sim N(0,1)$. So, by applying H\"{o}lder inequality one obtains the bound $E[\max_{1 \leq k \leq d} |\zeta_{k}| ] \frac{M_{n,k}}{\lambda_{k}} \leq \sqrt{\log (2 d)} \frac{M_{n,k}}{\lambda_{k}}$ (see Lemma \ref{lem:M-rate1b} and also \cite{Chatterjee-Jafarov-2015} Lemma A.2), which is precisely the term in the ``variance term" in expression \ref{eqn:HD-QR-LASSO-1}. This point illustrates the usefulness of our measure of complexity, vis-a-vis, say, Dudley's entropy. A final comment about the set  $\{ \theta \in \Theta \colon ||\theta||_{\ell^{1}} \leq M_{n,k}/\lambda_{k}  \}$  is in order. Since our estimator minimizes the regularized criterion function it belongs to  $\{ \theta \in \Theta \colon ||\theta||_{\ell^{1}} \leq M_{n,k}/\lambda_{k}  \}$. Thus, this set, rather than $\Theta$ is the relevant parameter set.

			%

			%

	\smallskip		
		
	\textbf{The $Pen = ||\cdot||^{2}_{\ell^{2}(p)}$ Case.} We now apply our results to the case where  $Pen=||\cdot||^{2}_{\ell^{2}(p)}$ with $||.||^{2}_{\ell^{2}(p)} = \sum_{j=1}^{d} |\theta_{j}|^{2} p_{j}$ and $p_{j} = j^{m}$ for $m \geq 0$. The point of this section is to illustrate how the choice of penalty function affects the measure of complexity of the parameter set and the bias term (but, essentially, nothing else). 
	
	For $m>1$, our theory delivers the following concentration property (proved in Proposition \ref{pro:HD-QR-L2m-CR} in Appendix \ref{app:examples}): For any $(n,k)$ and any $u>0$,
	\begin{align*}
	& || \sqrt{W_{k}}(\nu_{k}(P_{n}) - \theta_{\ast}) ||_{\ell^{2}} \leq u \mathbb{K} \max\{1, 2^{1/\pi_{0}} \mathbb{E}_{\pi_{0}}  \}  \left( \sqrt{ \lambda_{k} ||\theta_{\ast}||^{2}_{\ell^{2}(p)} }  +  \sqrt{ \frac{ \min \{  \lambda^{-1/m}_{k} , d \} } {n(\beta)}   \mathbb{B}_{k}  }  \right)
	\end{align*} 	 		 					
		with probability higher than $1-\mathbb{G}_{0}/u$, where $\mathbb{B}_{k} \equiv \left( \left( \frac{e_{min}(W_{k})}{2(m-1)}  \right)^{1/m}  +  1 \right) \frac{2}{ e_{min}(W_{k})}$. 
	
	If $d << n(\beta)$, then the ``variance" term can be majorized by $\sqrt{\mathbb{B}_{k}} \sqrt{ \frac{ d }{n(\beta)}} $.  But if $d>>n(\beta)$, there exists a possibly tighter upper bound given by $\sqrt{\mathbb{B}_{k} \frac{  \lambda_{k}^{-1/m}}{n(\beta)}}$. By choosing $\lambda_{k}$ to balance the ``variance" and ``(squared) bias" terms, it follows that  
		\begin{align*}
		||\sqrt{W_{k}}(\nu_{k}(P_{n}) - \theta_{\ast})||_{\ell^{2}} \leq  u \mathbb{K} \max\{1, 2^{1/\pi_{0}} \mathbb{E}_{\pi_{0}}  \}  n(\beta)^{-\frac{m}{2(m+1)}} ||\theta_{\ast}||^{\frac{1}{m+1}}_{\ell^{2}(p)} \mathbb{B}_{k}^{\frac{m}{m+1}} ,
		\end{align*}
		and $\lambda_{k} = (||\theta_{\ast}||^{2}_{\ell^{2}(p)} \mathbb{B}_{k}/n(\beta))^{m/(m+1)}$. The rate, $n^{-0.5m/(m+1)}$, coincides (up to constants) with the min-max rate obtained for a normal means model over $\ell^{2}(p)$-balls; e.g. \cite{Wasserman2006} Ch. 7. Unfortunately, choosing $\lambda_{k}$ to balance the variance-bias trade-off  is typically infeasible since the bias term depends on unknown quantities such as $\theta_{\ast}$. In section \ref{sec:choice} we address this issue by proposing a data-driven method that achieves the same concentration property (up to constants). 
		
		The same remarks about ``assumptionless" consistency in the $\ell^{1}$ case applies in this case, with the caveat that now the influence of $d$ in the ``variance term" is capped by $\lambda_{k}^{-1/m}$. Consequently, the concentration rate, $n(\beta)^{-\frac{m}{2(m+1)}}$, does not diverge as $d \rightarrow \infty$. Similar result holds for sequences $(p_{j})_{j}$ that diverges faster than $j^{m}$ for $m >1$, but the situation is different for $p_{j} = j^{m}$ with $m \in [0,1]$ as the next result illustrates for the case $m=0$ (proved in Proposition \ref{pro:HD-QR-L2m-CR}): For $\lambda_{k} \leq d/(2 tr\{ W_{k}^{-1}  \}) $,
		\begin{align*}
		& || \sqrt{W_{k}}(\nu_{k}(P_{n}) - \theta_{\ast}) ||_{\ell^{2}} \leq u \mathbb{K} \max\{1, 2^{1/\pi_{0}} \mathbb{E}_{\pi_{0}}  \}  \left( \sqrt{ \lambda_{k} ||\theta_{\ast}||^{2}_{\ell^{2}(p)} }  +  \sqrt{ \frac{ 2 tr\{ W_{k}^{-1}  \} } {n(\beta)}    }  \right).
		\end{align*}
		with probability higher than $1-\mathbb{G}_{0}/u$. $\triangle$
	\end{example}		

Other examples, such as Non-parametric regression also falls into our framework.

\begin{example}[Non-Parametric Linear Regression]\label{exa:NP-LR1}
	    Let $Y_{i} = \theta_{\ast}(X_{i}) + U_{i}$ with $E_{\mathbf{P}}[U_{i}|X_{i} ] = 0 $ for $i \in \{1,..,n\}$ where $\theta_{\ast} \in \Theta $ and $\Theta$ is some convex subset of $L^{2} = L^{2}(Lebesgue)$. Our results allow for the regressors to contain lagged values of $Y$; e.g. \cite{CS-1998} for more examples. In this case, $Z=(X,Y)$ and the loss function  $(z,\theta) \mapsto \phi(z,\theta) = (y - \theta(x))^{2}$, and $\theta_{\ast} = \arg\min_{\theta \in \Theta} E_{\mathbf{P}}[\phi(Z,\theta)]$.
	 	  	
	 	  	It is well-known that the estimation problem is ill-posed and needs  to be regularized; see \cite{Bickel-Li-TEST06}. A common regularized estimator is given by $\nu_{k}(P_{n}) = \tau_{k}(P_{n})^{T}\psi(\cdot )$, where  
	 	  	\begin{align*}
	 	  		\tau_{k}(P_{n}) = \arg\min_{\tau \in \mathbb{R}^{k}} n^{-1} \sum_{i=1}^{n} (Y_{i}- \tau^{T} \psi(X_{i}))^{2} + \lambda_{k} Pen(\tau)
	 	  	\end{align*}
	 	  	and $\psi = (\psi_{1},...,\psi_{k})^{T}$ being some basis functions for $L^{2}$ and $Pen$ is a convex.\footnote{Due to space constraints, we refer the reader to the Arxiv version \cite{2015arXivDemian} for a full-treatment of this example.} $\triangle$ 	
\end{example}

The following example is designed to provide an overview of the main results derived below, in Section \ref{sec:main}. In particular, the example showcases the key arguments behind the main theorems, and also highlights the role of the so-called effective number of observations. To keep the setup as simple as possible, we abstract from any regularization. 

\begin{example}[A Simple Linear Regression Model]\label{exa:OLS-q}
Consider the linear regression model $Y_{i} = x^{T}_{i} \theta_{\ast} + U_{i}$ for $i\in \{ 1,...,n\}$, $x_{i} \in \mathbb{R}^{d}$ non-random with $n^{-1} \sum_{i=1}^{n} x_{i} x^{T}_{i} = I$, $ d < n$ and $\max_{i} ||x_{i}||_{\ell^{2}} \leq K_{0} \sqrt{d} $, and $||\theta_{\ast}||_{\ell^{2}} < K_{2}$; $(U_{i})_{i}$ with $E[U]=0$ and $E[|U_{i}|^{2}] < \infty$. Finally, we assume that $(U_{i})_{i}$ is $\mu_{0}$-block independent, i.e., for any $j,j' \in \{ 0,...,n/\mu_{0}-1\}$, $(U_{i})_{i=j\mu_{0}+1}^{j\mu_{0}+\mu_{0}} $ and $(U_{i})_{i=j'\mu_{0}+1}^{j'\mu_{0}+\mu_{0}} $  are independent.\footnote{For simplicity, we set $\mu_{0}$ such that $n/\mu_{0} -1 \in \mathbb{N}$. Also, in the Online Appendix \ref{app:OLS-q-MA} we develop the MA($\mu_{0}$) case (another $\mu_{0}$-dependent process) and argue that the results remain the same.}


In this case $\nu(P_{n})  =  n^{-1} \sum_{i=1}^{n} x_{i} Y_{i} $.
The natural notion of distance for studying concentration properties is the $\ell^{2}$ norm, 
Therefore, obtaining concentration results in this problem, boils down to finding concentration results for $||\nu(P_{n})  - \theta_{\ast} ||_{\ell^{2}}  = || n^{-1} \sum_{i=1}^{n} x_{i} U_{i}    ||_{\ell^{2}}$. Straightforward calculations imply that, for any $u>0$,
\begin{align}\label{eqn:OLS-q}
\mathbf{P} \left(  ||\nu(P_{n})  - \theta_{\ast} ||_{\ell^{2}} \geq \sqrt{d} u  \right) \leq \sum_{1 \leq l \leq d }\mathbf{P} \left( | n^{-1} \sum_{i=1}^{n} x_{i,l} U_{i}    | \geq u  \right).
\end{align}

The previous display suggests that it suffices to study the concentration properties of $\left| n^{-1} \sum_{i=1}^{n} x_{i,l} U_{i}  \right|$ for any $l=1,...,d$. In our general approach it is also the case that for obtaining concentration properties we need to bound an average, namely a centered version of the (regularized) criterion function defined in \ref{def:reg}. Due to the level of generality, however, we cannot exploit closed form solution of the estimator; as a consequence, we need to control \emph{uniformly} the centered (regularized) criterion function (see Theorem \ref{thm:gbrack} in Section \ref{sec:heur} below). It is here that the complexity of the parameter set arises as an important element in our results. In Section \ref{sec:main} we define the \emph{Measure of Complexity} and in Section \ref{sec:MoC-bound} we provide further discussions and results.\footnote{In this simple example, the scaling $\sqrt{d}$ can be viewed as a measure of the complexity of the parameter space.} 


We can cast $|n^{-1} \sum_{i=1}^{n} x_{i,l} U_{i}  |$ as $|n^{-1} \sqrt{\mu_{0}} \sum_{j=0}^{J} \varDelta_{j,l} |$ with $ \varDelta_{j,l} =\mu_{0}^{-1/2}\sum_{i=j\mu_{0}+1}^{j\mu_{0}+\mu_{0}} x_{i,l} U_{i} $ and $J=n/\mu_{0}-1$. Since $(U_{i})_{i}$ is $\mu_{0}$-block independent, this decomposition implies that $(\varDelta_{j,l})_{j=0}^{J}$ are independent from each other. Using this fact, and decomposing $x_{l}U$ into bounded and unbounded parts, we invoke Bernstein (e.g. \cite{VdV-W1996} Lemma 2.2.9) and Markov inequalities to establish the following bound (proved in Proposition \ref{pro:OLS-q} in Appendix \ref{app:examples}):\footnote{The appendix also contains a discussion about alternative bounds.} For any $n \in \mathbb{N}$, $\mu_{0}\leq n$ and $l = 1,...,d$, for any $t \geq 1$
	\begin{align*}
	\mathbf{P} \left( |n^{-1} \sum_{i=1}^{n} x_{i,l} U_{i}  | \geq t \sqrt{\frac{\mu_{0}}{n}} \right) \leq  \frac{4}{t} \sqrt{E_{\mathbf{P}} \left[ \left(\frac{|\varDelta|}{\sqrt{\mu_{0}}} \right)^{2} \right]}.
	\end{align*}
We observe that given the dependence structure, the correct order for the variance $E[(\varDelta_{l})^{2}]$ is in fact $\mu_{0}$, which also appears scaling the sample size. This observation illustrates, in a simplistic way, a key feature of our general results: The dependence structure affects the variance part in Bernstein inequality, and, in turn, the concentration rate of the estimator. So, in this case, the \emph{effective number of observations} is given by $n/\mu_{0}$. 
This proposition and expression \ref{eqn:OLS-q}, imply that the estimator satisfies a concentration property with a \emph{concentration rate} given by $\sqrt{\frac{d}{n/\mu_{0}}}$ and a \emph{concentration bound} given by $u \mapsto  \frac{4d}{u} \sqrt{E_{\mathbf{P}} \left[ \left(\frac{|\varDelta|}{\sqrt{\mu_{0}}} \right)^{2} \right]} $, i.e., for any $n$ and any $u \geq 1$, 
\begin{align}\label{eqn:toy-concen}
\mathbf{P} \left( || \nu(P_{n}) - \theta_{\ast} ||_{\ell^{2}} \geq \sqrt{\frac{d}{n/\mu_{0}}} u   \right) \leq \frac{4d}{u} \sqrt{E_{\mathbf{P}} \left[ \left(\frac{|\varDelta|}{\sqrt{\mu_{0}}} \right)^{2} \right]}.
\end{align} 
At the core of the proof of our main theorem, Theorem \ref{thm:concen-main} below, there is also Bernstein inequality (see Lemma \ref{lem:bere} in Appendix  \ref{app:gbrack}), and the idea of approximating the original data with independent blocks. There are, however, some key differences with the calculations leading to expression \ref{eqn:toy-concen}. First, in this simple application, the approximation by independent block is exact, but for more general dependent structures an approximation error arises (see Lemma \ref{lem:beta-bdd} in Appendix \ref{app:gbrack} and the discussion in Section \ref{sec:prem}). Second, while we also decompose $\phi$ into a ``bounded" part and a ``unbounded" part, we use more sophisticated arguments based on the results in \cite{DMR1995} that only use $L^{1}$-norm bounds (see Lemmas \ref{lem:decom-L} and \ref{lem:L1-bdd} in Appendix  \ref{app:gbrack}). Both differences, however, only affect the constant in the concentration bound, not the concentration rate nor the geometric decay of the bound; in particular, our constant does not depend on $d$ (see Section \ref{sec:OLS-simple}). Finally, by abstracting from regularization, this example only illustrates the behavior of the ``variance" term in the concentration rate, ignoring the (more straightforward) ``bias" term (see Theorem \ref{thm:concen-main} for its role on the concentration rate). $\triangle$

\end{example}

 \section{The Model and the Regularized Estimator}\label{sec:prem}
 
We now introduce the data structure, define the model and the regularized estimator.
 
 \subsection{The Data}
 
  Let $\omega \equiv (...,Z_{-1},Z_{0},Z_{1},...)$ with $Z \in \mathbb{Z} \subseteq \mathbb{R}^{|\mathbb{Z}|}$ for some $|\mathbb{Z}| \in \mathbb{N}$ finite. Let $\Omega = (\mathbb{Z})^{\mathbf{N}}$ be the sample space. Let $\mathcal{Z}_{m}^{n}$ be the $\sigma$-algebra generated by $Z_{m:n} = (Z_{m},...,Z_{n})$ for any $m \leq n$. Let $\mathbf{P}$ be the true probability over $(\Omega,Borel)$. We assume that $ \mathbf{P}$ belongs to a class of stationary and $\beta$-mixing (or absolutely regular) probabilities, i.e., there exists a function $\beta : \mathbb{R}_{+} \rightarrow \mathbb{R}$ such that $\lim_{q \rightarrow \infty} \beta(q) =0$ and $\sup_{t} \boldsymbol{\beta}(\mathcal{Z}^{t}_{-\infty},\mathcal{Z}_{t+q}^{+\infty}) \leq \beta(q)$ for all $q \in \mathbb{N}_{0}$ where 
\begin{align}
	\boldsymbol{\beta}(\mathcal{U},\mathcal{V}) = \frac{1}{2} \sup \left\{ \sum_{i\in I,j \in J} \left| \mathbf{P}(U_{j} \cap V_{j} ) - \mathbf{P}(U_{j}) \mathbf{P}( V_{j} )    \right|   \right\}
\end{align}
where the \textquotedblleft sup" is taken over all pairs of partitions $(U_{i})_{i\in I}$ and $(V_{i})_{i\in I}$ on $\Omega$ such that $U_{i} \in \mathcal{U}$ and $V_{i} \in \mathcal{V}$, and $\mathcal{U}$ and $\mathcal{V}$ are $\sigma$-algebras; see \cite{VolRoz1959} and \cite{DMR1995} p. 397. For technical reasons, we also require $\beta$ to be cadlag and non-increasing.

  
  The results in the paper hinge on well-known coupling results for $\beta$-mixing processes; in particular, following \cite{YU1994}, in our proofs we use the following fact:\footnote{See also \cite{CS-1998}.} For any $q \in \mathbb{N}$, let $(Z^{\ast}_{i})_{i \in \mathbb{N}_{0}}$ be independent of $(Z_{i})_{i \in \mathbb{N}_{0}}$ and such that: (1) $U^{\ast}_{i}(q) \equiv (Z^{\ast}_{iq+1},...,Z^{\ast}_{iq+q})$ has the same distribution as $U_{i}(q) \equiv (Z_{iq+1},...,Z_{iq+q})$ for any $i=0,1,...$; (2) The sequence $(U^{\ast}_{2i}(q))_{i\geq 0}$ is i.i.d. and so is $(U^{\ast}_{2i+1}(q))_{i \geq 0}$; and (3) $\mathbb{P}(U^{\ast}_{i}(q) \ne U_{i}(q) ) \leq \beta(q)$ for any $i=0,1,...$.
  where $\mathbb{P}$ is the product measure of $\mathbf{P}$ and $\mathbf{P}^{\ast}$ --- the probability distribution of $\omega^{\ast} \equiv (...,Z^{\ast}_{-1},Z^{\ast}_{0},Z^{\ast}_{1},...)$; see \cite{DL2002} pp. 144-152 and  \cite{MP2002} Theorem 2.9 and references therein.


%
%
%

   	\subsection{The Model}
   	\label{sec:model}
   	
    Consider a Banach space $(\Theta,||.||_{\Theta})$ and some subset $\mathcal{P}$ of the space of Borel probability measures over $\Omega$ that are stationary and $\beta$-mixing, and $\mathbf{P} \in \mathcal{P}$. Our interest is the estimation of a parameter $\theta_{\ast} \in \Theta$ such that for some given \emph{criterion function} $Q : \mathcal{P} \cup \mathcal{D} \times \Theta \rightarrow \mathbb{R}_{+}$ (where $\mathcal{D}$ the set of discretely supported distributions)\footnote{The reason for including $\mathcal{D}$ in the domain of $Q$ is to ensure that, for our estimator, $Q$ is well-defined once it is evaluated in the empirical distribution. We could relax this assumption by generalizing the definition of regularized M-estimator below.}
    \begin{align*}
    	Q(\theta_{\ast},\mathbf{P}) \leq Q(\theta,\mathbf{P}),~\forall \theta \in \Theta.
    \end{align*}
    I.e., the parameter of interest, $\theta_{\ast}$, is characterized as the minimizer of the criterion function at $\mathbf{P}$ over $\Theta$. We focus on \emph{M-estimation} problems, wherein $Q(\theta,P)=E_{P}[\phi(Z,\theta)]$ for a given \emph{loss function} $\phi : \mathbb{Z} \times \Theta \rightarrow \mathbb{R}$ such that $\{  \phi (\cdot,\theta)  \colon \theta \in \Theta  \} \subseteq L^{1}(\mathbf{P})$.

    Let $\nu : \mathcal{P} \rightarrow 2^{\Theta}$ be the \emph{parameter mapping} where
   	\begin{align}
   		\nu(P) = \arg\min_{\theta \in \Theta}  Q(\theta,P),~\forall P \in \mathcal{P}
   	\end{align} 
   	 and $\nu(P)$ be the \emph{identified parameter (set)} at $P$. Clearly, if $\nu(\mathbf{P})$ is non-empty, $\theta_{\ast} \in \nu(\mathbf{P})$. In Appendix \ref{app:prelim} we present a low-level condition that ensures the non-emptiness of $\nu(\mathbf{P})$.\footnote{The condition essentially restricts the lower-contour sets of $Q(.,\mathbf{P})$ to be compact under some topology, not necessarily the one induced by $||.||_{\Theta}$.} Henceforth, $Q(\nu(\mathbf{P}),\mathbf{P})$ should be understood as $Q(\theta,\mathbf{P})$ for some (any) $\theta \in \nu(\mathbf{P})$. Our analysis admits the identified parameter set, $\nu(\mathbf{P})$, to be a non-trivial set; i.e., the criterion may not identify the parameter of interest. In this case, our results would be about concentration at the whole identified set $\nu(\mathbf{P})$. 
  
   In this general setup, it is well-known that the mapping $\nu$ may be ill-defined or even if it is well-defined, it could be ill-behaved (e.g., discontinuous) once evaluated in the empirical distribution; see \cite{Bickel-Li-TEST06} and references therein. Therefore, we need to regularize the problem.

   	\subsection{The Regularized Estimator and Concentration Properties}
   	\label{sec:reg}
   	
   	 	 Our goal is to study the concentration properties of the class of regularize M-estimators. We now define these concepts.
   	 	 
   	 	 \medskip
   	
   	 	\textbf{The Regularized M-estimator.} In the spirit of \cite{Bickel-Li-TEST06}, we define a \emph{regularized estimator} as a sequence of set-valued functions $\boldsymbol{\nu} = (\nu_{k})_{k \in \mathbb{N}}$ with $\nu_{k} : \mathcal{P} \cup \mathcal{D} \rightarrow 2^{\Theta}$ such that $\nu_{k}(P_{n})$ is a singleton, where $P_{n} \equiv n^{-1} \sum_{i=1}^{n} \delta_{Z_{i}}$ is the empirical distribution. 
\begin{remark}
    The regularized estimator needs only to be constructed at $P_{n}$, but is convenient to define it at $\mathbf{P}$ too, and interpret this quantity as the \textquotedblleft regularized parameter" (the precise definition is given below). Abusing terminology, we also call the sequence evaluated at $P_{n}$, $(\nu_{k}(P_{n}))_{n,k}$, a regularized estimator. $\triangle$
\end{remark}

   To define a regularized M-estimator, we need the following definition of regularization structure. This definition also clarifies the role of $k$ in the definition of regularized estimator.
 
\begin{definition}
	A \emph{regularization structure} is a tuple $\langle  \{ \lambda_{k}, \Theta_{k},   \}_{k=1}^{\infty}, Pen \rangle $ such that \begin{enumerate}
		\item (Penalization parameter) For each $k \in \mathbb{N}$, $\lambda_{k} > 0$ and $\lambda_{k} \downarrow 0$.
		\item (Sieve Spaces) For each $k \in \mathbb{N}$, $\Theta_{k} \subseteq \mathbb{R}^{k}$ is non-empty, $\tau$-closed and $\cup_{k} \Theta_{k}$ is $\tau$-dense in $\Theta$.\footnote{$\tau$ is some topology, not necessarily equal to the one induced by $||.||_{\Theta}$.}
		\item (Penalization function) $Pen : \Theta \rightarrow \mathbb{R}_{+}$ is $\tau$-continuous. 
	\end{enumerate}     
\end{definition}

For each $k \in \mathbb{N}$, let $Q_{k} = Q + \lambda_{k} Pen$ be the \emph{Regularized Criterion Function}, and


\begin{definition}[Regularized Estimator]
	\label{def:reg}
	Given a regularization structure $\langle \{ \lambda_{k}, \Theta_{k},   \}_{k=1}^{\infty},Pen \rangle$ and a positive real-valued sequence $(\eta_{n})_{n}$ such that $\eta_{n} \rightarrow 0$, the \emph{regularized M-estimator} $(\nu_{k}(P_{n}))_{n,k}$ is given by 
	\begin{align*}
		\nu_{k}(P_{n}) \in \Theta_{k}~and~Q_{k}(\nu_{k}(P_{n}),P_{n}) \leq \inf_{\theta \in \Theta_{k}} Q_{k}(\theta,P_{n}) + \eta_{n},~a.s.-\mathbf{P}
	\end{align*}
	for all $(n,k) \in \mathbb{N}^{2}$.
\end{definition}

As illustrated in the examples, this definition is quite general and encompasses many widely used estimators; e.g. such as Penalization-based LASSO or Ridge in high-dimensional models, or sieve/series and penalization for semi-/non-parametric models.

	\medskip
	
	\textbf{The Regularized Parameter Set.} For any $k \in \mathbb{N}$, let
\begin{align}
\nu_{k}(\mathbf{P}) = \arg\min_{\theta \in \Theta_{k}} Q_{k}(\theta,\mathbf{P})
\end{align}
be the \emph{regularized parameter (set)}.\footnote{Lemma \ref{lem:reg-min} in the Appendix \ref{app:prelim} provides sufficient conditions to show that $\nu_{k}(\mathbf{P})$ is non-empty.} As in the case of the identified parameter, our analysis goes through even if $\nu_{k}(\mathbf{P})$ is not a singleton. Henceforth, $Q_{k}(\nu_{k}(\mathbf{P}),\mathbf{P})$ should be understood as $Q_{k}(\theta,\mathbf{P})$ for some (any) $\theta \in \nu_{k}(\mathbf{P})$.

	\medskip

   	 Finally, we formally define the concentration property for a regularized estimator.
   	 
   	 \medskip
   	 
   	 \textbf{Concentration Property.} 	Let $\mathbf{r} = (r_{n,k})_{n,k \in \mathbb{N}^{2}}$ with $r_{n,k} : \Omega \rightarrow \mathbb{R}_{+}$ and $g : \mathbb{R}_{+} \rightarrow [0,1]$.   
   	 
   	 \begin{definition}[Concentration Property]\label{def:concen}
   	 	Given a parameter mapping $\nu$, a regularized estimator $\boldsymbol{\nu}$, $(\boldsymbol{r},g)$-concentrates around $\nu(\mathbf{P})$ under $d : \Theta^{2} \rightarrow \mathbb{R}_{+}$ if, for all $u>0$ and all $(n,k)$,
   	 	\begin{align}
   	 	\mathbf{P} \left( d(\nu_{k}(P_{n}),\nu(\mathbf{P})) \geq u r_{n,k}(\omega)    \right) \leq g(u).
   	 	\end{align}
   	 	We call $\mathbf{r}$ the \emph{concentration rate} and $g$ the \emph{concentration bound}.
   	 \end{definition} 
   	 
   	 That is, if a regularized estimator satisfies the concentration property with parameters $(\mathbf{r},g)$, with probability higher than $1-g(u)$, the estimator is within a $u r_{n,k}$-neighborhood (under $d$) of the identified set $\nu(\mathbf{P})$. In cases where $\lim_{u \rightarrow \infty} g(u) = 0$ and $r_{n,k(n)} \rightarrow 0$ as $n \rightarrow \infty$ a.s.-$\mathbf{P}$ for some $n \mapsto k(n)$, this property implies the asymptotic convergence (under $d$) at rate $(r_{n,k(n)})_{n}$ to the identified set $\nu(\mathbf{P})$.

\section{Main results}
\label{sec:main}

  We now present the main theorem of the paper which establishes a concentration property for our regularized M-estimator. We first define some necessary concepts to construct the ``bias" term and the ``variance" term of the concentration rate. In particular, for the latter term we need to define a notion of complexity of the parameter space.
  
  	Unless otherwise stated, we restrict our attention to $q \in \mathcal{Q}_{n} = \{ m \in \mathbb{N} \colon n/m \in \mathbb{N}  \}$ and $n = \prod_{i=1}^{\upsilon} p_{i}^{m_{i}}$ for some some $\upsilon \in \mathbb{N}$, $(m_{i})_{i=1}^{\upsilon} \in \mathbb{N}_{0}^{\upsilon}$ and $(p_{i})_{i=1}^{\upsilon}$ consecutive primes. This restriction is not crucial for our results and can be relaxed, but it simplifies the exposition and technical derivations. The choice of $\mathcal{Q}_{n}$ ensure that when constructing the blocks described in Section \ref{sec:prem} with length $q \in \mathcal{Q}_{n}$, the number of blocks is an integer.\footnote{For instance, this fact simplifies the decompositions in Lemma \ref{lem:bere} in the Appendix.} The restriction over $n$ is simply to ensure that $\mathcal{Q}_{n}$ is ``rich enough" and its elements are not too far apart;  see Online Appendix \ref{app:domN} for a more thorough discussion.
  
 \medskip

 \textbf{Notion of distance and the Bias Term.} For all $\theta \in \Theta$, let\footnote{Note that $\delta_{\mathbf{P}}(\cdot,\mathbf{P})$ is well-defined because $Q(\theta,\mathbf{P}) \geq Q(\nu(\mathbf{P}),\mathbf{P})$ for all $\theta \in \Theta$.}
 \begin{align*}
 \delta_{\mathbf{P}}(\theta,\nu(\mathbf{P})) =  \sqrt{ Q(\theta,\mathbf{P}) -  Q(\nu(\mathbf{P}),\mathbf{P})}.
 \end{align*}
  As illustrated in Section \ref{sec:heur} below, the proof for the concentration results amounts to studying the behavior of the criterion function. Hence, in this context, $\delta_{\mathbf{P}}$ presents itself as the natural choice to measure distance over $\Theta$ and, loosely speaking, can be viewed as a generalization of the root mean square error. This observation notwithstanding, the relevant notion of distance ultimately depends on the application at hand, and in many instances one would like a concentration property under more standard metrics such as $\ell^{p}$ or $L^{p}$ norms. For instance, in the HD-QR example \ref{exa:HD-QR}, we argue that a natural distance is $||\sqrt{W_{k}}(\cdot)||_{\ell^{2}}$, which has the property that $\delta_{\mathbf{P}}(\cdot,\nu(\mathbf{P})) \geq ||\sqrt{W_{k}}(\cdot - \nu(\mathbf{P}))||_{\ell^{2}}$. In Proposition \ref{pro:concen-w} below we generalize this idea and provide concentration results for general metrics.

 For any $k \in \mathbb{N}$, let 
 \begin{align*}
 B_{k}(\mathbf{P}) \equiv Q_{k}(\nu_{k}(\mathbf{P}) , \mathbf{P}) - Q(\nu(\mathbf{P}) , \mathbf{P})
 \end{align*}
 be the \textquotedblleft bias term" under the metric induced by $\delta_{\mathbf{P}}$. It reflects the two sources of the \textquotedblleft bias": The fact that our estimator is constructed over $\Theta_{k}$ (and not $\Theta$) and also the fact that we add a penalization term when $\lambda_{k}>0$.\footnote{Lemma \ref{lem:Bias-incr} in Appendix \ref{app:prelim} provides conditions to ensure that the bias vanishes as $k$ diverges.}

  \medskip
  
  \textbf{Measure of Complexity.} The measure of complexity is inspired by Talagrand's Generic Chaining results (see \cite{talagrand1996}, \cite{talagrand2005} and \cite{talagrand2014}), but to our knowledge the exact construction is new. 
    
   For any $A \subseteq \Theta$, let $\mathcal{F}(A) = \{ f : \mathbb{Z} \rightarrow \mathbb{R} \mid \exists \theta \in A,~f(.) = \phi(.,\theta)-\phi(.,\nu_{k}(\mathbf{P}))  \}$.\footnote{We are abusing notation by writing $\phi(.,\nu_{k}(\mathbf{P}))$. If $\nu_{k}(\mathbf{P})$ is a set, $\phi(.,\nu_{k}(\mathbf{P}))$ is defined as $\phi$ evaluated at one element of the set and fix it throughout the following arguments.}  For any $A \subseteq \Theta$, let $(\mathcal{T}_{l})_{l \in \mathbb{N}_{0}}$ be an increasing sequence of partitions (of $\mathcal{F}(A)$) such that $card(\mathcal{T}_{l}) \leq 2^{2^{l}}$ and $\mathcal{T}_{0} = \mathcal{F}(A)$. That is, each $\mathcal{T}_{l}$ consists of a partition of $\mathcal{F}(A)$ with (at most) $2^{2^{l}}$ elements. We call such sequence an \emph{admissible sequence}; we denote the set of all such sequences as $\mathbf{T}$. For any $f \in \mathcal{F}(A)$ and $l \in \mathbb{N}_{0}$, there is only one set in $\mathcal{T}_{l}$ that contains $f$; we call it $T(f,\mathcal{T}_{l})$. Finally, for any $z \in \mathbb{Z}$, let $S(f,\mathcal{T}_{l})(z) = \sup_{f_{1},f_{2} \in T(f,\mathcal{T}_{l})} |f_{1}(z) - f_{2}(z)|$.
  
 \begin{definition}[Complexity Measure]\label{def:MoC}
  For any set $A \subseteq \Theta$ and a family of quasi-norms, $\{ d_{l} : 
  \mathcal{F}(\Theta) \rightarrow \mathbb{R}_{+} \}_{l \in \mathbb{N}_{0}}$, let 
  \begin{align}
  	\gamma(A,(d_{l})_{l \in \mathbb{N}_{0}}) = \inf_{(\mathcal{T}_{l})_{l\in \mathbb{N}_{0}} \in \mathbf{T}}  \sup_{f \in \mathcal{F}(A)} \sqrt{2} \sum_{l=0}^{\infty} 2^{l/2} d_{l}(S(f,\mathcal{T}_{l}))
  \end{align}
  be the \emph{Complexity Measure of set $A$ under the family $(d_{l})_{l \in \mathbb{N}_{0} }$}. 
\end{definition}
  
  We relegate a discussion of its properties to Section \ref{sec:MoC-bound}.
  As explained below in Section \ref{sec:discussion}, the fact that the complexity depends on a \emph{family} of distances --- as opposed to only one norm/distance --- is important for our analysis.  The relevant family of norms is given by $\left( ||.||_{q_{n,k}} \right)_{k \in \mathbb{N}_{0}}$ where, for any $n$ and $k \in \mathbb{N}_{0}$,  
   \begin{align}\label{eqn:qk}
  q_{n,k} = & \min \{  s \in \mathcal{Q}_{n} \mid 0.5 \beta(s) n \leq s 2^{k+1}  \},
  \end{align}
  where for each $(n,k)$, $q_{n,k}$ acts as the parameter $q$ controlling the bound of the coupling results in Section \ref{sec:prem}. 
  And for any $f \in \mathcal{F}(\Theta)$ and $q \in \mathbb{N}$, let
  \begin{align}
  ||f||^{2}_{q} = 2\int_{0}^{1} \mu_{q}(u) Q^{2}_{f}(u) du,
  \end{align}
  where $Q_{f}$ is the quantile function of $|f|$ and\footnote{That is, for any $u \geq 0$, $Q_{f}(u) = \inf \{ s \mid H_{f}(s) \leq u  \}$ with $H_{f}(s) = \mathbf{P}(|f(Z)| > s)$.} 
  \begin{align}
  \label{eqn:muq}
  u \mapsto \mu_{q}(u) = \sum_{i=0}^{q} 1_{\{ u \leq 0.5 \beta(i)  \}} \in \{0,...,1+q\}.  
  \end{align}

  Observe that the dependence structure enters the definition of the norm $||.||_{q}$ through $\mu_{q}$. In particular, note that $||f||^{2}_{L^{2}(\mathbf{P})} = \int_{0}^{1} Q^{2}_{f}(u) du$, so if $q \mapsto \mu_{q}$ were constant, then $||.||_{q}$ is proportional to $||.||_{L^{2}(\mathbf{P})}$. However, for general $\beta$-mixing processes, $\mu_{q}$ is not constant and acts as a weight function for the quantile $Q_{f}$. 
  
  Finally, for any $A \subseteq \Theta$ and any $n$, let
  \begin{align*}
  	\gamma_{n}(A) = \gamma(A,(||.||_{q_{n,k}})_{k}).
  \end{align*}
  Abusing terminology, we call $\gamma_{n}(A)$ the \emph{Complexity Measure of set $A$ given $n$}.

  We now define the \textquotedblleft variance term" of the concentration rate. For any $k \in \mathbb{N}$ and any $M \geq 0$, let $\Theta_{k}(M) \equiv \{ \theta \in \Theta_{k} : \lambda_{k} Pen(\theta) \leq M  \}$, and let $\omega \mapsto M_{n,k}(\omega) \equiv Q_{k}(\theta_{k},P_{n}) + \eta_{n}$ for some (any) $\theta_{k} \in \Theta_{k}$. 
  It follows that our estimator belongs to $\Theta_{k}(M_{n,k}(\omega))$ a.s.-$\mathbf{P}$, and thus this is the relevant set over which we do our analysis.
  
  For any $s > 0$, let
  \begin{align}\label{eqn:H}
  	H_{n,k}(\omega)(s) \equiv \frac{\gamma_{n}( I_{n,k}( \omega  )(s))}{\sqrt{n}}  +  \eta_{n},~a.s.-\mathbf{P}
  \end{align}
   where $I_{n,k}(\omega)(s) \equiv \{ \theta \in  \Theta_{k}(M_{n,k}(\omega)) \mid s \geq \delta_{k,\mathbf{P}}(\theta, \nu_{k}(\mathbf{P})) \geq 0.5 s   \}  $ and \footnote{Observe that $\delta_{k,\mathbf{P}}(\theta,\nu_{k}(P)) \geq 0$ for all $\theta \in \Theta_{k}$.} 
  \begin{align*}
  \delta_{k,\mathbf{P}}(\theta,\nu_{k}(\mathbf{P})) \equiv \sqrt{Q_{k}(\theta,\mathbf{P}) - Q_{k}(\nu_{k}(\mathbf{P}),\mathbf{P})},~\forall \theta \in \Theta_{k},~k\in \mathbb{N}.
  \end{align*}
  That is, $\gamma_{n}( I_{n,k}( \omega  )(s))$ measures the complexity of an ``$\delta_{k,\mathbf{P}}$-strip'' of $\Theta_{k}(M_{n,k}(\omega))$. For any $(n,k)$, the ``variance term'' is given by  
   \begin{align}\label{eqn:V-H}
  	V_{n,k}(\omega) = \min \left\{ s > 0 \mid s \geq  5 \max_{x \geq 1}  \frac{H_{n,k}(\omega)(sx)}{sx}    \right\},~a.s.-\mathbf{P} .
  \end{align}

  \medskip
  
  \textbf{Concentration result.} Let $\boldsymbol{\varrho} = \{ \varrho_{n,k} : \Omega \rightarrow \mathbb{R}_{+}  \}_{n,k}$ such  that for any $(n,k)$
  \begin{align}
  	\varrho_{n,k}(\omega) = V_{n,k}(\omega)+ \sqrt{B_{k}(\mathbf{P})},~a.s.-\mathbf{P}.
  \end{align}
  This sequence is the concentration rate of our estimator.  
Let $\mathbb{G}_{0} =  3 \left( p_{\upsilon} \left( 8.1 \right) + \sqrt{2} \times 8 	 \right)$, and for any $u \geq \mathbb{G}_{0}$, let $g_{0}(u) = \mathbb{G}_{0} u^{-1} $, and $g_{0}(u) = 1$ for $ u \in [0,\mathbb{G}_{0})$.\footnote{The constant follows from the bounds that appear in the Propositions in the Appendix \ref{app:concen-main}.} We now establish a concentration property for our regularized M-estimator. We note that this result is obtained without assumptions, other than $\nu(\mathbf{P}) \ne \{ \emptyset \}$.
  
\begin{theorem}
	\label{thm:concen-main}
	Suppose $\nu(\mathbf{P}) \ne \{ \emptyset \}$. Then, the regularized M-estimator (defined in definition \ref{def:reg})  $(\boldsymbol{\varrho},g_{0})$-concentrates at $\nu(\mathbf{P})$ under $\delta_{\mathbf{P}}$. That is, for all $(n,k)$ and $u >0$
	\begin{align}\label{eqn:main-1}
	\mathbf{P} \left( \delta_{\mathbf{P}} \left( \nu_{k}(P_{n}), \nu(\mathbf{P}) \right) \geq u  \varrho_{n,k}(\omega)   \right)   \leq g_{0}(u).
	\end{align}
\end{theorem}

\begin{proof}
	See Appendix \ref{app:concen-main}.
\end{proof}

Before discussing its implications, we derive an upper bound on $||.||_{q_{n,k}}$ which introduces a key component of the variance term (and hence, of the concentration rate), the so-called \emph{effective number of observations}. 
It follows that for any $q \in \mathbb{N}$ and any $r>2$,\footnote{This is proven in Lemma \ref{lem:bound-normq} in the Online Appendix \ref{app:supp-lem-bound-H}.}
\begin{align}\label{eqn:qnorm-bound}
	||\cdot ||_{q} \leq \sqrt{2} \left( \int_{0}^{1} |\mu_{q}(u)|^{\frac{r}{r-2}}  du  \right)^{\frac{r-2}{2r}} ||\cdot ||_{L^{r}(\mathbf{P})}.
\end{align}

  Since $||.||_{L^{r}(\mathbf{P})}$ does not depend on the mixing structure, in order to understand how the mixing structure affects concentration rates, it suffices to study the sequence $\left\{ \left( \int_{0}^{1} |\mu_{q_{n,k}}(u)|^{\frac{r}{r-2}}    du\right)^{\frac{r-2}{2r}} \right\}_{n,k}$. Using this observation, the next Theorem provides bounds for the \textquotedblleft variance term" of the concentration rate. 

\begin{theorem}
	\label{thm:effe-n}
	For any $ r > 2$ and any $n,k$, 
		\begin{align*}
		V_{n,k}(\omega) \leq \min \left\{ s > 0 \mid s \geq  5 \max_{x \geq 1}  \frac{\frac{\gamma(I_{n,k}(\omega)(sx),2^{1/r} ||.||_{L^{r}(\mathbf{P})})}{\sqrt{n(\beta)}} + \eta_{n}}{sx}   \right\}
		\end{align*}
	where \begin{align}
		n(\beta) =  \frac{n}{2^{1-2/r} \left(  \int_{0}^{1} |\mu_{q_{n,0}}(u)|^{\frac{r}{r-2}}    du\right)^{\frac{r-2}{r}}}.
	\end{align}	
	
\end{theorem}

\begin{proof}
	See Appendix \ref{app:effe-n}.
\end{proof}

We call $n(\beta)$ the \emph{effective number of observations}. The Theorem indicates that is the effective number of observations --- rather than $n$ --- the right measure of the sample size in the variance term of the concentration rate. It also illustrates how the dependence structure affects the concentration rate: By scaling the sample size through the term $\left(2\int_{0}^{1} |\mu_{q_{n,0}}(u)|^{\frac{r}{r-2}}    du \right)^{\frac{r-2}{2}} $.
Moreover, by inspection of $\int_{0}^{1} |\mu_{q_{n,0}}(u)|^{\frac{r}{r-2}}    du$, we can see that the integrability of $\mu_{q_{n,0}}$ --- which in turn relates to the one of $\beta^{-1}$; see equation \ref{eqn:muq} --- plays an important role; we relegate a more thorough discussion of this and the effective number of observations to section \ref{sec:effe-n}.

	\medskip

	\textbf{Concentration Property under general metric.} The next proposition establishes the concentration property for a general metric $\varpi$; to do so, it is paramount to quantify the relationship between $\delta_{k,\mathbf{P}}$ and the desired metric $\varpi$. 
	\begin{assumption}\label{ass:IU}
		For any $\epsilon>0$, $k \in \mathbb{N}$ and $M>0$, 
		\begin{align}\label{eqn:IU}
		\inf_{\theta \in \Theta_{k}(M) \setminus \nu_{k}(\mathbf{P})^{\epsilon} } \frac{\delta_{k,\mathbf{P}} (\theta, \nu_{k}(\mathbf{P})) }{\varpi(\theta , \nu_{k}(\mathbf{P}) )} \geq \underline{\varpi}_{k}(M,\epsilon).
		\end{align}
		where $\nu_{k}(\mathbf{P})^{\epsilon} = \{ \theta \in \Theta_{k}(M) \mid \inf_{\theta_{0} \in \nu_{k}(\mathbf{P})} \varpi(\theta,\theta_{0}) < \epsilon    \}$, and  $\underline{\varpi}_{k}: \mathbb{R}^{2}_{+} \rightarrow \mathbb{R}_{++}$.
	\end{assumption}
	
	Let $\boldsymbol{\tilde{\varrho}} = (\tilde{\varrho}_{n,k})_{n,k}$, where, for any $(n,k)$, $\tilde{\varrho}_{n,k}(\omega)   = \tilde{V}_{n,k}(\omega)  + \varpi(\nu_{k}(\mathbf{P}),\nu(\mathbf{P}))$, with 
	\begin{align}\label{eqn:V-concen-w}
	\tilde{V}_{n,k}(\omega) = \min \left\{ s > 0 \mid s \geq 5 \max_{x \geq 1} \frac{\left\{ \gamma_{n}( \tilde{I}_{n,k}(\omega)(s x) )/\sqrt{n(\beta)}  + \eta_{n}      \right\} }{x s ( \underline{\varpi}_{k}(M_{n,k}(\omega), 0.5 x s) )^{2} }      \right\}
	\end{align}
	and $ \tilde{I}_{n,k}(\omega)(s)  = \{ \theta \in \Theta_{k}(M_{n,k}(\omega)) \mid s \geq \varpi(\theta, \nu_{k}(\mathbf{P}) ) \geq 0.5 s   \} $.

	\begin{proposition}\label{pro:concen-w}
		Suppose  $\nu(\mathbf{P}) \ne \{ \emptyset\}$. Let $\varpi$ be a metric over $\Theta$ such that Assumption \ref{ass:IU} holds. Then, the regularized M-estimator (defined in definition \ref{def:reg})  $(\boldsymbol{\tilde{\varrho}},g_{0})$-concentrates at $\nu(\mathbf{P})$ under $\varpi$.	
	\end{proposition}
	
	\begin{proof}
		See Appendix \ref{app:concen-w}.
	\end{proof}
	
	Condition 	\ref{eqn:IU} is akin to the identifiable uniqueness condition (e.g. see \cite{WW1991}) and to measures of ill-posedness in the context of ill-posed inverse problems (see \cite{CP-2012}).
	The condition quantifies how well the regularized criterion separates points (away from $\nu(\mathbf{P})$) in the metric space $(\Theta,\varpi)$.
	 The main difference with the results in the Theorem \ref{thm:concen-main} is the scaling by $\underline{\varpi}_{k}$ in the variance term. Ideally $\underline{\varpi}_{k}(M,\epsilon) \geq c >0$ for all $k,M,\epsilon$, so in this case $\tilde{V}_{n,k}$ is proportional to  $V_{n,k}$ for any $(n,k)$. There could be cases, however, where $\limsup_{k\rightarrow \infty} \underline{\varpi}_{k}(M,\epsilon) = 0$, thus implying that $\tilde{V}_{n,k}/V_{n,k}$ will diverge. 
	
	\medskip
	
		\textbf{$L^{1}$ non-asymptotic bound.} Theorem \ref{thm:concen-main} implies, under additional integrability restrictions, an $L^{1}$ non-asymptotic bound for our estimator.
		
		\begin{proposition}
			\label{pro:L1-rate}
			Suppose that $\nu(\mathbf{P}) \ne \{ \emptyset\}$, and also that for any $(n,k)$, there exists a decreasing bounded function $\varphi_{n,k} : \mathbb{R}_{+} \rightarrow \mathbb{R}_{+}$, such that  
			\begin{align*}
			E_{\mathbf{P}} \left[ \frac{\delta_{\mathbf{P}}(\nu_{k}(P_{n}),\nu(\mathbf{P}))}{\varrho_{k,n}(\omega)} 1\{ \frac{\delta_{\mathbf{P}}(\nu_{k}(P_{n}),\nu(\mathbf{P}))}{\varrho_{k,n}(\omega)} \geq A  \}    \right] \leq \varphi_{n,k}(A),~\forall A>0.
			\end{align*}
			Then $	E_{\mathbf{P}} \left[ \frac{\delta_{\mathbf{P}}(\nu_{k}(P_{n}),\nu(\mathbf{P}))}{\varrho_{k,n}(\omega)}     \right] \leq \inf_{A \geq 1} \{\varphi_{n,k}(A) + 1 + \mathbb{G}_{0} \ln (A)\}$
			for any $n,k$. 
		\end{proposition}
		
		\begin{proof}
			See Appendix \ref{app:concen-w}.
		\end{proof}
		
		A function $\varphi_{n,k}$ always exist if $\frac{\delta_{\mathbf{P}}(\nu_{k}(P_{n}),\nu(\mathbf{P}))}{\varrho_{k,n}(\omega)}$ is in $L^{1}(\mathbf{P})$. So we view the assumption in the proposition as a way of quantifying the tail behavior of $\frac{\delta_{\mathbf{P}}(\nu_{k}(P_{n}),\nu(\mathbf{P}))}{\varrho_{k,n}(\cdot)}$. For instance, if $\sup_{\theta \in \Theta_{k} } \delta_{\mathbf{P}}(\theta,\nu(\mathbf{P})) \leq  D_{k} < \infty$ and $\boldsymbol{\varrho}$ is non-random, we choose $\varphi_{n,k}(A) = 1\{ A \leq D_{k}/\varrho_{n,k}  \}\frac{D_{k} \log (1+D_{k})}{\varrho_{n,k}\log (1 + A \varrho_{n,k})}  $ and obtain $E_{\mathbf{P}} \left[\delta_{\mathbf{P}}(\nu_{k}(P_{n}),\nu(\mathbf{P}))     \right] \leq \varrho_{k,n} (1+\mathbb{G}_{0} \log (2D_{k}) + \mathbb{G}_{0} \log (1/\varrho_{k,n}))$ for any $(n,k)$.

\subsection{Discussion about Theorem \ref{thm:concen-main} and Theorem \ref{thm:effe-n} }
\label{sec:discussion}


\subsubsection{Heuristics}
\label{sec:heur}

 Informally, the proof of Theorem \ref{thm:concen-main} can be divided into two main parts. The first part relies on ``Wald's approach" (\cite{wald1949}) and is fairly standard in this setting (e.g., see \cite{CS-1998}). It hinges on first noting that $\delta_{\mathbf{P}} \left( \nu_{k}(P_{n}), \nu(\mathbf{P}) \right)  \leq  \delta_{k,\mathbf{P}} \left( \nu_{k}(P_{n}), \nu_{k}(\mathbf{P}) \right) + \sqrt{B_{k}(\mathbf{P})}$, so it is sufficient to show that $\{\delta_{k,\mathbf{P}} \left( \nu_{k}(P_{n}), \nu_{k}(\mathbf{P}) \right) \geq u V_{n,k}(\omega)\}$ has probability lower than $g_{0}(u)$. In order to do this we \textquotedblleft slice" this set into strips of the form $I_{n,k,l} \equiv \{ 2^{l} u V_{n,k}(\omega) \geq \delta_{k,\mathbf{P}} \left( \nu_{k}(P_{n}), \nu_{k}(\mathbf{P}) \right) \geq 2^{l-1} u V_{n,k}(\omega)\}$ for $l=1,2,...$. After some tedious calculations it follows that it suffices to control \emph{uniformly} the process $\theta \mapsto \mathcal{L}_{n}(\theta) = n^{-1}\sum_{i=1}^{n} \{\phi(Z_{i},\theta) - E_{\mathbf{P}}[\phi(Z,\theta))]\}$ over $I_{n,k,l}$. The second part of the proof essentially consists of showing that the set
 \begin{align*}
 	\left\{ \omega \colon  \sup_{\theta \in I_{n,k,l}} |\mathcal{L}_{n}(\theta) - \mathcal{L}_{n}(\nu_{k}(\mathbf{P})) | \geq 2^{l} u V_{n,k}(\omega)  \right\}
 \end{align*} occurs with probability less than $2^{-l}g_{0}(u)$, and it is shown here:
 
 \begin{theorem}
 	\label{thm:gbrack}
 	Let $A \subseteq \Theta$. Then, for all $(n,k)$ and all $u \geq 0$,
 	\begin{align}
 	\mathbf{P} \left(  \sup_{\theta \in A} |\mathcal{L}_{n}(\theta) - \mathcal{L}_{n}(\nu_{k}(\mathbf{P})) | \geq u \frac{\gamma_{n}(A)}{\sqrt{n}}      \right)  \leq g_{0}(u).
 	\end{align}
 \end{theorem}
 
\begin{proof}
	See Appendix \ref{app:gbrack}.
\end{proof}

By our definition of $\mathcal{F}$, the statement of the theorem can be thought directly in terms of $f \in \mathcal{F}(A)$, i.e., $\mathbf{P} \left(  \sup_{f \in \mathcal{F}(A)} |n^{-1} \sum_{i=1}^{n} f(Z_{i}) - E[f(Z)] | \geq u \frac{\gamma_{n}(A)}{\sqrt{n}}      \right)  \leq g_{0}(u)$. The proof of this statement relies on two main insights. First, by using a chaining argument akin to that in the proof of the CLT for empirical processes based on bracketing (see \cite{VdV-W1996} Ch. 2.5), we decompose any $f \in \mathcal{F}(A)$ (and consequently $n^{-1} \sum_{i=1}^{n} f(Z_{i}) - E[f(Z)]$) into several parts (see Lemma \ref{lem:decom-L} in the Appendix \ref{app:gbrack}); essentially, we decompose $f$ into ``bounded" parts and ``unbounded" parts. Second, we use Bernstein inequality for $\beta$-mixing (see Lemma \ref{lem:bere} in the Appendix \ref{app:gbrack}) and the ideas in \cite{talagrand2014} to control the ``bounded" parts \emph{uniformly}, and we use a $L^{1}$ bound for the ``unbounded" parts (see Lemma \ref{lem:L1-bdd} in the Appendix \ref{app:gbrack}).


\subsubsection{On the family of Norms $(||.||_{q_{n,k}})_{k}$.}
\label{sec:family}

 Our measure of complexity, $\gamma_{n}$, suggests that the appropriate notion of distance to measure the complexity of $\Theta$ is given by a family of norms --- as opposed to only one norm --- with each norm depending on the mixing structure. We now discuss these two features.
 
 	It follows from Lemma \ref{lem:q-norm}(3) in Appendix \ref{app:gbrack} that for all $q \in \mathbb{N}$
	\begin{align*}
		||f||_{q} \leq ||f||_{2,\beta} \equiv \sqrt{2} \int_{0}^{1} \beta^{-1}(2u) Q^{2}_{f}(u) du,~\forall f \in \mathcal{F}(\Theta).
	\end{align*} 
	(recall that $Q_{f}$ is the quantile function of $|f|$, see \ref{eqn:muq}). In the case where $\int_{0}^{1} \beta^{-1}(2u)du $ is finite, the RHS of the previous display provides a well-defined norm for $\mathcal{F}(\Theta)$ which was first proposed by \cite{DMR1995} (henceforth, DMR) for constructing bracketing entropies. Moreover, it is easy to see that $\gamma_{n}(A) \leq \inf_{(\mathcal{T}_{l})_{l} \in \mathbf{T}} \sup_{f \in \mathcal{F}(A)} \sqrt{2} \sum_{l=0}^{\infty} 2^{l/2} ||S(f,\mathcal{T}_{l})||_{2,\beta}$. This indicates that when $\int_{0}^{1} \beta^{-1}(2u)du $ is finite, we can use the norm proposed in DMR to construct our measure of complexity for the \textquotedblleft variance term". 
	If $\int_{0}^{1} \beta^{-1}(2u)du$ is not finite, however, the above proposal becomes infeasible since $||.||_{2,\beta}$ may not even be well-defined, and thus cannot be used to construct the measure of complexity; we need an alternative way of measuring distance. Since $||.||_{q}$ is always well-defined for any $q$, we rely on a \emph{family} of norms to describe the complexity measure.\footnote{ The fact that $||.||_{q}$ is always well-defined for any $q$ follows from Lemma \ref{lem:q-norm} in Appendix \ref{app:gbrack}, which shows that $||.||_{q}$ is bounded by $\int_{0}^{1} \min\{ \beta^{-1}(2u) ,q \} Q^{2}_{f}(u) du$} 
%
 
 \subsubsection{On the Effective Number of Observations.} 
 \label{sec:effe-n}

 In order to shed some light on the behavior of $n(\beta)$, we provide bounds for two widely use canonical mixing structures.

\begin{proposition}
	\label{pro:bound-H}
Let $m_{0} > 0$. Then, for any $n$, $q \in \mathcal{Q}_{n}$ and $r>2$,\footnote{The element $[x]$ is the smallest upper bound for $x$ in $\mathcal{Q}_{n}$.} 
		\begin{enumerate}
			\item If $\beta(q) = 1\{ q < m^{-1}_{0} \}$, then
			 \begin{align*}
			\frac{n}{  \min\{ 1 + [n/4]  , 1 + m_{0}^{-1} \}    } \leq	n(\beta) \leq \frac{n}{  \min\{ 1 + n/4  , m_{0}^{-1} \}    }.
			\end{align*}
		\end{enumerate}			
 And, if $\beta(q) = (1+q)^{-m_{0}}$, 
\begin{enumerate} 
			\item[2.] With $m_{0} > \frac{r}{r-2}$, then\begin{align*}
			 \frac{n}{\left(\frac{m_{0}(r-2)}{m_{0}(r-2) - r} \right)^{(r-2)/r}} \leq 	n(\beta) \leq   \frac{n}{ \max \left\{ \left(2^{-m_{0}} \right)^{(r-2)/r} , 1 \right\}}.
			\end{align*} 
				\item[3.] With $m_{0} = \frac{r}{r-2}$, then\begin{align*}
					\frac{n}{ \left( m_{0} \log \left( 1 +  \left[ \left( \frac{n}{4} \right)^{\frac{1}{m_{0}+1}} \right] \right) + 1   \right)^{\frac{1}{m_{0}}}} \leq n(\beta) \leq \frac{n}{\frac{1}{2} \left( m_{0} \log \left( \frac{1}{2} \left( 1 +  \left( \frac{n}{4} \right)^{\frac{1}{m_{0}+1}} \right) \right) + 2^{m_{0}}   \right)^{\frac{1}{m_{0}}}}.
				\end{align*} 
				\item[4.] With $m_{0} < \frac{r}{r-2}$, then \begin{align*}
					 \frac{n}{   \left( 1 +  \left[ \left(  \frac{n}{4}  \right)^{\frac{1}{m_{0}+1}} \right]  \right)^{\frac{r - (r-2)m_{0}}{r}}    A_{0}  } \leq n(\beta) \leq \frac{n}{  \left( \left( 1 +  \left(  \frac{n}{4}  \right)^{\frac{1}{m_{0}+1}} \right)^{\frac{r- (r-2)m_{0}}{r-2}} - 2^{\frac{r}{r-2} - m_{0}}   \right)^{\frac{r-2}{r}} A_{1}  }
				\end{align*}
				where $A_{1} = \left( \frac{1}{2} \right)^{\frac{r-(r-2)m_{0}}{r}} \left( \frac{2^{-m_{0}} m_{0}(r-2)}{r - m_{0}(r-2)} \right)^{\frac{r-2}{r}}$ and $A_{0} = \left( \frac{r}{r-m_{0}(r-2)}  \right)^{\frac{r-2}{r}}$.
		\end{enumerate}
\end{proposition}

\begin{proof}
	See Appendix \ref{app:sec-main-prop}. 
\end{proof}

\begin{remark}
	Our result for the $m^{-1}_{0}$-dependent case (Case 1) formalizes the intuition that if observations are $m^{-1}_{0}$-dependent (and we do not have any additional information regarding their behavior), then is as if we only had (up to constants) $n/m^{-1}_{0}$ observations for computing the estimator. Asymptotically, for $m_{0}$ fixed, our result implies that $n(\beta) \asymp n$ and thus the concentration rate for this case is \emph{asymptotically} the same as the one for the i.i.d. case. Our results, however, provide a more nuanced view by studying the finite sample behavior.\footnote{For comparison, for the i.i.d. case ($m_{0}=1$) our results show that $0.5 n  \leq  n(\beta) \leq n$ for any $n$. I.e., the lower bound is ``loose" by a factor of 2.}  For instance, if $m_{0}^{-1}$ is comparable to the number of observations, the effective number of observations can be small, ultimately yielding a larger value for the concentration rate. $\triangle$
\end{remark}

 \begin{remark}
 	Case 2 is analogous to Case 1 and has been widely used in the literature (e.g. DMR, \cite{Hansen96}, \cite{CS-1998} and \cite{Rio2013} among others) and many stationary processes have been shown to be $\beta$-mixing with decay faster than Case 2.\footnote{E.g. \cite{Chen2013} for a general review, and \cite{Beare2010} and \cite{CWY2009} for results for Markov Copula models.} Cases 3 and 4 are different, because, even asymptotically the effective number of observations differs from the actual number of observations. Roughly, for case 3 $n(\beta) \asymp n/\log(n)$ whereas for case 4 $n(\beta) \asymp n^{1-\frac{r-(r-2)m_{0}}{r(m_{0}+1)}}$. Perhaps surprisingly, even in this last case $n(\beta) \rightarrow \infty$ as $n \rightarrow \infty$, for any pair $(r,m_{0})$, but it can be at a very slow rate (e.g., the case $m_{0} \approx 0$). Cases 3 and 4, although not as widely used as Case 2, could be of interest since they allow for slowly decaying dependence structure and thus could be used as alternatives for modeling long-range dependency or long-memory.\footnote{See \cite{CCLJOE10} for examples of processes with slow polynomial decay in the $\beta$-mixing coefficients.}  $\triangle$ 
 \end{remark}

%

\begin{remark}
	The key difference between cases 1-2 and 3-4 is that for the former two $\int_{0}^{1} (\beta^{-1}(u))^{\frac{r}{r-2}} du $ is finite, whereas for the latter two it is not. Intuitively, if $\int_{0}^{1} (\beta^{-1}(u))^{\frac{r}{r-2}} du $ is finite, then it turns out that $\mu_{q}$ in display \ref{eqn:qnorm-bound} is bounded above by a constant (not depending on $q$) and thus it suffices to study the complexity measure under a single norm, $||.||_{L^{r}(\mathbf{P})}$. Moreover, in this case, the mixing structure has no (asymptotic) incidence on the convergence rate; this observation is consistent with the asymptotic results in \cite{CS-1998} and \cite{ChenLiao2013}. $\triangle$   
\end{remark}

\subsection{Upper Bounds for our Measure of Complexity} 
\label{sec:MoC-bound}

Given the result in Theorem \ref{thm:effe-n} a key object to bound the variance term is $\gamma(\cdot,||.||_{L^{r}(\mathbf{P})})$. We now provide bounds in terms of the more standard metric entropy-based measure of complexity (\cite{Dudley-1967}). We also link our measure of complexity to the Generic Chaining one (\cite{talagrand2014} and references therein), thereby providing ``easy" to compute bounds based on maximal inequalities for Gaussian processes. 
\label{sec:Dudley-bound}

We impose the following Lipschitz restriction on $\phi$.

\begin{assumption}\label{ass:Lip-phi}
	There exists a pseudo-distance $\mathbf{d} : \Theta^{2} \rightarrow \mathbb{R}_{+}$ such that for any $M \geq 0$ and $k \in \mathbb{N}$, there exists a $\mathbb{C}_{k,M} : \mathbb{Z} \rightarrow \mathbb{R}$ such that:
	
	(1) For all $\theta_{1}, \theta_{2} \in \Theta_{k}(M)$, $	|\phi(z,\theta_{1}) - \phi(z,\theta_{2})| \leq  \mathbb{C}_{k,M}(z) \mathbf{d}(\theta_{1},\theta_{2}),~a.s.-\mathbf{P}$.
	
	(2) There exists a $r>2$ such that for all $k \in \mathbb{N}$ and $M>0$, $||\mathbb{C}_{k,M}||_{L^{r}(\mathbf{P})} \in [0,\infty)$.
\end{assumption}

%
%

For instance, this assumptions is fulfilled in the HD-QR model (example \ref{exa:HD-QR}) with $r=\pi_{0}$, $\mathbf{d}=||.||_{\ell^{2}}$ and $\mathbb{C}_{k,M}(z) = (1+\tau)e_{max}(xx^{T})$.  The next proposition establishes an upper bound for $\gamma$.

  \begin{proposition}
  	\label{pro:bound-GC}
  	Suppose assumption \ref{ass:Lip-phi} holds. Then: For any $M\geq 0$, any $k \in \mathbb{N}$ and any $A \subseteq \Theta_{k}(M)$, 
  	\begin{align*}
  	\gamma(A,||.||_{L^{r}(\mathbf{P})}) \leq \sqrt{2}  ||\mathbb{C}_{k,M}||_{L^{r}(\mathbf{P})} \inf_{(\mathcal{S}_{k})_{k} \in \mathbf{S}} \sup_{\theta \in A}\sum_{l=0}^{\infty} 2^{l/2} Diam (T(\theta,\mathcal{S}_{l}),\mathbf{d})
  	\end{align*}
  	where $\mathbf{S}$ is the set of admissible sequences over $A$.\footnote{For any set $X$ and pseudo-distance $d$, $Diam(X,d) = \sup_{x_{1},x_{2} \in X} d(x_{1},x_{2})$.} 
  \end{proposition}
  
\begin{proof}
	See Appendix \ref{app:Dudley-bound}.
\end{proof}
  
  The expression $ \inf_{(\mathcal{S}_{k})_{k} \in \mathbf{S}} \sup_{\theta \in A}\sum_{k=0}^{\infty} 2^{k/2} Diam (T(\theta,\mathcal{S}_{k}),\mathbf{d})$ is exactly Talagrand's Generic Chaining bound 
  (\cite{talagrand2014}), 
  applied to $A$ under $\mathbf{d}$. By the calculations in \cite{talagrand2014} pp 21-24, this expression is bounded above by  $\int_{0}^{\infty} \sqrt{\log N(e,A,\mathbf{d})}de $; thus showing that our measure of complexity is sharper than Dudley's entropy bound. In fact, as argued in \cite{talagrand2014} Sec. 2.3, for certain sets of $\mathbb{R}^{k}$ the difference can even diverge with $k$.  
  
  The case where $\mathbf{d}$ is induced by $||.||_{\ell^{2}(b)}$ deserves an special mention due to the work by Talagrand (e.g. \cite{talagrand2014} and references therein). 
  Let, $\zeta_{j} \sim N(0,1)$ for $j=1,...,k$, and for any $k \in \mathbb{N}$ and $A \subseteq \Theta_{k}$
	\begin{align}\label{eqn:Gauss}
	\Gamma_{k}(A) \equiv E\left[\sup_{\theta \in A} \sum_{j=1}^{k} \zeta_{j} \sqrt{b_{j}} \theta_{j}   \right].
	\end{align}

	The next Proposition is a direct corollary of Proposition \ref{pro:bound-GC} and \cite{talagrand2014} Theorem 2.4.1.

	\begin{proposition}
		\label{pro:M-rate1}
		Suppose assumption \ref{ass:Lip-phi} holds with $\mathbf{d}$ induced by $||.||_{\ell^{2}(b)}$. Then: There exists a $L\geq 0$ such that, for any $M\geq 0$, any $A \subseteq \Theta_{k}(M)$ and any $k \in \mathbb{N}$, 
		\begin{align}
		\gamma(A,||.||_{L^{r}(\mathbf{P})}) \leq \sqrt{2} L ||\mathbb{C}_{k,M}||_{L^{r}(\mathbf{P})}  \Gamma_{k}(A) 
		\end{align}
	\end{proposition} 
		
	\begin{proof}
		See Appendix \ref{app:Dudley-bound}.\footnote{The $L$ comes from \cite{talagrand2014} Theorem 2.4.1, and is universal; it does not depend on $n$, $k$ nor on $\mathbf{P}$. }
	\end{proof}
	
	The proposition establishes an upper bound for our measure of complexity of $A$ (in particular $A = I_{n,k}(\omega)(s)$ or $A = \tilde{I}_{n,k}(\omega)(s)$ defined in Proposition \ref{pro:concen-w}) in terms of the expectation of the supremum of a Gaussian process. This quantity, $\Gamma_{k}(A)$, is a fairly simple object and relatively easy to bound. A general strategy to bound $\Gamma_{k}(A)$ consists of using H\"older inequality to obtain $\Gamma_{k}(A) \leq E[||\zeta||_{\ell^{q_{1}}}] \sup_{\theta \in A} ||b \cdot \theta||_{\ell^{q_{2}}}$ with $1/q_{1} + 1/q_{2} = 1$. The choice of $q_{2}$ is dictated by the geometric properties of $A$ under $|| b \cdot .||_{\ell^{q_{2}}}$, and bounding $E[||\zeta||_{\ell^{q_{1}}}]$ usually involves elementary operations since $\zeta$ are independent Gaussian. In the Online Appendix \ref{app:sup-Gau} we formalize this approach and provide bounds when the penalty $Pen$ is constructed using $\ell^{q}$-norms. These observations illustrate what we consider an additional advantage of our measure of complexity over the more standard ones. 
 
%
 
\subsection{Linear Regression Model (Example \ref{exa:OLS-q}) cont.}
\label{sec:OLS-simple}

It is instructive to compare the results we obtained in the simple linear regression model in Example \ref{exa:OLS-q}, with those obtained by applying our general method, which cannot exploit the explicit solution of the OLS estimator. 

Since $\phi(z,\theta) = n^{-1} \sum_{i=1}^{n} (y - x_{i}^{T} \theta)^{2}$ (recall that $(x_{i})_{i=1}^{n}$ are fixed, not random) and $n^{-1} \sum_{i=1}^{n} x_{i} x_{i}^{T} = I$, it can be shown that Assumption  \ref{ass:Lip-phi} holds with $\mathbb{C}_{k,M}(z) = 2(|y| + K_{2})$ , $r=2$  and $\mathbf{d} = ||.||_{\ell^{2}}$,    
So, by Proposition \ref{pro:M-rate1}, $\gamma(A,||.||_{L^{2}}) \leq \mathbb{K}_{1} E \left[ \sup_{\theta \in A} \sum_{k=1}^{d} \zeta_{j} \theta_{j}    \right] $, for any $A \subseteq \Theta$ with $\mathbb{K}_{1} \equiv \sqrt{2} L 2(E[|Y|^{2}] + K_{2})$. In particular, for $A=\{ \theta \in \Theta \mid  ||\theta - \theta_{\ast} ||_{\ell^{2}} \leq s   \}$ for any $s>0$, by Cauchy-Schwartz inequality, it follows that $\gamma(A,||.||_{L^{2}}) \leq \mathbb{K}_{1}s E \left[ ||\zeta ||_{\ell^{2}} \right] =  \mathbb{K}_{1} s  \sqrt{d}$. Also, by Proposition \ref{pro:bound-H}(1) (with $m^{-1}_{0} = \mu_{0}$ ), $n(\beta) \geq 2 \frac{n}{\mu_{0}}$.\footnote{Formally, this data structure is not stationary, but still sufficiently well-behaved to apply our theorems; in particular, $q \mapsto \beta(q) = 1\{ q \leq \mu_{0}  \}$ according to our definition in Section \ref{sec:prem}.} Therefore, by expressions \ref{eqn:H} and \ref{eqn:V-H}, it follows that $V_{n,k} \leq  2 \mathbb{K}_{1} \sqrt{\frac{d}{n/\mu_{0}}}  $, and thus Theorem \ref{thm:concen-main} implies that
\begin{align*}
	\mathbf{P} \left( ||\nu(P_{n}) - \theta_{\ast}||_{\ell^{2}} \geq u \sqrt{\frac{d}{n/\mu_{0}}}    \right) \leq 
	\mathbb{G}_{0} 2\mathbb{K}_{1} u^{-1}.
\end{align*}
The only difference between this expression and \ref{eqn:toy-concen} is the constant in the concentration bound.

   	\section{Choice of Regularization Parameters} 
   	\label{sec:choice}

   	We apply our concentration results to construct a data-driven method for choosing the tuning parameter. This method is an adaptation to regularized M-estimation of the one proposed in \cite{PereverzevSchock2006}. The salient feature of this method is that it does not use any knowledge of the \textquotedblleft bias" term; it is solely based on the \textquotedblleft variance term".
    
    Unfortunately, Theorem \ref{thm:concen-main} cannot be used to establish concentration results for the aforementioned method, since, $\delta_{\mathbf{P}}$  depends on $\mathbf{P}$ (which is unknown).\footnote{$V_{n,k}$ depends on $\mathbf{P}$ through $\delta_{k,\mathbf{P}}(.,\nu_{k}(\mathbf{P}))$.} Hence, we rely on Proposition \ref{pro:concen-w} which establishes concentration results for a metric $\varpi$ (e.g., a Banach norm over $\Theta$) and $\tilde{V}_{n,k}$ (see expression \ref{eqn:V-concen-w}). For reasons that will become apparent in the proof of Theorem \ref{thm:k-fea}, we require  $\tilde{V}_{n,k}$ to be increasing as a function of $k$, or at least find an upper bound that it is. Abusing notation, let $\sqrt{B_{k}(\mathbf{P})} = \varpi (\nu_{k}(\mathbf{P}),\nu(\mathbf{P}))$.
	\begin{assumption}
		\label{ass:V-bdd}
		(i) For each $(k,n) \in \mathbb{N}^{2}$ let $\tilde{V}_{k}(P_{n})$ be such that $\tilde{V}_{k}(P_{n}) \geq \mathbb{V} \tilde{V}_{n,k}(\omega)$, $k \mapsto \tilde{V}_{k}(P_{n})$ non-decreasing and $\mathbb{V} \geq 1$; (ii) $k \mapsto B_{k}(\mathbf{P})$ is decreasing.
	\end{assumption}
   	Part (i) is a high level assumption which seems easy to obtain; we postpone its discussion to the Appendix \ref{app:choice}. Part (ii) ensures monotonicity of the  ``bias" term, which is convenient for the proof and can be somewhat relaxed. 
   	

   	Let $\mathcal{K} = \{ k_{i} \in \mathbb{N} \mid 0 < k_{0} < ... < k_{|\mathcal{K}|}   \}$ for some $|\mathcal{K}| < \infty$, be the set from which the researcher selects the tuning parameter. This set is allowed to change with $n$. Let
   	\begin{align}
   	\mathcal{I}(P_{n},\mathbf{P}) \equiv \{ k \in \mathcal{K} \mid  \tilde{V}_{k}(P_{n}) \geq \sqrt{B_{k}(\mathbf{P})}  \} .
   	\end{align}
   	We assume $\mathcal{K}$ to be such that $\mathbf{P}(\mathcal{I}(P_{n},\mathbf{P}) \ne \{ \emptyset \}) = 1$; this assumption is quite mild since $k \mapsto \tilde{V}_{k}(P_{n})$ is nondecreasing and $ \limsup_{k \rightarrow \infty} B_{k}(\mathbf{P}) = 0$. Finally, we define the \emph{ideal tuning parameter} as \footnote{If there are several minimizer; we take the largest one.}
   	\begin{align*}
   	k^{I}(P_{n},\mathbf{P}) = \arg \min_{k \in \mathcal{I}(P_{n},\mathbf{P})} \tilde{V}_{k}(P_{n}).
   	\end{align*}

   	The next lemma establishes that the (infeasible) estimator $(\nu_{k^{I}(P_{n},\mathbf{P})}(P_{n}) )_{n}$ satisfies the concentration property at $\nu(\mathbf{P})$ under $\varpi$, with rate $(\tilde{V}_{k^{I}(P_{n},\mathbf{P})}(P_{n}) )_{n}$ and bound $u \mapsto g_{0}(0.5u/\mathbb{V})$.	
   	\begin{lemma}
   		\label{lem:k-inf}
   		Suppose Assumptions \ref{ass:IU} and \ref{ass:V-bdd} hold and $\nu(\mathbf{P}) \ne \{\emptyset\}$. Then,
   		\begin{align*}
   		\mathbf{P}( \varpi (\nu_{k^{I}(P_{n},\mathbf{P})}(P_{n}) , \nu(\mathbf{P})  )  \geq 2 \mathbb{V} u \tilde{V}_{k^{I}(P_{n},\mathbf{P})}(P_{n})   ) \leq g_{0}(u),~\forall n~and~u>0.
   		\end{align*} 
   	\end{lemma}

   	\begin{proof}
   		See Appendix \ref{app:choice}.
   	\end{proof}
   	   	
We now construct the \textquotedblleft feasible" tuning parameter. For any $s \in \mathbb{R}_+$
   	\begin{align*}
   	\mathcal{F}_{s}(P_{n}) \equiv \{ k \in \mathcal{K}  \mid    \varpi (\nu_{k}(P_{n}) ,  \nu_{k'}(P_{n})) \leq 4 s \tilde{V}_{k'}(P_{n}) ,~\forall k' \in \mathcal{K}~such~that~k' \geq k     \},
   	\end{align*}
   	be the \emph{test set}. For any $s>0$, the \emph{feasible tuning parameter} is given by
   	\begin{align}
   	k_{s}^{F}(P_{n}) = \arg \min_{k \in \mathcal{F}_{s}(P_{n})} \tilde{V}_{k}(P_{n}).
   	\end{align} 
   	The test set is random and known to the researcher since it depends only on known quantities. The motivation for its construction is as follows. By the triangle inequality, the fact that $k' \geq k$ and Proposition \ref{pro:concen-w}, it follows that, with probability higher than $1-g_{0}(s)$, $\varpi (\nu_{k}(P_{n}) ,  \nu_{k'}(P_{n})) \precsim s 2 \{ \tilde{V}_{k'}(P_{n}) +  \sqrt{B_{k}(\mathbf{P})} \} $. Hence, the ``extra" restriction imposed by the test set is to focus, roughly speaking, attention to $k \in \mathcal{K}$ for which the ``(squared) bias" term is dominated by the ``variance" term; i.e., $B_{k}(\mathbf{P}) \precsim \tilde{V}^{2}_{k}(P_{n}) \leq  \tilde{V}^{2}_{k'}(P_{n}) $. Since $\mathcal{I}(P_{n},\mathbf{P})$ imposes a similar restriction; one would expect that minimizing the variance over the test set would yield similar results to those given by the ideal tuning parameter. This is the idea behind the following theorem.

   	\begin{theorem}
   		\label{thm:k-fea}
   		Suppose Assumptions \ref{ass:IU} and \ref{ass:V-bdd} hold and $\nu(\mathbf{P}) \ne \{\emptyset\}$. Also, suppose that for some $s \in \mathbb{R}_{+}$, $	\mathbf{P} \left(  \mathcal{F}_{s}(P_{n}) \ne \{ \emptyset \}  \right) = 1$. Then, for all $n$
   		\begin{align*}
   		\mathbf{P}( \varpi (\nu_{k^{F}_{s}(P_{n})}(P_{n}) , \nu(\mathbf{P})  )  \geq  s 6 \mathbb{V} \tilde{V}_{k^{I}(P_{n},\mathbf{P})}(P_{n}) ) \leq 2 | \mathcal{K} | g_{0}(s) .
   		\end{align*} 
   	\end{theorem}

   	   	\begin{proof}
   	   		See Appendix \ref{app:choice}.
   	   	\end{proof}
  	   	
   	
   	   	The theorem essentially states that the concentration rate of the estimator $\nu_{k^{F}_{s}(P_{n})}(P_{n})$ is of the same order as the one corresponding to the ideal tuning parameter. That is, by minimizing the \textquotedblleft variance term" over the test set, one can construct an estimator which performance --- measured by the concentration rate --- is no worse (up to constants) than the one obtained by the ideal choice which uses knowledge of the ``bias" term to balance the ``bias" and ``variance" terms of the concentration rate.
   	
   	\medskip

   	   	\textbf{On the $|\mathcal{K}|$ factor.} The concentration function in the theorem is scaled by the complexity of the set $\mathcal{K}$, $|\mathcal{K}|$. Although it might not be surprising that the complexity of the set $\mathcal{K}$ affects the concentration bound, it still deserves some discussion, especially since it played no role in Lemma \ref{lem:k-inf}. 
   	   	
   	   	The reason for this scaling arises because we need to ensure that the set  $\cap_{k \in \mathcal{K}} E_{n}(u,k) \equiv \cap_{k \in \mathcal{K}} \left\{  \omega \in \Omega \mid \varpi (\nu_{k}(P_{n}) , \nu(\mathbf{P})  )  < u \left( \tilde{V}_{k}(P_{n}) + \sqrt{B_{k}(\mathbf{P})} \right) \right\}$ has high probability. Proposition \ref{pro:concen-w} states that for \emph{each $k$}, $\mathbf{P}(E_{n}(u,k)) \geq 1 - g_{0}(u)$, so by means of a crude union bound we obtain a lower bound $1-|\mathcal{K}| g_{0}(u)$ for $\mathbf{P}(\cap_{k \in \mathcal{K}} E_{n}(u,k))$; we refer the reader to the Online Appendix \ref{sec:discuss-V-bdd} for a formalization of this discussion.\footnote{In \cite{PereverzevSchock2006}  it is (implicitly) assumed that the sets $E_{n}(u,k)$ occur with probability one; thus in their results the complexity of $\mathcal{K}$ plays no role.}

   	   	\medskip
   	
   	\textbf{Implications of the Theorem.} The theorem provides ``$\varpi$-confidence-bands" in the sense that, for a given confidence level $\alpha \in (0,1)$, by choosing $s$ such that $2 | \mathcal{K} | g_{0}(s)  = \alpha$, the theorem implies that
   	\begin{align*}
   		\mathbf{P} \left( \varpi (\nu_{k^{F}_{s}(P_{n})}(P_{n}) , \nu(\mathbf{P})  ) \leq  s 6 \mathbb{V} \tilde{V}_{k^{I}(P_{n},\mathbf{P})}(P_{n})  \right) \geq 1- \alpha.
   	\end{align*}

   	As the next proposition shows, another implication of the Theorem is that our choice of tuning parameter yields a consistent estimator --- even allowing for  the complexity of $\mathcal{K}$ to grow with the sample size --- and moreover, its rate of convergence coincides with that of the ideal estimator.

%
%
%
   	
   	   	\begin{proposition}
   	   		\label{pro:choicek-asym}
   	   		Suppose all assumptions of Theorem \ref{thm:k-fea} hold. Then 
   	   		\begin{align*}
   	   			\frac{\varpi (\nu_{k_{s_{n}}^{F}(P_{n})}(P_{n}) , \nu(\mathbf{P}))}{\tilde{V}_{k^{I}(P_{n},\mathbf{P})}(P_{n})} = O_{\mathbf{P}}(s_{n}),~\forall~(s_{n})_{n}~such~that~|\mathcal{K}| g_{0}(s_{n}) \rightarrow 0.
   	   		\end{align*}

   	   	\end{proposition}

   	   	\begin{proof}
   	   		See Appendix \ref{app:choice}.
   	   	\end{proof}

%
%
So consistency of our estimator with the feasible tuning parameter is obtained provided that $s_{n} \tilde{V}_{k^{I}(P_{n},\mathbf{P})}(P_{n})=  o_{\mathbf{P}}(1)$.\footnote{For the case where $|\mathcal{K}|$ is uniformly bounded, it suffices to impose that  
 $\max_{k \in \mathcal{K}} \tilde{V}_{k}(P_{n}) = o_{\mathbf{P}}(1)$. This condition is rather mild and rules pathological cases where $k \mapsto \tilde{V}_{k}(P_{n})$ behaves like a \textquotedblleft traveling wave''. E.g. $\tilde{V}_{k}(P_{n}) = 1\{ k \geq K(n)  \}$ for a diverging $K(.)$. In this case, it could happen that $k^{I}(P_{n},\mathbf{P}) = K(n)$; hence $\tilde{V}_{k^{I}(P_{n},\mathbf{P})}(P_{n}) = 1$ and thus we cannot establish consistency. If the cardinality of $\mathcal{K}$ grows the condition states that is cannot grow too fast relative to $1/\max_{k} \tilde{V}_{k}(P_{n})$.}

   	\medskip

   	\textbf{Heuristics.} The proof is rather straightforward, albeit somewhat lengthy. The idea is to bound $\varpi (\nu_{k^{F}_{s}(P_{n})}(P_{n}) , \nu(\mathbf{P})  )  $ by bounding $\varpi (\nu_{k^{F}_{s}(P_{n})}(P_{n}) , \nu_{k^{I}(P_{n},\mathbf{P})}(P_{n})  )  $ and $\varpi (\nu_{k^{I}(P_{n},\mathbf{P})}(P_{n}) , \nu(\mathbf{P})  )  $. The latter term was bounded in Lemma \ref{lem:k-inf}, so it only remains to show that the former term is (up to constants) of the same order. To do this, the key step is that, with high probability, $k^{I}(P_{n},\mathbf{P})$ belongs to the test set. Thus, by construction of $k^{F}_{s}(P_{n})$, $\tilde{V}_{k^{F}_{s}(P_{n})}(P_{n}) \leq \tilde{V}_{k^{I}(P_{n},\mathbf{P})}(P_{n})$. Since $k \mapsto \tilde{V}_{k}(P_{n}) $ is non-decreasing, it must hold that $k^{I}(P_{n},\mathbf{P}) \geq k^{F}_{s}(P_{n})$. However, by construction of $k^{F}_{s}(P_{n})$ and the fact that $k^{I}(P_{n},\mathbf{P})$ belongs to the test set, it follows that $\varpi (\nu_{k^{F}_{s}(P_{n})}(P_{n}) ,  \nu_{k^{I}(P_{n},\mathbf{P})}(P_{n})) \precsim  \tilde{V}_{k^{I}(P_{n},\mathbf{P})}(P_{n}) $ with high probability. 

   	\medskip
   	
   	We illustrate how the choice of tuning parameter works in the HD-QR example.

      \begin{example}[HD-QR (cont.)]
      	We study the case with $d >> n(\beta)$. Here,  choosing $k$ amounts to choosing $\lambda_{k}$; hence we index the relevant quantities directly by $\lambda$. The metric $\varpi = ||.||_{\ell^{2}}$ and $\underline{\varpi}_{k}(M,\epsilon) = e_{min}(W_{k})$ and we assume $\min_{k} e_{min}(W_{k}) \equiv \underline{\varpi} \in (0,1]$. 
      	
      	For $Pen(.) = ||.||_{\ell^{1}}$, Assumption \ref{ass:V-bdd} is satisfied with $\tilde{V}_{\lambda}(P_{n}) = \left(\frac{\log (2 d)}{n(\beta)} \right)^{1/4} \sqrt{ \frac{ n^{-1}\sum_{i=1}^{n} \phi(Z_{i},0)} { \lambda} }$ and $\mathbb{V} = \mathbb{V}_{1} \equiv \underline{\varpi}^{-1}\mathbb{K} \max\{1, 2^{1/\pi_{0}} \mathbb{E}_{\pi_{0}}  \}$. For $Pen(.) = ||.||_{\ell^{2}(p)}$ (with $m>0$), Assumption \ref{ass:V-bdd} is satisfied with $\tilde{V}_{\lambda}(P_{n}) = 2 \sqrt{\frac{\lambda_{k}^{-1/m}}{n(\beta)}  } $ and $\mathbb{V} = \mathbb{V}_{1} \max\{ 1, \left(\frac{\underline{\varpi}}{2(m-1)} \right)^{\frac{1-m}{2m}}  \}$. Finally, the test set can be written directly in term of $\lambda$ as
      	\begin{align*}
      	\mathcal{L}_{s}(P_{n}) \equiv \{ \lambda \in \mathcal{L} \mid ||\nu_{\lambda}(P_{n}) - \nu_{\lambda'}(P_{n}) ||_{\ell^{2}} \leq 4 s \tilde{V}_{\lambda'}(P_{n}) ,~\forall \lambda' \leq \lambda   \}
      	\end{align*}  where  $\mathcal{L}$ is a grid in $\mathbb{R}_{++}$. So, $\lambda_{k^{F}(P_{n})}$ amounts to choosing the largest $\lambda$ in $\mathcal{L}_{s}(P_{n})$.
            	
      	 The next proposition establishes the ``ideal" concentration rate, which, by Theorem \ref{thm:k-fea}, coincides (up to constants) with the one obtained by our feasible choice.
      	

      	\begin{proposition}
      		\label{pro:LASSO-choice}
      		Suppose  and $0 \in \Theta$. Then: 
      		
      		(1) If $Pen = ||.||_{\ell^{1}}$, then $\tilde{V}_{\lambda_{k^{I}(P_{n},\mathbf{P})}}(P_{n}) \leq  (0.5 \underline{\varpi} )^{-1} \left(\frac{\log (2 d)}{n(\beta)} \right)^{1/8}  \left( ||\theta_{\ast}||_{\ell^{1}}  n^{-1}\sum_{i=1}^{n} \phi(Z_{i},0) \right)^{1/4}$.
      		
      		(2) If $Pen = ||.||_{\ell^{2}(p)}$, then $\tilde{V}_{\lambda_{k^{I}(P_{n},\mathbf{P})}}(P_{n}) \leq  n(\beta)^{-\frac{m}{2(m+1)}} ||\theta_{\ast}||_{\ell^{2}(p)}^{\frac{1}{m+1}} $.
      	\end{proposition}

      	\begin{proof}
      		See Appendix \ref{app:choice}.\footnote{The assumption $0 \in \Theta$ is only to get more tractable expressions. See the proof for an explanation.}
      	\end{proof}
      	
		The rate for (1) is very slow. This follows from the fact that we do not impose sparsity nor compatibility-type conditions (e.g. \cite{VdG-Buhlmann11}), nor bounded parameter set (e.g. \cite{Chatterjee2013}). In fact, we think it is rather surprising that consistency can be achieved without any of these assumptions and still allowing for large number of parameters relative to the (effective) sample size. The rate for (2) coincides with the min-max rate for Normal Means models on $\ell^{2}(p)$ balls (\cite{Wasserman2006} Ch. 7). $\triangle$
      \end{example}      
      
      \section{Numerical Simulations}
      \label{sec:simul}
      
      We now present numerical simulations to further illustrate the role of the effective number of observations and also the behavior of our choice of tuning parameters.
      
      \subsection{The effective number of observations}
      
      We study the following simple regression model with $m$-dependent data: $Y_{i} = \sum_{j=1}^{d} \theta_{\ast}(j) X_{j,i} + 0.5 U_{i}$ for $i=1,....,n$ and $\theta_{\ast}(j) = j^{-0.5}$ for $j=1,...,d$. The data $(X_{i},U_{i})_{i=1}^{n}$ is constructed as follows: Let $q=n/m$, $(X_{1,j})_{j=1}^{q}$ and $(U_{1,j})_{j=1}^{q}$ be independent standard Gaussian, and $(X_{2,i})_{i=1}^{n}$ and $(U_{2,i})_{i=1}^{n}$ be independent standard Gaussian, finally 
      \begin{align*}
      	U_{l+(j-1)m} = U_{1,j} + 0.01 U_{2,l+(j-1)m},~and~X_{l+(j-1)m} = X_{1,j} + 0.01 X_{2,l+(j-1)m}
      \end{align*}
      for all $j=1,...,q$ and $l=1,...,m$. It is clear that $(Y_{i},X_{i})_{i=1}^{n}$ are $m$-block i.i.d., and, as pointed in Example \ref{exa:OLS-q}, $m$-dependent. 
      
      Let $(z,\theta) \mapsto \phi (z,\theta) = \rho_{j}(y - \sum_{l=1}^{d} \theta(l) x_{l})$ for $j=1,2$,  with $t \mapsto \rho_{1}(t) = 0.5|t|$ and $t \mapsto \rho_{2}(t) = t^{2}$.  We want to assess how sharp our upper bounds are and, in particular, assess the role of the effective number of observations. To this end, we abstract from other features of the setup and set $d = 3$ and $\lambda_{k} = 0$ for all $k \in \mathbb{N}$.
      
             We perform $MC=2,000$ Monte Carlo repetitions. For each $(n,m)$, and $j=1,2$ (indexing the type of regression),  we report the expectation of $ \delta_{\mathbf{P}} (\nu_{k}(P_{n}),\theta_{0}) $, 
            \begin{align*}
            	 \mu_{j}(n,m)  \equiv E_{\mathbf{P}} \left[  \delta_{\mathbf{P}} (\nu_{k}(P_{n}),\theta_{0})  \right].
            \end{align*}
       
            
      For this case, $\varrho_{n,k} = \mathbb{C}_{j} \sqrt{m} \sqrt{\frac{3}{n}}  $ where $(\mathbb{C}_{j})_{j=1,2}$ are universal constants. Thus Proposition \ref{pro:L1-rate} predicts that $\mu_{j}(n,m) \leq \mathbb{C}_{j} \sqrt{m} \sqrt{\frac{3}{n}} \left( \mathbb{C}'_{j} + \mathbb{G}_{0} \log \left( \sqrt{\frac{n}{3 m}}  \right) \right) $, for some universal constant $\mathbb{C}'_{j}$. Essentially, our theory predicts an upper bound for the growth of $m \mapsto  \mu_{j}(n,m)$ of the order of $\sqrt{m}$.
            
      In Tables \ref{tab:median} and \ref{tab:mean} we report $\mu_{1}(n,m)$ and $\mu_{2}(n,m)$ resp. for different values of $(n,m)$.\footnote{Because each row has a different value of $n$, the values of $m$ also differ. For all $n$ except $n=1000$, the right most number in each row corresponds to the case of $n(\beta) = 25$. For $n=1000$ we add the case $n(\beta) = 10$ for further comparison.} In the tables, ``actual" stands for the actual value of $\mu_{j}(n,m)$ stemming from the numerical simulations. The ``predicted" value for any $m>1$, is constructed as $\sqrt{m} \times \mu_{j}(n,1)$ for each $j=1,2$. The discrepancy between the ``predicted" and the ``actual" value gauges how sharp our bound is and ultimately, how good of a description of the estimator our theory provides. For all cases, even for fairly small sample sizes (e.g., $n=50$ or $n=100$), the ``predicted" value seems to be a remarkably good approximation of the actual quantity, for both mean and median regressions; except perhaps the case $(n,m)=(1000,100)$ for the median regression case.\footnote{In most $(n,m)$ cases, the predicted value is below the actual value, but in these cases the discrepancy is very small and attributed to numerical randomness in the Monte Carlo simulations.} 
      
      For comparison, the \textquotedblleft standard" asymptotic results predict that $\mu_{j}(n,m)$ has a convergence rate of order $n^{-1}$, independently of the value of $m$. Tables \ref{tab:median} and \ref{tab:mean} show that our results provide a sharper description of the behavior of the estimator. Moreover, for small sample sizes such as $n=50$, one might feel that \emph{asymptotic} results are not even applicable.

          \begin{table}[ht]
          	\centering
          	{\footnotesize{
          	\label{multiprogram}
          	\begin{tabular}{c|c|c||c|c|c|c} \hline \hline
          		  		& & $m=1$ & $m=2$ &  &   &\\	\hline
          		  		\multirow{ 2}{*}{$n=50$} & Actual & 11.88 & 16.95 & 
          		  		&  & \\
          		  		&  Predicted   &  & 16.81 & 	  &  &\\	\hline\hline
          		  		& & $m=1$ & $m=2$ &  $m=4$ &  &  \\	\hline
          		  		\multirow{ 2}{*}{$n=100$} & Actual & 6.184 & 8.896 & 12.62 & & \\
          		  		&  Predicted   &  & 8.745 & 12.37  &  & \\	\hline\hline
          		  		& & $m=1$ & $m=5$ &  $m=10$ &  & \\	\hline
          		  		\multirow{ 2}{*}{$n=250$} & Actual & 4.021 &   8.992 & 13.08  & 
          		  	&	\\
          		  		&  Predicted   &   & 8.991  & 12.72  &  &	\\	\hline\hline
          		  		& & $m=1$ & $m=10$ &  $m=20$ & $m=40$  & $m=100$  \\	\hline
          		  		\multirow{ 2}{*}{$n=1000$} & Actual & 1.986 & 6.318 & 9.171 & 12.68  & 21.24  \\
          		  		&  Predicted   &  & 6.282   &  8.890  & 12.56  & 19.86  \\	\hline\hline
          	\end{tabular}
          		\caption{Values of $100 \times \mu_{1}(n,m)$ for different values of $(n,m)$.}
          	}}
          		\label{tab:median}
          \end{table}

                    \begin{table}[ht]
                    	\centering
                    	{\footnotesize{
                    	\label{multiprogram}
                    	\begin{tabular}{c|c|c||c|c|c|c} \hline \hline
                    		& & $m=1$ & $m=2$ &  & & \\	\hline
                    		\multirow{ 2}{*}{$n=50$} & Actual & 8.931  & 12.97	 &  & & \\
                    		&  Predicted   &   & 12.63    &    &  & \\	\hline\hline
                    		& & $m=1$ & $m=2$ &  $m=4$ &  &  \\	\hline
                    		\multirow{ 2}{*}{$n=100$} & Actual & 8.122 & 11.53  & 17.25 &  &\\
                    		&  Predicted   &   & 11.48  & 16.24   & & \\	\hline\hline
                    		& & $m=1$ & $m=5$ &  $m=10$ &  & \\	\hline
                    		\multirow{ 2}{*}{$n=250$} & Actual & 5.078 & 11.77  & 17.07
                    		&  & \\
                    		&  Predicted   &   & 11.37  & 16.07 &  & \\	\hline\hline
                    		& & $m=1$ & $m=10$ &  $m=20$ & $m=40$ & $m=100$  \\	\hline
                    		\multirow{ 2}{*}{$n=1000$} & Actual & 2.544 & 8.228 & 11.65  &  17.24  & 22.07  \\
                    		&  Predicted   &    & 8.045   & 11.38 & 16.09  & 22.44 \\	\hline\hline
                    	\end{tabular}
                    	\caption{Values of $100 \times \mu_{2}(n,m)$ for different values of $(n,m)$.}
                    }}
                    	\label{tab:mean}
                    \end{table}

     Overall, the numerical simulations seem to suggest that our upper bound provides an accurate description of the size of $\mu$ for both models.  Notably, it does so for values of $n$ which one could find to be too low to apply asymptotic results. 
     
     \subsection{The choice of tuning parameter}
     
     In order to study the performance of our method for choosing the tuning parameter, we study the following non-parametric regression model with $m$-dependent data: $Y_{i} = \theta_{\ast}(W_{i}) + 0.5 U_{i}$ for $i=1,....,n$ and $\theta_{\ast}(w) = 2 \cos(w) + w$ for any $w \in [-6,6]$. The sequence $(W_{i})_{i=1}^{n}$ is such that $W_{i} = 6 \frac{X_{i}}{1 + |X_{i}|} $ and the data $(X_{i},U_{i})_{i=1}^{n}$ is constructed in the same way as in the previous design (except that the $X$'s are now in $\mathbb{R}$).

     We consider a sieve-based estimator with $\Theta_{k} = \{ f \colon f = \sum_{j=1}^{k} \theta_{j} q_{j},~\theta_{j} \in \mathbb{R},~\forall j =1,...,k   \}$ for two distinct basis functions  $(q_{j})_{j}$: (a) A polynomial basis and (b) a P-Spline basis with 3 equally spaced knots. The second case is of particular interest because assumption \ref{ass:V-bdd}(ii) may not hold here. We set $\lambda_{k} = 0$, so the regularization parameter is $k \in \mathcal{K}$ where $\mathcal{K} = \{ 3,...,8\}$. We consider $MC=2,500$. 
     
     It is easy to show that in this case, $\varpi$ is induced by the $L^{2}(\mathbf{P})$ norm. Therefore, the ideal tuning parameter, $k^{I}_{n}$, is proportional to the solution of $\sqrt{\frac{k}{n/m}} = || \sum_{j=1}^{k} \nu_{k}(j) q_{j}  - \theta_{\ast}||_{L^{2}(\mathbf{P})}$ where $(\nu_{k}(1),...,\nu_{k}(k))^{T} = \left( E[(q^{k}(W)) (q^{k}(W))^{T}]  \right)^{-1} E[q^{k}(W) \theta_{\ast}(W)]$ with $q^{k}(w) = (q_{1}(w),...,q_{k}(w))^{T}$. The feasible tuning parameter, $k^{F}_{n}$, is chosen as the minimal $k$ inside
     \begin{align*}
     	\{ k \in \mathcal{K} \colon  \sqrt{(\bar{\hat{\theta}}_{k,k'} - \hat{\theta}_{k'})^{T} M_{k'} (\bar{\hat{\theta}}_{k,k'} - \hat{\theta}_{k'})}  \leq 4 s_{n} \sqrt{\frac{k'}{n/m}},~\forall k' \in \mathcal{K}~s.t.~k' \geq k    \}
     \end{align*} 
     where $s_{n}$ was chosen as $0.5 \log(n/m)$ and for any $k$ $\hat{\theta}_{k} \equiv \left( Q_{k} Q^{T}_{k}  \right)^{-1} \sum_{i=1}^{n} q^{k}(W_{i}) Y_{i}$ with $Q_{k} = (q^{k}(W_{1}),...,q^{k}(W_{n}))$, and $\bar{\hat{\theta}}_{k,k'} = (\hat{\theta}_{k},0,...,0) \in \mathbb{R}^{k'} $, finally $M_{k} = E[(q^{k}(W)) (q^{k}(W))^{T}]$. Proposition \ref{pro:choicek-asym} states that $r_{n} \equiv \frac{\sqrt{n} || (\hat{\theta}_{k^{F}_{n}})^{T} q^{k^{F}_{n}}   - \theta_{\ast}||_{L^{2}(\mathbf{P})} }{\sqrt{m k^{I}_{n}} s_{n} }$ is bounded in probability; so in order to gauge the behavior of $r_{n}$ we report its MC quantiles for different values of $n$. 
     
     Tables   \ref{tab:choice-psp} and \ref{tab:choice-pol} presents the results for P-splines and Polynomial basis, resp.. Even though there is a positive trend for the quantiles in Polynomial case, the values are well below $6$, the scaling constant in Theorem \ref{thm:k-fea}. For all columns except the last one, we set $m=1$; for the last column we set $m=6$ so the effective number of observations is $500$. The results show that with $m=6$ the behavior is closer to $n=500$ than it is to $n=3,000$ (as would have been predicted by asymptotic theory). Finally, even in cases where ideal and feasible tuning parameters differ, the corresponding ``variance" terms are close. Interestingly, for P-splines the ideal and feasible choices happen to coincide (except for $n=100$); for this last case,  $k^{I}_{n} < k^{F}_{n}$.

               \begin{table}[ht]
               	\centering
               	{\footnotesize{
               	\begin{tabular}{c|cccccc|c} \hline \hline 
               		$n(\beta)$            & 100   & 300   & 500   & 1000  & 2000  & 3000  & 500=3000/6 \\ \hline 
               		$quant_{0.1}(r_{n})$  & 0.151 & 0.101 & 0.097 & 0.092 & 0.091 & 0.094 & 0.091 \\[1ex]
               		$quant_{0.5}(r_{n})$  & 0.247 & 0.156 & 0.147 & 0.134 & 0.125 & 0.125 & 0.144 \\[1ex]
               		$quant_{0.9}(r_{n})$  & 0.388 & 0.228 & 0.206 & 0.185 & 0.170 & 0.166 & 0.204 \\ 	[1ex]		
               		$k^{I}_{n}$           & 5     & 7     & 7     & 7     & 7     & 7     & 7 \\[1ex]
               		$k^{F}_{n}$           & 7     & 7     & 7     & 7     & 7     & 7     & 7 \\     [1ex]
               		$V_{k^{I}_{n}}$       & 0.224 & 0.153 & 0.118 & 0.084 & 0.059 & 0.048 & 0.118  \\     [1ex]
                   	$V_{k^{F}_{n}}$       & 0.264 & 0.153 & 0.118 & 0.084 & 0.059 & 0.048 & 0.118 \\     [1ex]		
               			\hline \hline
               	\end{tabular}
               	\caption{Parameter choice for P-Splines basis.}
               }}
               	\label{tab:choice-psp}
               \end{table}

     \begin{table}[ht]
     	\centering
     	{\footnotesize{
     	\begin{tabular}{c|cccccc|c} \hline \hline 
     		n                    & 100   & 300   & 500   & 1000  & 2000  & 3000  & 500=3000/6 \\ \hline 
     		$quant_{0.1}(r_{n})$ & 0.417 & 0.471 & 0.558 & 0.696 & 0.896 & 1.031 & 0.555 \\[1ex]
     	    $quant_{0.5}(r_{n})$ & 0.464 & 0.492 & 0.573 & 0.705 & 0.902 & 1.036 & 0.571 \\[1ex]
      		$quant_{0.9}(r_{n})$ & 0.550 & 0.527 & 0.599 & 0.723 & 0.913 & 1.045 & 0.597\\ 	[1ex]		
     		$k^{I}_{n}$          & 5     & 7     & 7     & 7     & 7     & 7     & 7  \\[1ex]
     		$k^{F}_{n}$          & 5     & 5     & 5     & 5     & 5     & 5     & 5 \\     [1ex]	
     		$V_{k^{I}_{n}}$      & 0.224 & 0.153 & 0.118 & 0.084 & 0.059 & 0.048 & 0.118 \\     [1ex]
       		$V_{k^{F}_{n}}$      & 0.224 & 0.130 & 0.100 & 0.071 & 0.050 & 0.041 & 0.100 \\     [1ex]		
     		\hline \hline	
     	\end{tabular}
     	\caption{Parameter choice for Polynomial basis.}
     }}
     	\label{tab:choice-pol}
     \end{table}

\small
\bibliographystyle{plain}
\bibliography{myref}

\newpage
\appendix

\begin{center}
{\Huge{Appendix}}
\end{center}

	\section{On the Properties of the Parameter Set, Regularized Parameter Set and ``Bias" term}
	\label{app:prelim}
	
	The next assumption is sufficient for $\nu(\mathbf{P})$ being non-empty. Let $\tau$ be a topology over $\Theta$ which might differ from the one induced by $||.||_{\Theta}$; it coincides with the one in the Definition of Regularization structure.

	\begin{assumption}
		\label{ass:Q-lsco}
		There exists $M_{0} > \inf_{\theta \in \Theta }Q(\theta,\mathbf{P})$, such that $\{ \theta \in \Theta \mid Q(\theta,\mathbf{P}) \leq M \}$ is $\tau$-compact for all $M \leq M_{0}$.
	\end{assumption}
	
	Assumption \ref{ass:Q-lsco} is relatively weak, especially because the topology $\tau$ can be weaker than the one induced by $||.||_{\Theta}$; see \cite{CP-2012} for a discussion in the context of semi-/non-parametric estimation.
	
	\begin{lemma}
		\label{lem:vP-exist}
		Suppose assumption \ref{ass:Q-lsco} holds. Then, $\nu(\mathbf{P})$ is non-empty.
	\end{lemma}
	
	\begin{proof}[Proof of Lemma \ref{lem:vP-exist} ]
		Let $\Theta(M_{0}) = \{ \theta \in \Theta \mid Q(\theta,\mathbf{P}) \leq M_{0}  \}$. The set is $\tau$-compact and since $M_{0} > \inf_{\theta \in \Theta} Q(\theta,\mathbf{P})$, $\Theta(M_{0})$ is non-empty. Consider $\arg \min_{\theta \in \Theta(M_{0})} Q(\theta, \mathbf{P})$. By assumption, $Q(\cdot, \mathbf{P}) : \Theta(M_{0}) \subseteq \Theta \rightarrow \mathbb{R}$ is $\tau$ lower semi-continuous relative to $\Theta(M_{0})$ (see Definition 38.5 in \cite{Zeidler1985}). By Theorem 38.B in \cite{Zeidler1985},  $\arg \min_{\theta \in \Theta(M_{0})} Q(\theta, \mathbf{P})$ is non-empty. This establishes the desired result since any $\theta$ outside $\Theta(M_{0})$ will yield a strictly large value than $\min_{\theta \in \Theta(M_{0})} Q(\theta, \mathbf{P})$.
	\end{proof}

	The next lemma establishes existence of the regularized parameter set.
	
	\begin{lemma}
		\label{lem:reg-min}
		Suppose assumption \ref{ass:Q-lsco} holds. Then, for any $k \in \mathbb{N}$, $\nu_{k}(\mathbf{P})$ is non-empty.
	\end{lemma}
	
	\begin{proof}[Proof of Lemma \ref{lem:reg-min} ]
		From the definition of a regularization structure and assumption \ref{ass:Q-lsco}, $Q_{k}$ is $\tau$ - lower-semicompact. Since $\Theta_{k}$ is closed, $\{ \theta \in \Theta_{k} : Q_{k}(\theta, \mathbf{P}) \leq M  \}$ is $\tau$-compact for all $M \leq M_{0}$. The proof follows along the same lines of the one in Lemma \ref{lem:vP-exist}.
	\end{proof}



The next Lemma provides sufficient conditions to ensure that the ``bias" term vanishes. 

  \begin{lemma}
  	\label{lem:Bias-incr}
  	Suppose assumption \ref{ass:Q-lsco} holds. Then, if  $Q(\cdot,\mathbf{P})$ is u.s.c. under the $\tau$-topology, then $\lim_{k \rightarrow \infty} B_{k}(\mathbf{P}) = 0$.
  \end{lemma}

    \begin{proof}[Proof of Lemma \ref{lem:Bias-incr}]   	
    	Suppose $Q(\cdot,\mathbf{P})$ is u.s.c. Let $(\prod_{k} \nu (\mathbf{P}))_{k}$ be such that $\prod_{k} \nu (\mathbf{P}) \in \Theta_{k}$ for each $k \in \mathbb{N}$ and $\prod_{k} \nu (\mathbf{P})$ converges to $\nu(\mathbf{P})$ under $\tau$. Such sequence exists because condition 2 in the definition of regularization structure. Moreover, $Q_{k}(\nu_{k}(\mathbf{P}) , \mathbf{P})  \leq Q_{k}(\prod_{k} \nu(\mathbf{P}) , \mathbf{P}) $ for all $k \in \mathbb{N}$.  It is enough to show that the limit of $Q_{k}(\prod_{k} \nu(\mathbf{P}) , \mathbf{P})$ vanishes.
    	
    	Under assumption \ref{ass:Q-lsco} and the condition of u.s.c., $Q(\cdot,\mathbf{P})$ is continuous under $\tau$ and thus $\lim_{k \rightarrow \infty} Q(\prod_{k} \nu(\mathbf{P}) , \mathbf{P}) = Q(\nu(\mathbf{P}) , \mathbf{P})$. By condition 3 in the definition of regularization structure, $Pen$ is $\tau$-continuous and thus $\lim_{k \rightarrow \infty} Pen(\prod_{k} \nu(\mathbf{P}) , \mathbf{P}) = Pen(\nu(\mathbf{P})) < \infty$ and thus $\lim_{k \rightarrow \infty} \lambda_{k} Pen(\prod_{k} \nu(\mathbf{P}) , \mathbf{P}) = 0$.  	
       \end{proof}
       
       \begin{remark}
       	If, in addition, $\Theta_{k} \uparrow$. Then $k \mapsto B_{k}(\mathbf{P})$ is non-increasing. To show this, take $k_{1} \leq k_{2}$. Then, since $k \mapsto \lambda_{k}$ is non-increasing, it follows that $Q_{k_{2}}(., \mathbf{P}) \leq Q_{k_{1}}(. , \mathbf{P})$. Moreover, $\Theta_{k_{1}} \subseteq \Theta_{k_{2}}$, then since $\nu_{k}(\mathbf{P})$ is a minimizer, $Q_{k_{2}}(\nu_{k_{2}}(\mathbf{P}) , \mathbf{P}) \leq Q_{k_{1}}(\nu_{k_{1}}(\mathbf{P}) , \mathbf{P})$. $\triangle$
       \end{remark}

\section{Proofs of Results in Section \ref{sec:main}}

Following \cite{talagrand2014} p. 13, for any metric space $T$, stochastic process $(X_{t})_{t\in T}$ and $A \subseteq T$, we define  $\sup_{t \in A} X_{t}$  as \begin{align*}
	\sup_{t \in A} X_{t} = \sup \left\{  \sup_{t \in F} X_{t};~F\subseteq A,~F~finite       \right\}.
\end{align*}

\subsection{Proof of Theorem \ref{thm:concen-main}}
\label{app:concen-main}

 \begin{proof}[Proof of Theorem \ref{thm:concen-main}]
 	Throughout the proof, fix a $(n,k)$. We divide the proof into several steps.  We note that for $u \in [0,\mathbb{G}_{0}]$ the bound is trivial, so we focus on $u \geq \mathbb{G}_{0} \geq 1$. \\
 	
 	\textsc{Step 1.} It follows that 
 	\begin{align*}
 		|\delta_{\mathbf{P}}(\theta,\nu(\mathbf{P}))|^{2} \leq |\delta_{k,\mathbf{P}}(\theta,\nu_{k}(\mathbf{P}))|^{2} + B_{k}(\mathbf{P}).
 	\end{align*}
 	Therefore  \begin{align*}
 	 \mathbf{P} \left( \delta_{\mathbf{P}}(\theta,\nu(\mathbf{P})) \geq u \varrho_{n,k}(\omega)   \right) \leq \mathbf{P} \left( \delta_{k,\mathbf{P}}(\theta,\nu(\mathbf{P})) + \sqrt{B_{k}(\mathbf{P})} \geq u \left( V_{n,k}(\omega) +  \sqrt{B_{k}(\mathbf{P})} \right)   \right).
 	\end{align*}
 	Thus, for any $u \geq 1$, it suffices to show that $\mathbf{P} \left( \delta_{k,\mathbf{P}}(\theta,\nu(\mathbf{P})) \geq u  V_{n,k}(\omega) \right) \leq g_{0}(u)$. 
 	
 	In order to do this, note that  	
 	\begin{align*}
 	& \mathbf{P} \left( \delta_{k,\mathbf{P}}(\nu_{k}(P_{n}),\nu_{k}(\mathbf{P})) \geq u V_{n,k}(\omega) \right)  \\
 	& \leq \sum_{l=1}^{\infty} \mathbf{P} \left( 2^{l}  u V_{n,k}(\omega)  \geq \delta_{k,\mathbf{P}}(\nu_{k}(P_{n}),\nu_{k}(\mathbf{P})) \geq 2^{l-1} u V_{n,k}(\omega) \right) \\
 	& \equiv \sum_{l=1}^{\infty} P_{n,k}(l).
 	\end{align*}	
 	We now bound $P_{n,k}(l)$ for each $l \in \{1,2,... \}$.\\

 	\textsc{Step 2.} Let $M_{n,k}(\omega) = Q_{k}(\theta_{k},P_{n}) + \eta_{n}$ for some $\theta_{k} \in \Theta_{k}$. Let, for any $r > 0$, \begin{align*}
 	I_{n,k}(\omega)(r) \equiv \{  \theta \in \Theta_{k}(M_{n,k}(\omega)) \mid  r  \geq \delta_{k,\mathbf{P}}(\theta,\nu_{k}(\mathbf{P})) \geq 0.5 r \},
 	\end{align*}
 	and $O_{n,k}(\omega)(r)  \equiv \Theta_{k}(M_{n,k}(\omega)) \setminus I_{n,k}(\omega)(r) $.  Observe that these are random sets. By definition, $Q_{k}(\nu_{k}(P_{n}),P_{n}) \leq Q_{k}(\theta_{k},P_{n}) + \eta_{n}$, and thus $\nu_{k}(\mathbf{P}_{n}) \in \Theta_{k}(M_{n,k}(\omega))$. Observe that $\Theta_{k}(M_{n,k}(\omega))$ is also a random set. We also note that $\nu_{k}(P_{n})  \in \{  \theta \in \Theta_{k} : Q(\theta,P_{n}) \leq M_{n,k}(\omega)  \}$ (which is also a random set). We will use this set and $\Theta_{k}(M_{n,k}(\omega))$ interchangeably.
 	
 	Let $\mathcal{L}_{n}(\theta) \equiv n^{-1}\sum_{i=1}^{n} \phi(Z_{i},\theta) - E_{\mathbf{P}}[\phi(Z,\theta)]$. Observe that 
 	\begin{align*}
 	 P_{n,k}(l) 	=  \mathbf{P} \left(  \inf_{\theta \in I_{n,k}(\omega)(2^{l} u V_{n,k}(\omega)) }  Q_{k}(\theta, P_{n}) \leq  \inf_{\theta \in O_{n,k}(\omega)(2^{l} u V_{n,k}(\omega)  ) }  Q_{k}(\theta, P_{n})  + \eta_{n}   \right),
 	\end{align*}
 	and that $\nu_{k}(\mathbf{P}) \subseteq O_{n,k}(\omega)(2^{l} u V_{n,k}(\omega)  )$. And also that, for any $r>0$, \footnote{If $\nu_{k}(\mathbf{P})$ is a set, to construct $\mathcal{L}_{n}(.)$ at $\nu_{k}(\mathbf{P})$, we take one element, which we still denote as $\nu_{k}(\mathbf{P})$. } 
 	\begin{align*}
 	& \inf_{\theta \in I_{n,k}( \omega  )(r) } Q_{k}(\theta,P_{n})  - Q_{k}(\nu_{k}(\mathbf{P}),P_{n})  \\
 	= & \inf_{\theta \in      I_{n,k}( \omega  )(r)   } \left\{ Q_{k}(\theta,\mathbf{P})  - Q_{k}(\nu_{k}(\mathbf{P}),\mathbf{P})  + \mathcal{L}_{n}(\theta)  - \mathcal{L}_{n}(\nu_{k}(\mathbf{P})) \right\} \\
 	\geq & \inf_{\theta \in    I_{n,k}( \omega  )(r) } Q_{k}(\theta,\mathbf{P})  - Q_{k}(\nu_{k}(\mathbf{P}),\mathbf{P})  - \sup_{\theta \in   I_{n,k}( \omega  )(r) } |\mathcal{L}_{n}(\theta)  - \mathcal{L}_{n}(\nu_{k}(\mathbf{P}))|.
 	\end{align*}
 	
 	Thus, taking $r = 2^{l} u  V_{n,k}(\omega)$, it follows that
 	\begin{align*}
 	P_{n,k}(l) 	\leq  & \mathbf{P} \left(  \inf_{\theta \in  I_{n,k}( \omega  )(2^{l} u V_{n,k}(\omega)) } Q_{k}(\theta,\mathbf{P})  - Q_{k}(\nu_{k}(\mathbf{P}),\mathbf{P})  \right. \\
 	& \left. 	\leq   \sup_{\theta \in    I_{n,k}( \omega  )(2^{l} u V_{n,k}(\omega)) } |\mathcal{L}_{n}(\theta)  - \mathcal{L}_{n}(\nu_{k}(\mathbf{P}))| + \eta_{n}  \right) .
 	\end{align*}
 	
 	\medskip

 	\textsc{Step 3.} Note that $(\delta_{k,\mathbf{P}} ( \theta , \nu_{k}(\mathbf{P})  ))^{2} = Q_{k}( \theta,\mathbf{P})  - Q_{k}(\nu_{k}(\mathbf{P}),\mathbf{P}) $ for any $\theta \in \Theta_{k}$. Thus by the previous step
 	\begin{align*}
 	P_{n,k}(l) 	\leq  & \mathbf{P} \left(  (2^{l-1} u V_{n,k}(\omega))^{2}  \leq   \sup_{\theta \in    I_{n,k}( \omega  )(2^{l} u V_{n,k}(\omega)) } |\mathcal{L}_{n}(\theta)  - \mathcal{L}_{n}(\nu_{k}(\mathbf{P}))| + \eta_{n}  \right) \\
 	\leq & \mathbf{P} \left(  (2^{l-1} u V_{n,k}(\omega))^{2}  	\leq  (2^{l} u  ) \gamma_{n,k }(2^{l} u)  + \eta_{n}  \right) \\
 	& + \mathbf{P} \left(  \sup_{\theta \in    I_{n,k}( \omega  )(2^{l} u V_{n,k}(\omega)) } |\mathcal{L}_{n}(\theta)  - \mathcal{L}_{n}(\nu_{k}(\mathbf{P}))|   \geq (2^{l} u  ) \gamma_{n,k }(2^{l} u)  \right)\\
 	\leq & \mathbf{P} \left(  (2^{l-1} u V_{n,k}(\omega))^{2}  	\leq  (2^{l} u  ) \left\{ \gamma_{n,k }(2^{l} u) + \eta_{n}  \right\} \right) \\
 	& + \mathbf{P} \left(  \sup_{\theta \in    I_{n,k}( \omega  )(2^{l} u V_{n,k}(\omega) ) } |\mathcal{L}_{n}(\theta)  - \mathcal{L}_{n}(\nu_{k}(\mathbf{P}))|   \geq (2^{l} u  ) \gamma_{n,k }(2^{l} u)  \right)
 	\end{align*}
 	where $\gamma_{n,k }(2^{l} u) \equiv \gamma_{n}( I_{n,k}( \omega  )(2^{l} u V_{n,k}(\omega)))/\sqrt{n} $; the last line follows since $u \geq 1$.
 	
	For any $r > 0$, recall that $H_{n,k}(\omega)(r) \equiv \gamma_{n}( I_{n,k}( \omega )(r))/\sqrt{n} + \eta_{n}  $. By definition \begin{align*}
V_{n,k}(\omega) = \min  \left\{  r > 0 \mid  r^{2} \geq 5  \max _{x  \geq 1 } \frac{H_{n,k}(\omega)(r x) }{x}       \right\},
\end{align*}
 	thus $(V_{n,k}(\omega))^{2}  >  4 \max _{x  \geq 1 } \frac{H_{n,k}(\omega)(V_{n,k}(\omega) x) }{x} \geq 4 \frac{H_{n,k}(\omega)(V_{n,k}(\omega) 2^{l} u ) }{2^{l} u }  $ (because $x$ is maximal and $2^{l} u \geq 1 $). Multiplying at both sides by $(2^{l-1}u)^{2}$, this result implies 

 	\begin{align*}
 	P_{n,k}(l) 	\leq  \mathbf{P} \left(  \sup_{\theta \in    I_{n,k}( \omega  )(2^{l} u V_{n,k}(\omega)) } |\mathcal{L}_{n}(\theta)  - \mathcal{L}_{n}(\nu_{k}(\mathbf{P}))|   \geq (2^{l} u  ) \gamma_{n,k }(2^{l} u)  \right).
 	\end{align*}
 	
 	By theorem \ref{thm:gbrack} with $A \equiv I_{n,k}( \omega  )(2^{l} u V_{n,k}(\omega) )$, the RHS is less than $g_{0} (2^{l} u  )$. 
 	
 	\medskip
 	
 	\textsc{Step 4.} By Step 1-3 it follows that 
 	\begin{align*}
 		\mathbf{P} \left( \delta_{\mathbf{P}}(\theta,\nu(\mathbf{P})) \geq u \varrho_{n,k}(\omega)   \right) \leq  \sum_{l=1}^{\infty} g_{0} (2^{l} u  ) = \mathbb{G}_{0} u^{-1} \sum_{l=1}^{\infty} 2^{-l} \leq \mathbb{G}_{0} u^{-1}.
 	\end{align*}
 	
 \end{proof}

\subsection{Proof of Theorem \ref{thm:effe-n}} 
\label{app:effe-n}

\begin{proof}[Proof of Theorem \ref{thm:effe-n}]
	Fix any $r>2$. Suppose that, for any $s>0$, 	
	\begin{align*}
	H_{n,k}(\omega)(s) \leq  \frac{\gamma(I_{n,k}(\omega)(s),2^{1/r} ||.||_{L^{r}(\mathbf{P})})}{\sqrt{n(\beta)}} + \eta_{n}.
	\end{align*}
	Then, $V_{n,k}(\omega) \leq \min \left\{ s > 0 \mid  s \geq 5 \max _{x\geq 1} \frac{\frac{\gamma(I_{n,k}(\omega)(sx),2^{1/r} ||.||_{L^{r}(\mathbf{P})})}{\sqrt{n(\beta)}} + \eta_{n}}{sx}    \right\} $. To show this, suppose the minimum in the RHS is attained (if not, the proof follows by using subsequences); denote it as $v$. We show the result by means of contradiction. That is, suppose that $V_{n,k}(\omega) > v$. We know that $v \geq \max _{x\geq 1} \frac{\frac{\gamma(I_{n,k}(\omega)(vx),2^{1/r}||.||_{L^{r}(\mathbf{P})})}{\sqrt{n(\beta)}} + \eta_{n}}{vx} \geq \max _{x\geq 1} \frac{H_{n,k}(\omega)(vx)}{vx} $. But this contradicts that $V_{n,k}(\omega)$ is minimal. 
	
	Therefore, we need to establish the first display. 	By definition, for any $k =0,1,2,...$, the quantity $q_{n,k}$ is the minimal $q \in \mathcal{Q}_{n}$ such that $0.5 \frac{\beta(q)}{q} \leq \frac{2^{k+1}}{n}$. Since $q \mapsto \beta(q)/q$ is non-increasing, it follows that $q_{n,k} \leq q_{n,0}$. Consequently $\int_{0}^{1} |\mu_{q_{n,k}}(u)|^{\frac{r}{r-2}}    du \leq \int_{0}^{1} |\mu_{q_{n,0}}(u)|^{\frac{r}{r-2}}    du$. Hence, by display \ref{eqn:qnorm-bound}, $||.||_{q_{n,k}} \leq 2^{(r-2)/(2r)} \left( \int_{0}^{1} |\mu_{q_{n,0}}(u)|^{\frac{r}{r-2}}  du  \right)^{\frac{r-2}{2r}} 2^{1/r}||\cdot ||_{L^{r}(\mathbf{P})}$ 
	for any $k=0,1,2,...$. 
	
	This implies that for any $k=0,1,2,...$ and any $f$, \begin{align*}
	||S(f,\mathcal{T}_{k})||_{q_{n,k}} \leq \sqrt{2} \left( \int_{0}^{1} |\mu_{q_{n,0}}(u)|^{\frac{r}{r-2}}  du  \right)^{\frac{r-2}{2r}} ||S(f,\mathcal{T}_{k})||_{L^{r}(\mathbf{P})}.
	\end{align*}
	
	By our definition of $\gamma_{n}$: For any $A \subseteq \Theta$,
	\begin{align*}
	\gamma_{n}(A) \leq & \sqrt{2} \inf \sup \sqrt{2} \sum_{k} 2^{k/2} \left( \int_{0}^{1} |\mu_{q_{n,0}}(u)|^{\frac{r}{r-2}}  du  \right)^{\frac{r-2}{2r}} ||S(f,\mathcal{T}_{k})||_{L^{r}(\mathbf{P})}\\
	= & \sqrt{2} \left( \int_{0}^{1} |\mu_{q_{n,0}}(u)|^{\frac{r}{r-2}}  du  \right)^{\frac{r-2}{2r}} \inf \sup \sqrt{2} \sum_{k} 2^{k/2} ||S(f,\mathcal{T}_{k})||_{L^{r}(\mathbf{P})}\\
	= & \sqrt{2} \left( \int_{0}^{1} |\mu_{q_{n,0}}(u)|^{\frac{r}{r-2}}  du  \right)^{\frac{r-2}{2r}} \gamma (A, ||.||_{L^{r}(\mathbf{P})})\\
	= & 2^{(r-2)/(2r)} \left( \int_{0}^{1} |\mu_{q_{n,0}}(u)|^{\frac{r}{r-2}}  du  \right)^{\frac{r-2}{2r}} \gamma (A, 2^{1/r} ||.||_{L^{r}(\mathbf{P})}).
	\end{align*}
	
	Hence, by setting $A = I_{n,k}(\omega)(s)$, the result follows from the definition of $H_{n,k}$.
\end{proof}

\subsection{Proof of Proposition \ref{pro:L1-rate}, Proposition \ref{pro:concen-w} and Proposition \ref{pro:bound-H}}
\label{app:sec-main-prop}
\label{app:concen-w}

\begin{proof}[Proof of Proposition \ref{pro:concen-w}]
	The proof is analogous to the one of Theorem \ref{thm:concen-main}, so we only present a sketch highlighting the discrepancies. Recall that  $\varpi(\cdot, \nu_{k}(\mathbf{P})) = \inf_{\theta_{0} \in \nu_{k}(\mathbf{P})} \varpi(\cdot,\theta_{0})$. 
	
	As in the proof of Theorem \ref{thm:concen-main}, it suffices to show, for any $l = 1,2,...$,
	\begin{align*}
	\mathbf{P} \left(  2^{l} u \tilde{V}_{n,k}(\omega) \geq \varpi (\nu_{k}(P_{n}),\nu_{k}(\mathbf{P})) \geq 2^{l-1} u  \tilde{V}_{n,k}(\omega)      \right) \leq g_{0}(2^{l} u)
	\end{align*}
	for $u \geq 1$. The RHS is bounded by\\ $\mathbf{P} \left(  \inf_{\theta \in I_{n,k}(\omega)(2^{l} u \tilde{V}_{n,k}(\omega)) } Q_{k}(\theta,P_{n}) \leq \inf_{\theta \in \Theta_{n,k}(M_{n,k}) \setminus I_{n,k}(\omega)(2^{l} u \tilde{V}_{n,k}(\omega)) } Q_{k}(\theta,P_{n})     \right)$. Our goal is to bound this quantity by $ g_{0}(2^{l} u)$. We note that here, for any $r>0$ $I_{n,k}(\omega)(r) = \{ \theta \in \Theta_{k}(M_{n,k}(\omega)) \mid r \geq \varpi (\theta,\nu_{k}(\mathbf{P}))  \geq 0.5 r \} \subseteq \{ \theta \in \Theta_{k}(M_{n,k}(\omega)) \mid \varpi (\theta,\nu_{k}(\mathbf{P}))  \geq 0.5 r \} $. So by assumption \ref{ass:IU}, for any $ r>0$
	\begin{align*}
	Q_{k}(\theta,\mathbf{P}) - Q_{k}(\nu_{k}(\mathbf{P}),\mathbf{P}) \geq \varpi(\theta,\nu_{k}(\mathbf{P}))^{2}  (\underline{\varpi}_{k}(M_{n,k}(\omega),0.5 r)) ^{2}  \geq (0.5 r)^{2}  (\underline{\varpi}_{k}(M_{n,k}(\omega), 0.5 r))^{2}
	\end{align*}
	for any $\theta \in I_{n,k}(\omega)(r)$. Taking, $r = 2^{l} u \tilde{V}_{n,k}(\omega)$, it follows that  
	\begin{align*}
	& \mathbf{P} \left(   (2^{l-1} u  V_{n,k}(\omega) \underline{\varpi}_{k}(M_{n,k}(\omega),2^{l-1} u  \tilde{V}_{n,k}(\omega)) )^{2}  \leq  2^{l} u \left\{ \gamma_{n,k}(2^{l} u)  + \eta_{n}      \right\}       \right) \\
	& \iff 	\mathbf{P} \left(     (\tilde{V}_{n,k}(\omega))^{2}   \leq 4 \frac{\left\{ \gamma_{n,k}(2^{l} u)  + \eta_{n}      \right\} }{2^{l} u ( \underline{\varpi}_{k}(M_{n,k}(\omega),2^{l-1} u  \tilde{V}_{n,k}(\omega)) )^{2} }      \right).
	\end{align*}
	
	Since $\tilde{V}_{n,k}(\omega) = \min \left\{ r > 0 \colon r^{2} \geq 5 \max_{x \geq 1} \frac{\left\{ \gamma_{n,k}(x)  + \eta_{n}      \right\} }{x ( \underline{\varpi}_{k}(M_{n,k}(\omega), 0.5 x r) )^{2} }      \right\}$, it follows that  $\tilde{V}_{n,k}(\omega)^{2} > 4 \max_{x \geq 1} \frac{\left\{ \gamma_{n,k}(x)  + \eta_{n}      \right\} }{x ( \underline{\varpi}_{k}( M_{n,k}(\omega), 0.5 x \tilde{V}_{n,k}(\omega) ) )^{2} } \geq  4 \frac{\left\{ \gamma_{n,k}(2^{l} u)  + \eta_{n}      \right\} }{2^{l} u ( \underline{\varpi}_{k}(M_{n,k}(\omega), 2^{l-1} u  \tilde{V}_{n,k}(\omega)) )^{2} }    $. Thus, the probability in the previous display is zero.
	
	By using this observation and analogous algebra to the one in the proof of Theorem \ref{thm:concen-main}, it follows that in order to achieve our goal if suffices to bound \\ $\mathbf{P} \left(  \sup_{\theta \in I_{n,k}(\omega)(2^{l} u \tilde{V}_{n,k}(\omega))} |\mathcal{L}_{n}(\theta) - \mathcal{L}_{n}(\nu_{k}(\mathbf{P})) | \geq 2^{l} u \gamma_{n,k}(2^{l} u)    \right) $ by $g_{0}(2^{l}u)$.  Its steps are identical to those in the proof of Theorem \ref{thm:concen-main} and thus omitted. 
\end{proof}					

\begin{proof}[Proof of Proposition \ref{pro:L1-rate}]
	Let $\pi_{n,k}(\omega) =  \frac{\delta_{\mathbf{P}}(\nu_{k}(P_{n}),\nu(\mathbf{P}))}{\varrho_{k,n}(\omega)}  $ for any $\omega \in \Omega$. Observe that, for any $A \geq 1$,
	\begin{align*}
	E_{\mathbf{P}} \left[  \pi_{n,k}(\omega)  \right] = & 	E_{\mathbf{P}} \left[ \pi_{n,k}(\omega) 1\{ \pi_{n,k}(\omega) \geq A  \}    \right]  + 	E_{\mathbf{P}} \left[ \pi_{n,k}(\omega) 1\{  \pi_{n,k}(\omega) \leq  A  \}    \right].
	\end{align*}
	By assumption, the first term in the RHS, is bounded above by $\varphi_{n,k}(A)$. Regarding the second term, by Theorem \ref{thm:concen-main}, it follows that
	\begin{align*}
	E_{\mathbf{P}} \left[ \pi_{n,k}(\omega) 1\{  \pi_{n,k}(\omega) \leq  A  \}    \right] \leq & \int_{0}^{1} \mathbf{P} \left(  \frac{\delta_{\mathbf{P}}(\nu_{k}(P_{n}),\nu(\mathbf{P}))}{\varrho_{k,n}(\omega)} \geq t   \right) dt \\
	& + \int_{1}^{A} \mathbf{P} \left(  \frac{\delta_{\mathbf{P}}(\nu_{k}(P_{n}),\nu(\mathbf{P}))}{\varrho_{k,n}(\omega)} \geq t   \right) dt \\
	\leq & 1 + \int_{1}^{A} g_{0}(t)dt = 1 + \mathbb{G}_{0} \ln (A).
	\end{align*}
	Since $A \geq 1$ is arbitrary, we can choose the minimal value.
\end{proof}

In what follows let $\mathbb{B}_{r}(q) = \sqrt{2} \left( \int_{0}^{1} |\mu_{q}(u)|^{\frac{r}{r-2}}    \right)^{(r-2)/(2r)}$, for any $ r> 2$. 
The next lemma solves for $q \mapsto \mathbb{B}_{r}(q)$ for different cases of $\beta$.

\begin{lemma}\label{lem:bound-B}
	Let $m_{0} > 0$. For any $n$, any $q \in \mathcal{Q}_{n}$ and any $ r> 2$, the following hold:\\

	\begin{enumerate}
		\item If $\beta(q) = 1\{ q < m^{-1}_{0} \}$, then 
		\begin{align*}
		&\mathbb{B}_{r}(q) \leq  2^{1/r} \sqrt{\min \{ 1+q, 1+m^{-1}_{0} \}},\\
		& and~ \mathbb{B}_{r}(q) \geq 2^{1/r} \sqrt{\min \{ 1+q, m^{-1}_{0} \}}.
		\end{align*}
		\item If $\beta(q) = (1+q)^{-m_{0}}$ with $m_{0} > \frac{r}{r-2}$, then
		\begin{align*}
		\mathbb{B}_{r}(q) \leq 2^{1/r} \left( \frac{m_{0}}{m_{0}-\frac{r}{r-2}} \right)^{(r-2)/(2r)}
		\end{align*}
		
		and
		\begin{align*}
		\mathbb{B}_{r}(q) \geq   2^{1/r} \left( 2^{- m_{0}}   \right)^{\frac{r-2}{2r}}.
		\end{align*}
		\item If $\beta(q) = (1+q)^{-m_{0}}$ with $m_{0} = \frac{r}{r-2}$, then
		\begin{align*}
		\mathbb{B}_{r}(q) \leq 2^{1/r}   \left( m_{0} \log (1+q) +1     \right)^{\frac{1}{2m_{0}}},
		\end{align*}
		and 
		\begin{align*}
		\mathbb{B}_{r}(q) \geq 2^{1/r}  \sqrt{0.5} \left( m_{0} \log (1+0.5q) + 2^{m_{0}}    \right)^{\frac{1}{2m_{0}}}
		\end{align*}

		\item If $\beta(q) = (1+q)^{-m_{0}}$ with $m_{0} < \frac{r}{r-2}$, then
		\begin{align*}
		\mathbb{B}_{r}(q) \leq 2^{1/r} \left( \frac{\frac{r}{r-2}}{\frac{r}{r-2}-m_{0}} \right)^{(r-2)/(2r)}  (1+q)^{\frac{r -(r-2)m_{0}}{2r}}
		\end{align*}
		and 
		\begin{align*}
		\mathbb{B}_{r}(q) \geq  2^{1/r} \left( \frac{2^{- m_{0}}  m_{0}}{\frac{r}{r-2}-m_{0}}  \right)^{\frac{r-2}{2r}} ((1+0.5q)^{\frac{r}{r-2}- m_{0}} - 1)^{\frac{r-2}{2r}}
		\end{align*}
		
	\end{enumerate}
\end{lemma}

\begin{proof}
	See the Online Appendix \ref{app:supp-lem-bound-H}.
\end{proof}

\begin{proof}[Proof of Proposition \ref{pro:bound-H}]
	Given our definition of $n(\beta)$ and our notation, it suffices to provide upper and lower bounds on $\mathbb{B}_{r}(q)$. Throughout the proof, $r$ is fixed to an arbitrary value $r>2$. \\

	(1) Note that $q_{n,0}$ is such that $\beta(q_{n,0})/q_{n,0} \leq \frac{4}{n}$. Recall that for any $x \in \mathbb{R}_{+}$, $[x]$ is the smallest element, $s$ in $\mathcal{Q}_{n}$ such that $s \geq x$. 
	
	Observe that $\beta(q_{n,0})/q_{n,0} \leq 1/q_{n,0}$, so any value of $q_{n,0} \geq 0.25n$ will satisfy the restriction $\beta(q_{n,0})/q_{n,0} \leq \frac{4}{n}$. Thus, $q_{n,0} = [0.25n]$. Therefore, by Lemma \ref{lem:bound-B}(1), $\mathbb{B}_{r}(q_{n,0}) \leq 2^{1/r} \sqrt{ \min\{  1 + [0.25n] , 1+ m^{-1}_{0}   \} } $. Therefore, $n(\beta) \geq \frac{n}{  \min\{ 1 + [0.25n], 1+ m^{-1}_{0}    \} } $. Also, $\mathbb{B}_{r}(q_{n,0}) \geq 2^{1/r} \sqrt{ \min\{ 1 + [0.25n]  , m_{0}^{-1} \} } $, and thus $n(\beta) \leq \frac{n}{ 2^{2/r} \min\{ 1 + [0.25n]  , m_{0}^{-1} \}    } \leq \frac{n}{  \min\{ 1 + 0.25n  , m_{0}^{-1} \}    }$.

	\medskip 
	
	(2) It is easy to see that by Lemma \ref{lem:bound-B}(2), a bound for  $\mathbb{B}_{r}(q)$ does not depend on $q$ and thus the result holds with $n(\beta) \geq \frac{n}{\left( \frac{m_{0}(r-2)}{m_{0}(r-2)-r}  \right)^{(r-2)/r}}$. The lower bound for $\mathbb{B}_{r}(q)$ in Lemma \ref{lem:bound-B}(2) implies that $n(\beta) \leq \frac{n}{\left(2^{-m_{0}} \right)^{(r-2)/r}}$ .

	\medskip 
	
	(3) We first establish that $q_{n,0} \leq \left[ \left( \frac{n}{4} \right)^{1/(m_{0}+1)}  \right]$ and also that $q_{n,0} \geq \left( \frac{n}{4} \right)^{1/(m_{0}+1)} - 1$. To show this note that $\beta(q)/q \leq q^{-(m_{0}+1)}$. If $q_{n,0} > \left[ \left( \frac{n}{4} \right)^{1/(m_{0}+1)}  \right]$, then by definition of the operator $[.]$ there exists a $s \in \mathcal{Q}_{n}$ such that $ q_{n,0} > s \geq \left( \frac{n}{4} \right)^{1/(m_{0}+1)}$. Thus $\frac{4}{n} > s^{-(1+m_{0})}  \geq \beta(s)/s$. But this contradicts the fact that $q_{n,0}$ is minimal in $\mathcal{Q}_{n}$ such that $\beta(s)/s \leq  \frac{4}{n}$. Clearly, $(q_{n,0})^{-(m_{0}+1)} \leq \frac{4}{n}$ and thus $q_{n,0} \geq \left( \frac{n}{4} \right)^{1/(m_{0}+1)} - 1$.

	Thus, by Lemma \ref{lem:bound-B}(3), 
	\begin{align*}
	\mathbb{B}_{r}(q_{n,0}) \leq 2^{1/r} \left( m_{0} \log (1+ \left[ \left( \frac{n}{4} \right)^{\frac{1}{m_{0}+1}}  \right]) + 1 \right)^{\frac{1}{2m_{0}}},
	\end{align*}
	and 
	\begin{align*}
	\mathbb{B}_{r}(q_{n,0}) \geq 2^{1/r} \sqrt{0.5}  \left( m_{0}  \log  \left(0.5 +  0.5 \left( \frac{n}{4}  \right)^{\frac{1}{m_{0}+1}} \right) + 2^{m_{0}} \right)^{\frac{1}{2m_{0}}},
	\end{align*}
	
	These inequalities imply that\begin{align*}
	&n(\beta) \geq \frac{n}{  \left( m_{0} \log (1+ \left[ \left(\frac{n}{4} \right)^{\frac{1}{m_{0}+1}} \right]) + 1 \right)^{\frac{1}{m_{0}}} }\\
	and&~n(\beta) \leq \frac{n}{ 0.5  \left( m_{0}  \log  \left(0.5 +  0.5 \left( \frac{n}{4}  \right)^{\frac{1}{m_{0}+1}} \right) +  2^{m_{0}}\right)^{\frac{1}{m_{0}}} }.
	\end{align*}

	
	\medskip 
	
	(4) We can obtain the same bound for $q_{n,0}$ as in point (3), so by Lemma \ref{lem:bound-B}(4), \begin{align*}
	\mathbb{B}_{r}(q_{n,0}) \leq 2^{1/r} \left(1+ \left[ \left( \frac{n}{4} \right)^{\frac{1}{m_{0}+1}} \right] \right)^{\frac{r-(r-2)m_{0}}{2r}} \left( \frac{r}{r-m_{0}(r-2)}  \right)^{(r-2)/(2r)}, 
	\end{align*}
	and 
	\begin{align*}
	\mathbb{B}_{r}(q_{n,0}) 	\geq  2^{1/r} \left( \left(0.5 (1+ \left( \frac{n}{4} \right)^{\frac{1}{m_{0}+1}}) \right)^{\frac{r}{r-2} - m_{0}} - 1 \right)^{\frac{r-2}{2r}} \left( \frac{2^{-m_{0}}m_{0}(r-2)}{r-m_{0}(r-2)}  \right)^{(r-2)/(2r)}
	\end{align*}
	
	The result follows from analogous calculations to those in step (3). 	
\end{proof}

\subsection{Proof of Theorem \ref{thm:gbrack}}

\label{app:gbrack}

To simplify notation we use $\mathcal{F}$ to denote $\mathcal{F}(A)$. For any $(Z_{i})_{i} = \omega \in \Omega$ and any $f \in \mathcal{F}$, let $\mathbf{L}_{n}(f) = n^{-1} \sum_{i=1}^{n} f(Z_{i}) - E_{\mathbf{P}}[f(Z)]$, and  for any $(Z^{\ast}_{i})_{i} = \omega^{\ast} \in \Omega$, let $\mathbf{L}^{\ast}_{n}(f) = n^{-1} \sum_{i=1}^{n} f(Z^{\ast}_{i}) - E_{\mathbf{P}^{\ast}}[f(Z)]$. Recall that the $Z^{\ast}$'s were constructed in Section \ref{sec:prem}.

	 Observe that the following equivalence holds: For any $f \in \mathcal{F}$ and its corresponding $\theta$, i.e., $f = \phi(.,\theta) - \phi(.,\nu_{k}(\mathbf{P}))$, it follows that
	\begin{align}
	\mathbf{L}_{n}(f) = \mathcal{L}_{n}(\theta) - \mathcal{L}_{n}(\nu_{k}(\mathbf{P})).
	\end{align}
	(Recall that $\phi(.,\nu_{k}(\mathbf{P}))$ is defined as $\phi$ evaluated at some element of $\nu_{k}(\mathbf{P})$ which is chosen a priori).
	
	Hence, in order to establish the expression in Theorem \ref{thm:gbrack} it suffices to show 
	\begin{align}\label{eqn:thm33-1}
	\mathbf{P}\left( \sup_{f \in \mathcal{F}} |\mathbf{L}_{n}(f) |\geq u \frac{\gamma_{n}(A)}{\sqrt{n}}  \right)  \leq g_{0}(u).
	\end{align}
	
	In order to show this we need some lemmas (the proofs are relegated to the Online Appendix \ref{app:supp-gbrack}).

For each $k$, given a partition $\mathcal{T}_{k}$, let $\pi_{k}(f)$ a point in $T(f,\mathcal{T}_{k})$ with $\pi_{0}(f) = 0$ (recall that $\mathcal{T}_{0} = \{\mathcal{F} \}$ and $0 \in \mathcal{F}$). Let $\Delta_{k} f \equiv  \pi_{k}(f) - \pi_{k-1}(f)$ and $\Xi_{k}(f) = -f + \pi_{k}(f)$.

The next lemma is essentially the first part of the Chaining argument in the proof of the CLT for empirical processes, see \cite{VdV-W1996} Ch. 2.5.

\begin{lemma}
	\label{lem:f-decom}
	For any admissible partition sequence of $\mathcal{F}$, $(\mathcal{T}_{n})_{n}$, any $f \in \mathcal{F}$, any $(v_{k})_{k}$ with $v_{k} : \mathcal{F} \rightarrow \mathbb{R}^{\mathbb{Z}}$, and any $(a_{k})_{k}$ with $a_{k} \geq 0$,
	\begin{align*}
	f = & \sum_{k=1}^{\infty} \Delta_{k} f 1_{\{m(f) \geq k \cap |v_{k} (f)| \leq a_{k}\}} - \sum_{k=1}^{\infty} \Xi_{k-1}(f) 1_{\{m(f) = k \cap |v_{k} (f)| > a_{k}\}} \\
	& - \Xi_{0}(f) 1_{m(f) = 0},
	\end{align*}
	where $m(f) = \min \{ j \mid  |v_{j} (f) | > a_{j}    \}$.
\end{lemma}

\begin{proof}
	See the Online Appendix \ref{app:supp-gbrack}.
\end{proof}

		Observe that $\Delta_{k} f = \Xi_{k}(f) - \Xi_{k-1}(f)$ and 
		\begin{align}\label{eqn:bdd-Xik}
		|\Xi_{k}(f)| = |\pi_{k}(f) - f| \leq S(f,\mathcal{T}_{k}),~\forall k,f
		\end{align}
		since both $f$ and $\pi_{k}(f)$ are in $T(f,\mathcal{T}_{k})$.

		The next lemma, uses the decomposition in Lemma \ref{lem:f-decom} to provide an upper bound for the LHS of display \ref{eqn:thm33-1}.			
		
		\begin{lemma}
			\label{lem:decom-L}
			For any admissible partition sequence of $\mathcal{F}$, $(\mathcal{T}_{n})_{n}$, any $(v_{k})_{k}$ with $v_{k} : \mathcal{F} \rightarrow \mathbb{R}^{\mathbb{Z}}$, any $(a_{k})_{k}$ with $a_{k} \geq 0$, the following holds: For all $\Gamma = \Gamma_{1} + \Gamma_{2} + \Gamma_{3}$ with $\Gamma_{l} \geq 0$, all $u \geq 0$, and all $n \in \mathbb{N}$,
			\begin{align}
			\mathbf{P}\left( \sup_{f \in \mathcal{F}} |\mathbf{L}_{n}(f) |\geq u \Gamma   \right) \leq & \sum_{l=1}^{2} \mathbf{P} \left( \sup_{f \in \mathcal{F}} |\mathbf{L}_{n}(\sum_{k=1}^{\infty} g^{l}_{k}(f)) |\geq u \Gamma_{l} \right) \\
			& + \mathbf{P} \left( \sup_{f \in \mathcal{F}} |\mathbf{L}_{n}(\Xi_{0}(f)1_{\{ m(f) = 0 \}}) |\geq u \Gamma_{3} \right)
			\end{align} 
			with, $z \mapsto g_{k}^{1}(f)(z) \equiv \Delta_{k} (f) (z) 1_{\{m(f,z) \geq k \cap |v_{k} (f)(z)| \leq a_{k}\}}$ and\\  $z \mapsto g_{k}^{2}(f)(z) \equiv \Xi_{k-1} (f) (z)  1_{\{m(f,z) = k \cap |v_{k} (f)(z)| > a_{k}\}} $. 
		\end{lemma}
		
		\begin{proof}
			See the Online Appendix  \ref{app:supp-gbrack}.
		\end{proof}
		
		We use the previous lemmas with the following choices for $(v_{k}(f))_{k}$ and $(a_{k})_{k}$: For any $(n,k) \in \mathbb{N}^{2}$
		\begin{align}
		v_{k}(f) = & S(f,\mathcal{T}_{k}),~and~a_{n,k}(f,\mathcal{T}_{k}) = \frac{b_{n,k}(f,\mathcal{T}_{k})}{q_{n,k}} \\
		with~b_{n,k}(f,\mathcal{T}_{k}) = & B_{0} \frac{\sqrt{n} \delta(f,\mathcal{T}_{k})}{\sqrt{2^{k+1}} },\\
		and~\delta(f,\mathcal{T}_{k}) = & ||S(f,\mathcal{T}_{k})||_{q_{n,k}},
		and~B_{0} = p_{\upsilon}.
		\end{align}
		Recall that $p_{\upsilon}$ defines (among other things) the possible values of $n$; see Online Appendix \ref{app:domN} for details. To simplify the notation, most of the time we use: $S_{k}(f)$, $a_{k}$ and $b_{k}$   to denote $S(f,\mathcal{T}_{k}) $, $a_{n,k}(f,\mathcal{T}_{k})$ and $b_{n,k}(f,\mathcal{T}_{k})$ resp. We also use $q_{k}$ instead of $q_{n,k}$.

		Set
	\begin{align*}
	\Gamma_{1} =& \left(  \mathbb{E}  \mathbb{C}_{0} + 4 B_{0} + 2^{3/2} \right) \frac{\gamma_{n}(A)}{\sqrt{n}} \equiv \mathbb{E}_{1} \frac{\gamma_{n}(A)}{\sqrt{n}}  \\
	\Gamma_{2} =& \left( \mathbb{E} \left( \frac{2}{e^{15}} \right)^{2} \mathbb{C}_{0} + 4 B_{0} + 2^{3/2} \right) \frac{\gamma_{n}(A)}{\sqrt{n}}  \equiv \mathbb{E}_{2} \frac{\gamma_{n}(A)}{\sqrt{n}}\\
	\Gamma_{3} =& 4 \frac{\gamma_{n}(A)}{\sqrt{n}}  \equiv \mathbb{E}_{3} \frac{\gamma_{n}(A)}{\sqrt{n}}\\
	with~ \mathbb{C}_{0} = & \sum_{j=1}^{\infty} (2/e^{33})^{2^{j-1}}~and~\mathbb{E}= \frac{8}{3}   \times B_{0}, 
	\end{align*}
		 in the Lemma \ref{lem:decom-L}. In order to proof the Theorem, we claim that it suffices to show that: For all $u \geq 33$, \footnote{The logic behind $33$ is as follows. It holds that $\mathbb{G}_{0} > 3(\mathbb{E}_{1} + \mathbb{E}_{2} + \mathbb{E}_{3})$, and thus $\mathbb{G}_{0} > 3(8 B_{0} + 2^{1+3/2} + 2^{2}) > 3 \times 2^{3}(1 + 2^{-0.5} + 2^{-1}) > 3 \times 2^{3.5} > 33$. Since we only care about $u \geq \mathbb{G}_{0}$ and also $\mathbb{G}_{0} > 33$, it suffices to show the propositions below for $u > 33$.}
         \begin{align*}
			\mathbf{P} \left( \sup_{f \in \mathcal{F}} |\mathbf{L}_{n}(\sum_{k=1}^{\infty}g^{l}_{k}(f)) |\geq u \Gamma_{l} \right) \leq 1/u,~for~l=1,2
		\end{align*}
		and 
		\begin{align*}
		\mathbf{P} \left( \sup_{f \in \mathcal{F}} |\mathbf{L}_{n}(\Xi_{0}(f)1_{\{ m(f) = 0 \}}) |\geq u \Gamma_{3} \right) \leq 1/u.
		\end{align*}		
		Because, if this is the case, then by Lemma \ref{lem:decom-L}, it follows that 
		\begin{align*}
			\mathbf{P}\left( \sup_{f \in \mathcal{F}} |\mathbf{L}_{n}(f) |\geq u \left( \mathbb{E}_{1} + \mathbb{E}_{2} + \mathbb{E}_{3} \right)\frac{\gamma_{n}(A)}{\sqrt{n}}   \right) \leq 3/u,
		\end{align*}
		which implies, for any $u \geq 33$ (which implies that it holds for any $u \geq 0.5\mathbb{G}_{0}$),
				\begin{align*}
				\mathbf{P}\left( \sup_{f \in \mathcal{F}} |\mathbf{L}_{n}(f) |\geq u \frac{\gamma_{n}(A)}{\sqrt{n}}   \right) \leq \left( \mathbb{E}_{1} + \mathbb{E}_{2} + \mathbb{E}_{3} \right)\frac{3}{u} = g_{0}(u).
				\end{align*}
		
	After tedious algebra and noting that $B_{0} = p_{\upsilon}$ and $\mathbb{C}_{0} < 0.00001$, it follows that \begin{align}\label{eqn:const-main}
		\mathbb{E}_{1} + \mathbb{E}_{2} + \mathbb{E}_{3} < p_{\upsilon} \left( 8.1 \right) + \sqrt{2} 8
		\end{align} 
		and thus $\mathbb{G}_{0} > 3(\mathbb{E}_{1} + \mathbb{E}_{2} + \mathbb{E}_{3})$.

		Therefore, in order to establish the desired result, we need to prove the following propositions 
							\begin{proposition}\label{pro:Gama1}
								For all $n$,
							$\mathbf{P} \left( \sup_{f \in \mathcal{F}} |\mathbf{L}_{n}(\sum_{k=1}^{\infty} g^{1}_{k}(f)) |\geq u \Gamma_{1} \right) \leq \frac{1}{u}$, for all $u \geq 33$.
							\end{proposition}
							
						\begin{proposition}\label{pro:Gama2}
							For all $n$,
							$\mathbf{P} \left( \sup_{f \in \mathcal{F}} |\mathbf{L}_{n}(\sum_{k=1}^{\infty} g^{2}_{k}(f)) |\geq u \Gamma_{2} \right) \leq \frac{1}{u}$, for all $u \geq 33$.
						\end{proposition}
											
						\begin{proposition}\label{pro:Gama3}
							For all $n$,
							$\mathbf{P} \left( \sup_{f \in \mathcal{F}} |\mathbf{L}_{n}(\Xi_{0}(f)1_{\{ m(f) = 0 \}}) |\geq u \Gamma_{3} \right) \leq \frac{1}{u}$, for all $u \geq 33$.	
						\end{proposition}

					\begin{remark}\label{rem:constants}
						The constants in $(\Gamma_{l})_{l=1,2,3}$ essentially arise from Bernstein inequality (see the proof of Lemma \ref{lem:bere}) and Lemma \ref{lem:L1-bdd}.  $\triangle$
					\end{remark}

\subsubsection{Proofs of Propositions \ref{pro:Gama1}, \ref{pro:Gama2} and \ref{pro:Gama3}}
\label{sub:gbrack}

We first establish the proof of Proposition \ref{pro:Gama2}. Recall that  $z \mapsto g_{k}^{2}(f)(z) \equiv \Xi_{k-1} (f) (z)  1_{\{m(f,z) = k \cap |v_{k} (f)(z)| > a_{k}\}} $. The proof proceeds in several steps: $g_{k}^{2}(f)$ is majorized by $h^{2}_{k}(f) = S_{k-1}(f)1_{\{ m(f) = k \cap |v_{k}(f)| > a_{k}  \}}$ which requires some re-centering by the expectation. We then use Lemmas \ref{lem:beta-bdd} (see below) and Lemma \ref{lem:bere} (see below) to bound the centered part (this is done in Step 1) and Lemma  \ref{lem:L1-bdd} (see below) to bound the expectation part (this is done in Step 2). The proofs of Propositions \ref{pro:Gama1} and \ref{pro:Gama3} are slight modifications of this one. For the former, we essentially use the methodology in Step 1 of the proof of Proposition \ref{pro:Gama2} and for the latter we essentially use the methodology in Step 2 of the same proof.


The next lemma is a simple adaptation of the results in \cite{DMR1995} p. 408.

\begin{lemma}
	\label{lem:beta-bdd}
	Let $g \in B \equiv \{ f \in \mathcal{F} \colon ||f||_{L^{\infty}} \leq a  \}$ for some $a > 0$. Then, for all $n$ 
	\begin{align*}
	E_{\mathbb{P}} \left[ \sup_{g \in B}  \sqrt{n} |\mathbf{L}_{n}(g) - \mathbf{L}^{\ast}_{n}(g)|          \right] \leq 2 \sqrt{n} a \beta(q).
	\end{align*}
\end{lemma}

\begin{proof}
	See the Online Appendix \ref{app:supp-gbrack}.
\end{proof}

The next lemma establishes some useful properties of $||.||_{q}$ and $\mu_{q}$. Recall that $u \mapsto \mu_{q} = \sum_{i=0}^{q} 1\{ u  \leq 0.5 \beta(i) \}$ and $(u,f) \mapsto Q_{f}(u) = \inf \{ s \mid \mathbf{P}(|f(Z)| > s) \leq u \}$.  The function $\beta^{-1} : [0,1] \rightarrow \mathbb{R}_{+}$ is defined as $\beta^{-1}(u) = \min \{ s \mid \beta(s) \leq u   \}$ --- since $\beta$ is right-continuous, the min exists. 

\begin{lemma}
	\label{lem:q-norm}
	The following assertions hold:			\\		
	(1) For all $q \in \mathbb{N}$ and $f,g \in \mathcal{F}$ such that $|f| \leq |g|$, then $||f||_{q} \leq ||g||_{q}$. \\
	(2)  For all $u \in [0,1]$, $q \mapsto \mu_{q}(u)$ is non-increasing.\\
	(3)  For all $u \in [0,1]$ and $q \in \mathbb{N}$, $\mu_{q}(u) \leq \min \{ \beta^{-1}(2u) +1, q+1  \}$ and for all $ a \geq 0$, $\int_{0}^{1} (\mu_{q}(u))^{a} du \leq 0.5 \left(\int_{\beta(q)}^{1} (\beta^{-1}(v)+1)^{a}  dv + (q+1)^{a} \beta(q) \right)$.\\
	(4) For all $u \in [0,1]$ and $q \in \mathbb{N}$, $\mu_{q}(u) \geq \min \{ \beta^{-1}(2u) , q+1  \}$ and for all $ a \geq 0$, $\int_{0}^{1} (\mu_{q}(u))^{a} du \geq 0.5 \left(\int_{\beta(q+1)}^{1} (\beta^{-1}(v))^{a}  dv + (q+1)^{a} \beta(q+1) \right)$.
\end{lemma}

\begin{proof}
	See the Online Appendix \ref{app:supp-gbrack}.
\end{proof}

The next lemma controls the $L^{1}$ norm for $f 1_{\{ f > b/q_{n,k}    \}}$.
							
\begin{lemma}
	\label{lem:L1-bdd}
	For all $n$, $ k $ and $f \in \mathcal{F}$, let $b = B_{0} \sqrt{n} ||f||_{q_{n,k}} /\sqrt{2^{k+1}}$ and $(q_{n,k})_{n,k}$ as defined in equation \ref{eqn:qk}. Then 
	\begin{align*}
	|| f 1_{\{ f > b/q_{n,k}    \}} ||_{L^{1}} \leq  \sqrt{2} \frac{||f ||_{q_{n,k}}}{\sqrt{n}} \sqrt{2^{k+1}}.
	\end{align*}
\end{lemma}

\begin{proof}
	See the Online  Appendix \ref{app:supp-gbrack}.
\end{proof}

		The following Bernstein exponential inequality result is needed to bound ``the bounded part" of $\mathbf{L}_{n}(g)$, for any $g \in \mathcal{F}$. 
		
		\begin{lemma}
			\label{lem:bere}
			For any $(n,k)\in \mathbb{N}^{2}$, any $q \leq n$ and any $\mathbf{b}>0$ and $g \in \mathcal{F}$ such that $||g||_{L^{\infty}} \leq b_{k-1}$ with $b_{k-1} = B_{0} q^{-1} \sqrt{n} \mathbf{b}/\sqrt{2^{k}}$ and $\mathbf{b} \geq ||g||_{q}$, it follows that 
			\begin{align*}
			\mathbf{P}^{\ast} \left( \sqrt{n} |\mathbf{L}_{n}(g)| \geq u \sqrt{2^{k}} \mathbf{b} \mathbb{E}  \right) \leq 2 \exp \{ - u 2^{k}    \}
			\end{align*}			
			for any $u \geq 1$ and with $\mathbb{E} = \frac{8}{3} \times B_{0} $.
		\end{lemma}
		
		\begin{proof}
			See the Online Appendix  \ref{app:supp-gbrack}.
		\end{proof}

				\begin{proof}[Proof of Proposition \ref{pro:Gama2}]
				

					By the Markov inequality 
					\begin{align}\label{eqn:g2-0}
					\mathbf{P} \left( \sup_{f \in \mathcal{F}} |\mathbf{L}_{n}(\sum_{k=1}^{\infty} g^{2}_{k}(f)) |\geq u \Gamma_{2} \right) \leq \frac{E_{\mathbf{P}} \left[ \sup_{f \in \mathcal{F}} |\mathbf{L}_{n}(\sum_{k=1}^{\infty} g^{2}_{k}(f)) |  \right]}{u\Gamma_{2}} .
					\end{align}
					
					Thus is sufficient to show that $E_{\mathbf{P}} \left[ \sup_{f \in \mathcal{F}(A)} |\mathbf{L}_{n}(\sum_{k=1}^{\infty} g^{2}_{k}(f)) |  \right] \leq \Gamma_{2}$.	Let $h^{2}_{k}(f) = S_{k-1}(f)1_{\{ m(f) = k \cap |v_{k}(f)| > a_{k}  \}}$. By display \ref{eqn:bdd-Xik}, $|\Xi_{k-1}(f)| \leq S_{k-1}(f)$ and thus $|g^{2}_{k}(f)| \leq h^{2}_{k}(f)$ and moreover,
					\begin{align}\notag
					\sup_{f \in \mathcal{F}(A)} | \sum_{k=1}^{\infty} \mathbf{L}_{n}(g^{2}_{k}(f)) | \leq & \sup_{f \in \mathcal{F}(A)} | \sum_{k=1}^{\infty} n^{-1} \sum_{i=1}^{n} (g^{2}_{k}(f)(Z_{i})) |  + \sup_{f \in \mathcal{F}(A)} | \sum_{k=1}^{\infty} E_{\mathbf{P}}[(g^{2}_{k}(f)(Z))] |\\\notag
					\leq & \sup_{f \in \mathcal{F}(A)} | \sum_{k=1}^{\infty} n^{-1} \sum_{i=1}^{n} (h^{2}_{k}(f)(Z_{i})) | + \sup_{f \in \mathcal{F}(A)} | \sum_{k=1}^{\infty} E_{\mathbf{P}}[(h^{2}_{k}(f)(Z))] |\\ \label{eqn:g2-1}
					\leq &  \sup_{f \in \mathcal{F}(A)} |\sum_{k=1}^{\infty} \mathbf{L}_{n}(h^{2}_{k}(f))|   + 2\sup_{f \in \mathcal{F}(A)} \sum_{k=1}^{\infty} | E_{\mathbf{P}}[h^{2}_{k}(f)(Z)]| .
					\end{align}
					
					In the following steps we provide an upper bound for the expectation of the two terms in the RHS. 
					
					\medskip 					
					
					\textsc{Step 1.} We show that \begin{align*}
					E_{\mathbf{P}} \left[ \sup_{f \in \mathcal{F}(A)} |\sum_{k=1}^{\infty} \mathbf{L}_{n}(h^{2}_{k}(f))|   \right] = & E_{\mathbf{P}} \left[ \sup_{f \in \mathcal{F}(A)} |\sum_{k=1}^{\infty} \mathbf{L}_{n}(S_{k-1}(f)1_{\{ m(f) = k \cap |v_{k}(f)| > a_{k}  \}})| \right] \\
					\leq & (\mathbb{C}_{0} \mathbb{E}  +  4 B_{0} ) \mathbb{C}_{0} \frac{\gamma_{n}(A)}{\sqrt{n}}.
					\end{align*}

					Since for any $z$, such that $m(f)(z) = k$, it follows that $|v_{k-1}(f)| = S_{k-1}(f) \leq a_{k-1}$ and thus $||S_{k-1}(f)1_{\{ m(f) = k \cap |v_{k}(f)| > a_{k}  \}}||_{L^\infty} \leq a_{k-1} = B_{0} q^{-1}_{k-1}\sqrt{n} ||S_{k-1}(f)||_{q_{k-1}} /\sqrt{2^{k}}$ and\\ $||S_{k-1}(f)1_{\{ m(f) = k \cap |v_{k}(f)| > a_{k}  \}}||_{q_{k-1}} \leq ||S_{k-1}(f)||_{q_{k-1}}$ (by Lemma \ref{lem:q-norm}(1)).

					By Lemma \ref{lem:beta-bdd} (with $h^{2}_{k}(f)$ and $a_{k-1}$ playing the role of $g$ and $a$ in the lemma)
					\begin{align*}
					E_{\mathbf{P}} \left[ \sup_{f \in \mathcal{F}(A)}  \sqrt{n} |\mathbf{L}_{n}(\sum_{k=1}^{\infty}h^{2}_{k}(f))|          \right] \leq & E_{\mathbf{P}^{\ast}} \left[ \sup_{f \in \mathcal{F}(A)}  \sqrt{n} |\mathbf{L}^{\ast}_{n}(\sum_{k=1}^{\infty} h^{2}_{k}(f))|          \right] \\
					& + E_{\mathbb{P}} \left[ \sup_{f \in \mathcal{F}(A)}  \sqrt{n} |\mathbf{L}^{\ast}_{n}(\sum_{k=1}^{\infty} h^{2}_{k}(f)) - \mathbf{L}_{n}(\sum_{k=1}^{\infty} h^{2}_{k}(f))|          \right]\\
					\leq & E_{\mathbf{P}^{\ast}} \left[ \sup_{f \in \mathcal{F}(A)}  \sqrt{n} |\mathbf{L}^{\ast}_{n}(\sum_{k=1}^{\infty} h^{2}_{k}(f))|          \right] + 2 \sum_{k=1}^{\infty} a_{k-1} \beta(q_{k-1}) \sqrt{n}.
					\end{align*}

					Moreover, \begin{align*}
					\sum_{k=1}^{\infty} a_{k-1} \beta(q_{k-1}) \sqrt{n} = & \sum_{k=1}^{\infty} \frac{B_{0}  ||S_{k-1}(f)||_{q_{k-1}}}{\sqrt{2^{k}}}  \frac{n\beta(q_{k-1})}{q_{k-1}}\\
					\leq & 2 B_{0} \sum_{k=1}^{\infty} \frac{ ||S_{k-1}(f)||_{q_{k-1}}}{\sqrt{2^{k}}}  2^{k}\\
					= & 2 B_{0}  \sum_{k=1}^{\infty}   \sqrt{2^{k}} ||S_{k-1}(f)||_{q_{k-1}}
					\end{align*}
					where the first line follows from the definition of $a_{k-1}$ and second line follows from the definition of $q_{k}$. Hence\begin{align*}
					E_{\mathbf{P}} \left[ \sup_{f \in \mathcal{F}(A)}  \sqrt{n} |\mathbf{L}_{n}(\sum_{k=1}^{\infty}h^{2}_{k}(f))|          \right] \leq  E_{\mathbf{P}^{\ast}} \left[ \sup_{f \in \mathcal{F}(A)}  \sqrt{n} |\mathbf{L}^{\ast}_{n}(\sum_{k=1}^{\infty} h^{2}_{k}(f))|          \right] + 4 B_{0} \sum_{k=1}^{\infty}   \sqrt{2^{k}} ||S_{k-1}(f)||_{q_{k-1}}.
					\end{align*}
					We now bound the first term in the RHS.

					Recall that $\gamma_{n}(A) = \inf_{(\mathcal{T}_{n})_{n} \in \mathbf{T}} \sup_{f \in \mathcal{F}(A)}  \sum_{j=1}^{\infty} \sqrt{2^{j}} ||S_{j-1}(f)||_{q_{j-1}}   $  and let
					\begin{align*}
					\Omega_{n}(u) = \left\{ \omega^{\ast} \in \Omega \mid \forall f\forall j,~ \sqrt{n} \mathbf{L}^{\ast}_{n}(h^{2}_{j}(f)) \leq u  \mathbb{E} \sqrt{2^{j}} ||S_{j-1}(f)||_{q_{j-1}}       \right\}.
					\end{align*}
					($\mathbb{E}$ is as in Lemma \ref{lem:bere}). Note that $\mathbf{P}^{\ast} \left( \sup_{f \in \mathcal{F}(A)} |\sqrt{n}  \mathbf{L}^{\ast}_{n}(\sum_{k=1}^{\infty} h ^{2}_{k}(f) ) | \geq u \mathbb{E}  \gamma_{n}(\mathcal{F})  \right) \leq \mathbf{P}^{\ast} \left( \Omega \setminus \Omega_{n}(u)  \right) $. And $\mathbf{P}^{\ast} \left( \Omega \setminus \Omega_{n}(u)  \right) = \mathbf{P}^{\ast} \left( \exists f \exists j,~ \sqrt{n} \mathbf{L}^{\ast}_{n}(h^{2}_{j}(f)) \geq u  \mathbb{E} \sqrt{2^{j}} ||S_{j-1}(f)||_{q_{j-1}} \right) $.

					By Lemma \ref{lem:bere} with $q = q_{k-1}$, $g = h^{2}_{k}(f)$, $b_{k-1} = a_{k-1}$ and $\mathbf{b} = ||S_{k-1}(f)||_{q_{k-1}}$  (for these choices is true that $||h^{2}_{k}(f)||_{L^{\infty}} \leq b_{k-1} = a_{k-1}$ and $||h^{2}_{k}(f)||_{q_{k-1}} \leq \mathbf{b} = ||S_{k-1}(f)||_{q_{k-1}}$), we obtain that for all $f$ and all $j$,
					\begin{align*}
					\mathbf{P}^{\ast} \left( \sqrt{n} |\mathbf{L}^{\ast}_{n}(h^{2}_{j}(f))| \geq u \sqrt{2^{j}} ||S_{j-1}(f)||_{q_{j-1}} \mathbb{E}  \right) \leq 2 \exp \{ - u 2^{j}    \}.
					\end{align*}

					We also note that $card \left( \left\{ g = h^{2}_{k}(f),~with~f \in \mathcal{F} (A)   \right\}  \right) = card \left( \left\{ g = S_{k-1}(f),~with~f \in \mathcal{F}(A)    \right\}  \right)$. Since $S_{k-1}(f) = S_{k-1}(f')$ for all $f,f'$ in the same set of $\mathcal{T}_{k-1}$,   $card \left( \left\{ g = S_{k-1}(f),~with~f \in \mathcal{F}(A)    \right\}  \right)  = card (\mathcal{T}_{k-1}) \leq 2^{2^{k-1} }$. 						 
					
					Therefore
					\begin{align*}
					& \mathbf{P}^{\ast} \left( \sup_{f \in \mathcal{F}(A)} |\sqrt{n}  \mathbf{L}^{\ast}_{n}(\sum_{k=1}^{\infty} h ^{2}_{k}(f) ) | \geq u \mathbb{E}  \gamma_{n}(\mathcal{F})  \right) \leq \mathbf{P}^{\ast} \left( \Omega \setminus \Omega_{n}(u)  \right) \\
					\leq & \sum_{j=1}^{\infty} \sum_{l=1}^{card (\mathcal{T}_{j-1})}\mathbf{P}^{\ast} \left( \sqrt{n} |\mathbf{L}^{\ast}_{n}(h^{2}_{j}(f))| \geq u \sqrt{2^{j}} ||S_{j-1}(f)||_{q_{j-1}} \left( \mathbb{E}  \right)  \right) \\
					\leq & 2 \sum_{j=1}^{\infty} card (\mathcal{T}_{j-1})  \exp \left\{ - u 2^{j}      \right\} .
					\end{align*}
					
					Since $u 2^{j} = u (2 + 2^{j-1}) \geq 2u + 33 2^{j-1}$ for $u \geq 33$, it follows that
					\begin{align*}
					\mathbf{P}^{\ast} \left( \sup_{f \in \mathcal{F}} |\sqrt{n}  \mathbf{L}^{\ast}_{n}(\sum_{k=1}^{\infty} h ^{2}_{k}(f) ) | \geq  u \mathbb{E}  \gamma_{n}(\mathcal{F})  \right) \leq &  2  e^{-2u} \sum_{j=1}^{\infty} card (\mathcal{T}_{j-1})  \exp \left\{ - 33 2^{j-1}      \right\} \\
					\leq & 2  e^{-2u} \sum_{j=1}^{\infty} \left( \frac{2}{e^{33}}  \right) ^{2^{j-1}}.
					\end{align*}
					Let $\mathbb{C}_{0} = \sum_{j=1}^{\infty} \left( \frac{2}{e^{33}}  \right) ^{2^{j-1}} < 0.00001$. Hence \begin{align*}
					E_{\mathbf{P}^{\ast}} \left[ \sup_{f \in \mathcal{F}(A)}  \sqrt{n} |\mathbf{L}^{\ast}_{n}(\sum_{k=1}^{\infty} h^{2}_{k}(f))|          \right] = & \int_{0}^{\infty} \mathbf{P}^{\ast} \left( \sup_{f \in \mathcal{F}} |\sqrt{n}  \mathbf{L}^{\ast}_{n}(\sum_{k=1}^{\infty} h ^{2}_{k}(f) ) | \geq u  \right) du \\
					= & \mathbb{E} \gamma_{n}(A) \int_{0}^{\infty} \mathbf{P}^{\ast} \left( \sup_{f \in \mathcal{F}} | \sqrt{n} \mathbf{L}^{\ast}_{n}(\sum_{k=1}^{\infty} h ^{2}_{k}(f) ) | \geq \mathbb{E} \gamma_{n}(\mathcal{F}) t   \right) dt\\
					= & \mathbb{E} \gamma_{n}(A) \mathbb{C}_{0} \int_{0}^{\infty} 2 e^{-2t} dt \\
					= & \mathbb{E} \gamma_{n}(A) \mathbb{C}_{0}.
					\end{align*}

					Therefore, \begin{align*}
					E_{\mathbf{P}} \left[ \sup_{f \in \mathcal{F}(A)}  \sqrt{n} |\mathbf{L}_{n}(\sum_{k=1}^{\infty}h^{2}_{k}(f))| \right] \leq   \mathbb{C}_{0} \mathbb{E} \gamma_{n}(A) + 4 B_{0} \sum_{k=1}^{\infty}   \sqrt{2^{k}} ||S(f,\mathcal{T}_{k-1})||_{q_{k-1}},
					\end{align*}
					since the $(\mathcal{T}_{k-1})_{k}$ in the second term of the RHS is arbitrary, it follows that the RHS is bounded above by $(\mathbb{C}_{0} \mathbb{E} + 4 B_{0} )\gamma_{n}(\mathcal{F})$.\footnote{Formally, it should be $(4 + \epsilon) B_{0}$ for any $\epsilon>0$, but we omit this term.}

					\medskip 					
					
					\textsc{Step 2.} We now show that 
					\begin{align*}
					\sup_{f \in \mathcal{F}(A)} \sum_{k=1}^{\infty} | E_{\mathbf{P}}[h^{2}_{k}(f)(Z)]| = \sup_{f \in \mathcal{F}(A)} \sum_{k=1}^{\infty} | E_{\mathbf{P}}[S_{k-1}(f)1_{\{ m(f) = k \cap |v_{k}(f)| > a_{k}  \}}]|  \leq 2^{3/2} \frac{\gamma_{n}(A)}{\sqrt{n}}.
					\end{align*}

					For any $z$, such that $m(f)(z) = k$, it follows that $|v_{k-1}(f)(z)| \leq a_{k-1}$ and since $|v_{k}(f)| = S_{k}(f)$ for all $k$, in turn it follows that
					\begin{align*}
					q_{k-1} S_{k-1}(f)(z) \leq \frac{b_{k-1}}{b_{k}} b_{k} \leq \frac{b_{k-1}}{b_{k}} q_{k} S_{k}(f)(z) \leq \frac{b_{k-1}}{b_{k}} q_{k-1} S_{k}(f)(z). 
					\end{align*}
					Where the inequalities follow because for any fix $n$, $q_{k-1} \geq q_{k}$, and also because $m(f)(z) = k$ implies $S_{k}(f)(z) > a_{k}$. Therefore, $S_{k-1}(f)1_{\{ m(f) = k \cap |v_{k}(f)| > a_{k}  \}} \leq \frac{b_{k-1}}{b_{k}}  S_{k}(f)1_{\{ m(f) = k \cap S_{k}(f) > a_{k}  \}}$.
					
					By Lemma \ref{lem:L1-bdd} with $b = b_{n,k}(f,\mathcal{T}_{k})$
					\begin{align*}
					E_{\mathbf{P}}[S_{k-1}(f)1_{\{ m(f) = k \cap |v_{k}(f)| > a_{k}  \}}] \leq & \frac{b_{k-1}}{b_{k}}  ||   S_{k}(f)1_{\{ m(f) = k \cap S_{k}(f) > a_{k}  \}} ||_{L^{1}} \\
					\leq & \frac{b_{k-1}}{b_{k}}  ||   S_{k}(f)1_{\{ S_{k}(f) > a_{k}  \}} ||_{L^{1}} \\
					\leq & \frac{b_{k-1}}{b_{k}} \sqrt{2}  \sqrt{\frac{2^{k+1}}{n}} ||S_{k}(f)||_{q_{k}}. 
					\end{align*}
					
					Therefore,
					\begin{align*}
					\sup_{f \in \mathcal{F}(A)} \sum_{k=1}^{\infty} | E_{\mathbf{P}}[S_{k-1}(f)1_{\{ m(f) = k \cap |v_{k}(f)| > a_{k}  \}}]|  \leq  \sqrt{2} \sqrt{\frac{1}{n}} \sup_{f \in \mathcal{F}} \sum_{k=1}^{\infty} \sqrt{2^{k+1}}  \frac{b_{k-1}}{b_{k}}  ||S_{k}(f)||_{q_{k}}.
					\end{align*}
					
					Note that $\frac{b_{k-1}}{b_{k}} = \sqrt{\frac{2^{k+1}}{2^{k}}} \frac{||S_{k-1}(f)||_{q_{k-1}}}{||S_{k}(f)||_{q_{k}}}  $, then $\sum_{k=1}^{\infty} \sqrt{2^{k+1}}  \frac{b_{k-1}}{b_{k}}  ||S_{k}(f)||_{q_{k}} = \sum_{k=1}^{\infty} \frac{2^{k+1}}{\sqrt{2^{k}}} ||S_{k-1}(f)||_{q_{k-1}}$. Finally, since the implied partition defining $S_{k-1}(f)$ is arbitrary, it follows that\begin{align*}
						\sqrt{2}  \sup _{f \in \mathcal{F}(A)}\sum_{k=1}^{\infty} \frac{2^{k+1}}{\sqrt{2^{k}}} ||S_{k-1}(f)||_{q_{k-1}} = 2^{3/2}   \sup _{f \in \mathcal{F}(A)}\sum_{k=1}^{\infty} \sqrt{2^{k}} ||S_{k-1}(f)||_{q_{k-1}} \leq 2^{3/2} \gamma_{n}(A).
					\end{align*} 
					
					\medskip

					\textsc{Step 3.} By display \ref{eqn:g2-0} - \ref{eqn:g2-1}, Step 1 and Step 2 it follows that
					\begin{align*}
					\mathbf{P} \left( \sup_{f \in \mathcal{F}(A)} |\mathbf{L}_{n}(\sum_{k=1}^{\infty} g^{2}_{k}(f)) |\geq u \Gamma_{2} \right) \leq& \frac{ (\mathbb{E}  \mathbb{C}_{0} + 4 B_{0}) \frac{\gamma_{n}(A)}{\sqrt{n}} + 2^{3/2} \frac{\gamma_{n}(A)}{\sqrt{n}} }{u\Gamma_{2}} \\
					= & u ^{-1}\frac{\gamma_{n}(A)}{\sqrt{n}} \frac{ \mathbb{E}  \mathbb{C}_{0} + 4 B_{0} + 2^{3/2}  }{\Gamma_{2}}
					\end{align*}
					and by definition of $\Gamma_{2}$ the RHS equals $u^{-1}$ and the desired result follows. 				
				\end{proof}
				
				\begin{remark}As shown in the previous proof and discussed in Section \ref{sec:Dudley-bound}, a crucial difference between our chaining and, say, the one in \cite{VdV-W1996} Ch. 2.5 is that the latter basically bounds the size of each link in the chain uniformly for all $f$, whereas we allow the size (measured by $||S(f,\mathcal{T}_{l})||_{q_{n,l}}$) to vary with each $f$. $\triangle$
				\end{remark}

\begin{proof}[Proof of Proposition \ref{pro:Gama1}]
	 Recall that $z \mapsto g_{k}^{1}(f)(z) \equiv \Delta_{k} (f) (z) 1_{\{m(f,z) \geq k \cap |v_{k} (f)(z)| \leq a_{k}\}}$. The proof is analogous to that of Proposition \ref{pro:Gama2} so we only present a sketch. By the Markov inequality 
	\begin{align*}
	\mathbf{P} \left( \sup_{f \in \mathcal{F}(A)} |\mathbf{L}_{n}(\sum_{k=1}^{\infty} g^{1}_{k}(f)) |\geq u \Gamma_{1} \right) \leq \frac{E_{\mathbf{P}} \left[ \sup_{f \in \mathcal{F}(A)} |\mathbf{L}_{n}(\sum_{k=1}^{\infty} g^{1}_{k}(f)) |  \right]}{u\Gamma_{1}} .
	\end{align*}

	It follows that $||g^{1}_{k}(f)||_{L^\infty} \leq a_{k} = B_{0} q^{-1}_{n,k}\sqrt{n} ||S(f,\mathcal{T}_{k})||_{q_{n,k}} /\sqrt{2^{k+1}}$ and, by lemma \ref{lem:q-norm}(1), $||g^{1}_{k}(f)||_{q_{n,k}} \leq ||\Delta_{k}(f)||_{q_{n,k}} \leq ||S(f,\mathcal{T}_{k})||_{q_{n,k}}$.	By following the same steps in Step 1 of the proof of proposition \ref{pro:Gama2},  
	\begin{align*}
	E_{\mathbf{P}} \left[ \sup_{f \in \mathcal{F}(A)} |\mathbf{L}_{n}(\sum_{k=1}^{\infty} g^{1}_{k}(f)) |  \right] \leq E_{\mathbf{P}^{\ast}} \left[ \sup_{f \in \mathcal{F}(A)} |\mathbf{L}^{\ast}_{n}(\sum_{k=1}^{\infty} g^{1}_{k}(f)) |  \right] + 2 \sum_{k=1}^{\infty} a_{k} \beta(q_{n,k}) \sqrt{n}.
	\end{align*}
	
	We now bound the first term in the RHS. For this, we first show that\begin{align*}
		\mathbf{P}^{\ast} \left( \sup_{f \in \mathcal{F}(A)} | \mathbf{L}^{\ast}_{n}(\sum_{k=1}^{\infty} g^{1}_{k}(f) ) | \geq u \mathbb{E}  \gamma_{n}(A)  \right) \leq 2 e^{-2u } \left( \frac{2}{e^{15}} \right)^{2}  \mathbb{C}_{0}.
	\end{align*}
	In order to do this we proceed as in Step 1 of the proof of proposition \ref{pro:Gama2}. By invoking Lemma \ref{lem:bere} with $q = q_{n,k}$, $g = g^{1}_{k}(f)$, $b_{k-1} = a_{k}$ and $\mathbf{b} = ||S(f,\mathcal{T}_{k})||_{q_{n,k}}$ it follows, for any $j$ and $f$:
	\begin{align*}
	\mathbf{P}^{\ast} \left(  \sup_{f \in \mathcal{F}(A)} \sqrt{n} |\mathbf{L}^{\ast}_{n}(g^{1}_{j}(f))| \geq u \sqrt{2^{j+1}} ||S(f,\mathcal{T}_{j})||_{q_{n,j}} \left( \mathbb{E}  \right)  \right) \leq \exp\{  - u 2^{j+1} \}.
	\end{align*}
	Therefore, by analogous calculations to those in the proof of Proposition \ref{pro:Gama2}, it follows that
	\begin{align*}
	\mathbf{P}^{\ast} \left( \sup_{f \in \mathcal{F}(A)} | \sqrt{n} \mathbf{L}^{\ast}_{n}(\sum_{k=1}^{\infty} g^{1}_{k}(f) ) | \geq u \mathbb{E} \gamma_{n}(A)  \right) \leq  \sum_{j=1}^{\infty} \sum_{l=1}^{card (\{ g^{1}_{j}(f) : f \in \mathcal{F}(A)  \}) } \exp\{  - u 2^{j+1} \}.
	\end{align*}
	The cardinality of $\{ g^{1}_{j}(f) : f \in \mathcal{F}(A)  \}$ is no larger than $card (\{ \pi_{j}(f) :f \in \mathcal{F}(A)   \}) \times card(\{ \pi_{j-1}(f) :f \in \mathcal{F}(A)   \}) \leq 2^{2^{j}+2^{j-1}}$, thus --- since $u \geq 33 $ and $2^{j} \geq 2$ for all $j \geq 1$, $u 2^{j}   \geq u 2 +  15 \times 2^{j}$ ---
	\begin{align*}
	\mathbf{P}^{\ast} \left( \sup_{f \in \mathcal{F}(A)} | \sqrt{n} \mathbf{L}^{\ast}_{n}(\sum_{k=1}^{\infty} g^{1}_{k}(f) ) | \geq  u \mathbb{E} \gamma_{n}(A)  \right) \leq &  2  e^{-2 u} \sum_{j=1}^{\infty}  2^{2^{j}+2^{j-1}} \exp \left\{ - 15 \times 2^{j+1}      \right\} \\
	\leq  & 2  e^{-2 u} \sum_{j=1}^{\infty} \left( \frac{2}{e^{15}}  \right) ^{2^{j+1}} = 2 e^{-2u} \left( \frac{2}{e^{15}} \right)^{2} \mathbb{C}_{0},
	\end{align*}
	because $2^{j-1}+2^{j} \leq 2^{j+1}$.

	Therefore, by following the steps  in step 1  of the proof of proposition \ref{pro:Gama2}
	\begin{align*}
	E_{\mathbf{P}^{\ast}} \left[ \sup_{f \in \mathcal{F}(A)} |\mathbf{L}^{\ast}_{n}(\sum_{k=1}^{\infty} g^{1}_{k}(f)) |  \right] \leq \left( \frac{2}{e^{15}} \right)^{2} \mathbb{C}_{0} \mathbb{E} \frac{\gamma_{n}(A) }{\sqrt{n}}.
	\end{align*}
	Also
	\begin{align*}
	\sum_{k=1}^{\infty} a_{k} \beta(q_{k}) \sqrt{n} \leq 2 B_{0} \sum_{k=1}^{\infty}   \sqrt{2^{k+1}} ||S(f,\mathcal{T}_{k})||_{q_{k}},
	\end{align*}		
	with $\gamma_{n}(A) = \inf_{(\mathcal{T}_{l})_{l} \in \mathbf{T}} \sup_{f \in \mathcal{F}(A)}  \sum_{j=0}^{\infty} \sqrt{ 2^{j+1}} ||S(f,\mathcal{T}_{j})||_{q_{j}}$. Finally, since the partition implied by the term $ \sum_{k=1}^{\infty}   \sqrt{2^{k+1}} ||S(f,\mathcal{T}_{k})||_{q_{k}}$ is arbitrary, we can bound the term by $ \gamma_{n}(A)$.
	
	Therefore $E_{\mathbf{P}} \left[ \sup_{f \in \mathcal{F}(A)} |\mathbf{L}_{n}(\sum_{k=1}^{\infty} g^{1}_{k}(f)) |  \right] \leq \left( \left( \frac{2}{e^{15}} \right)^{2} \mathbb{C}_{0} \mathbb{E} + 6 B_{0} \right) \frac{\gamma_{n}(A) }{\sqrt{n}}$.
	
\end{proof}

	\begin{proof}[Proof of Proposition \ref{pro:Gama3}]
		By the Markov inequality and analogous calculations to those in the Proof of Proposition \ref{pro:Gama2}, it suffices to show
		\begin{align}\label{eqn:Gamma-3-0}
		E_{\mathbf{P}} \left[  \sup_{f \in \mathcal{F}(A)} |\mathbf{L}_{n}(S_{0}(f)1_{\{ m(f) = 0  \}}) |     \right] \leq  2^{3/2} n^{-1/2} \gamma_{n}(A)
		\end{align}
		and \begin{align}\label{eqn:Gamma-3-1}
		\sup_{f \in \mathcal{F}(A)} E_{\mathbf{P}}[ S_{0}(f)1_{\{ m(f) = 0  \}}  ] \leq 2^{1/2} n^{-1/2} \gamma_{n}(A)
		\end{align}
		
		We now establish equation \ref{eqn:Gamma-3-1}. The fact that $m(f) = 0$ implies that $S_{0}(f) > a_{0} = B_{0}\frac{ \sqrt{n} \delta_{0}/\sqrt{2} }{q_{0}}$. Thus, by Lemma \ref{lem:L1-bdd} with $f$ (in the lemma) equal to $S_{0}(f)$ and noting that $b = b_{n,0}(f,\mathcal{T}_{0})$ satisfies the condition, it follows that 
		\begin{align*}
		E_{\mathbf{P}}[ S_{0}(f)1_{\{ m(f) = 0  \}}  ] =& || S_{0}(f)1_{\{ S_{0}(f) > a_{0}  \}}  ||_{L^{1}} \\
		\leq & \sqrt{2} \sqrt{ \frac{2}{n}   } ||S_{0}(f)||_{q_{0}}.
		\end{align*}
		Clearly, the RHS is bounded above by $ n^{-1/2} \sqrt{2} \gamma_{n}(\mathcal{F}) $.
		
		In order to establish equation \ref{eqn:Gamma-3-0}. We first note that $T(f,\mathcal{T}_{0}) = \mathcal{F}(A)$ and thus $z \mapsto S_{0}(f)(z) = \sup_{f_{1},f_{2} \in T(f,\mathcal{T}_{0}) } |f_{1}(z) - f_{2}(z)| = \sup_{f_{1},f_{2} \in \mathcal{F} } |f_{1}(z) - f_{2}(z)|$ which does not depend on $f$; we denote this function as $z \mapsto \sigma_{0}(z)$. Consequently $\sup_{f \in \mathcal{F}(A)}  |\mathbf{L}_{n}(S_{0}(f)1_{S_{0}(f) > a_{0}})|  =|\mathbf{L}_{n}(\sigma_{0}1_{\sigma_{0} > a_{0}})| $. Thus 
		\begin{align*}
		E_{\mathbf{P}}\left[ \sup_{f \in \mathcal{F}(A)}  |\mathbf{L}_{n}(S_{0}(f)1_{S_{0}(f) > a_{0}})|   \right] \leq E_{\mathbf{P}}\left[ |\mathbf{L}_{n}(\sigma_{0}1_{\sigma_{0} > a_{0}})| \right] \leq 2||\sigma_{0}1_{\sigma_{0} > a_{0}}||_{L^{1}}.
		\end{align*}
		
		By Lemma \ref{lem:L1-bdd} with $f = \sigma_{0}$, $2||\sigma_{0}1_{\sigma_{0} > a_{0}}||_{L^{1}} \leq 2 \sqrt{2} \sqrt{\frac{2}{n}} ||\sigma_{0}||_{q_{0}} = 2^{3/2} \sqrt{\frac{2}{n}} ||S_{0}(f)||_{q_{0}} $. The desired result follows since the last term is clearly bounded by $2^{3/2} n^{-1/2} \gamma_{n}(A) $.
	\end{proof}

 	   	    \section{Proofs of Proposition \ref{pro:bound-GC} and Proposition \ref{pro:M-rate1}}
 	   	    \label{app:Dudley-bound}
 	   	    
  \begin{proof}[Proof of Proposition \ref{pro:bound-GC}]
  	Let 	
  	\begin{align*}
  	\bar{\gamma}(A,\mathbf{d}) = \inf_{(\mathcal{S}_{l})_{l} \in \mathbf{S}} \sup_{\theta \in A} \sum_{l=0}^{\infty} 2^{l/2} Diam(T(\theta,\mathcal{S}_{l}),\mathbf{d}),       	
  	\end{align*}
  	where $\mathbf{S}$ is the set of sequences of partitions of $A$, $(\mathcal{S}_{l})_{l \in \mathbb{N}_{0}}$, such that $\mathcal{S}_{0} = \{ A \}$ and $card(\mathcal{S}_{l}) \leq 2^{2^{l}}$. We show that $\gamma(A,||.||_{L^{r}(\mathbf{P})})  \leq \sqrt{2} ||C_{k,M}||_{L^{r}(\mathbf{P})}  \bar{\gamma}(A,\mathbf{d}) $.
  	
  	Consider the following mapping $\Phi : \Theta \rightarrow \mathcal{F}(\Theta)$ with $\Phi(\theta) = \phi(.,\theta)$. By construction this mapping it 1-to-1 and onto. Let $\theta_{f} = \Phi^{-1}(f)$. Moreover, given any partition $\mathcal{T} = \{ T_{1},...,T_{N} \}$ over $\mathcal{F}(A)$, we can generate an induced partition over $A$, $\bar{\mathcal{T}} = \{ \bar{T}_{1},...,\bar{T}_{N} \}$, such that $\bar{T}_{i} = \Phi^{-1}(T_{i}) = \{ \theta \in A \mid \exists f \in T_{i},~s.t.~\phi(.,\theta) = f \}$ for all $i=1,...,N$.\footnote{Clearly $\bar{T}_{j} \cap \bar{T}_{k} = \{ \emptyset \}$ otherwise there exists a $\theta$ such that $\phi(.,\theta)=f=f'$ for $f\ne f'$, and $\cup_{i=1}^{N} \bar{T}_{i} = A$ because otherwise there exists a $\theta$ such that $\phi(.,\theta) \in \mathcal{F}(A) \setminus \{ f \mid \exists \theta \in \cup_{i=1}^{N} \bar{T}_{i},~s.t.~\phi(.,\theta) = f  \} = \cup_{i=1}^{N} T_{i}$, but this is a contradiction to the fact that $\mathcal{T}$ is a partition.} It also holds that any partition over $A$, $\bar{\mathcal{T}} = \{ \bar{T}_{1},...,\bar{T}_{N} \}$, has an associated partition over $\mathcal{F}(A)$, with sets of the form $\{ f \mid \exists \theta \in \bar{T}_{i},~s.t.~f = \Phi(\theta) - \Phi(\nu_{k}(\mathbf{P}))  \}$.
  	
  	Therefore, under assumption \ref{ass:Lip-phi},  for any $f \in \mathcal{F}(A)$, and $l \in \mathbb{N}$ 
  	\begin{align*}
  	\sup_{f_{1},f_{2} \in T(f,\mathcal{T}_{l})} |f_{1}(z) - f_{2}(z) | \leq & \sup_{\theta_{1},\theta_{2} \in T(\theta_{f},\bar{\mathcal{T}}_{l})} |C_{k,M}(z)| \times  \mathbf{d}(\theta_{1},\theta_{2})\\
  	= & |C_{k,M}(z)| Diam(T(\theta_{f},\bar{\mathcal{T}}_{l}),\mathbf{d}),~a.s.-\mathbf{P}
  	\end{align*} 
  	where $T(\theta_{f},\bar{\mathcal{T}}_{l})$ is the (only) set of the partition $\bar{\mathcal{T}}_{l}$ that contains $\theta_{f}$. Therefore, $||S(f,\mathcal{T}_{l})||_{L^{r}(\mathbf{P})} \leq ||C_{k,M}||_{L^{r}(\mathbf{P})} Diam(T(\theta_{f},\bar{\mathcal{T}}_{l}),\mathbf{d})$.
  	
  	It thus follows that for any admissible sequence $(\mathcal{T}_{l})_{l}$ over $\mathcal{F}(A)$, \begin{align*}
  	\sum_{l=0}^{\infty} 2^{l/2} ||S(f,\mathcal{T}_{l})||_{L^{r}(\mathbf{P})}  \leq ||C_{k,M}||_{L^{r}(\mathbf{P})} \sum_{l=0}^{\infty} 2^{l/2} Diam(T(\theta_{f},\bar{\mathcal{T}}_{l}),\mathbf{d}).
  	\end{align*}
  	Taking the supremum over $f$ and the infimum over partitions, the desired result follows.
  \end{proof} 	   	    
 	   	    
\begin{proof}[Proof of Proposition \ref{pro:M-rate1}]	
  Let
	  	   	\begin{align*}
	  	   	\theta \mapsto X_{\theta} \equiv \sum_{j=1}^{k} \zeta_{j}(\theta)
	  	   	\end{align*}
	  	   	with $\zeta_{j}(\theta) = b_{j} \zeta_{j} \theta_{j}$. Observe that $\sqrt{E[(X_{\theta_{1}}-X_{\theta_{2}})^{2}]} = \sqrt{\sum_{j=1}^{k} b_{j} \left( \theta_{1,j} - \theta_{2,j} \right)^{2}} =  ||\theta_{1} - \theta_{2} ||_{\ell^{2}(b)}$. Thus, by \cite{talagrand2014} Theorem 2.4.1, \begin{align*}
	  	   	\bar{\gamma}(A,||\cdot ||_{\ell^{2}(b)} ) \leq  L  \Gamma_{k}(A),
	  	   	\end{align*}
	  	   	and by Proposition \ref{pro:bound-GC}, $\gamma(A,||.||_{L^{r}(\mathbf{P})})  \leq \sqrt{2} L ||C_{k,M}||_{L^{r}(\mathbf{P})}   \Gamma_{k}(A)$.	  	   	\end{proof}

	  	\section{Proofs for Results in Section \ref{sec:examples}}
	  	\label{app:examples}

	  		In order to show this proposition we need the following results; The proofs of all supplementary lemmas are relegated to the Online Appendix \ref{sup:examples}.

	  			  		\begin{lemma}
	  			  			\label{lem:Q-LASSO-1}
	  			  			The function $E_{\mathbf{P}}[\phi(Z,\cdot)]$ is convex, continuous and twice-continuously differentiable, and coercive.\footnote{We say a function $f$ is coercive if for all $M \in \mathbb{R} $, $\{ \theta \in \Theta \mid f(\theta) \leq M  \}$ is bounded.} 
	  			  		\end{lemma}
	  			  		
	  			  		\begin{proof}
	  			  			See the Online Appendix \ref{sup:examples}.
	  			  		\end{proof}
	  		
\begin{remark}
	The boundedness assumption of the eigenvalues of $E[XX^{T}]$ is to ensure the finiteness of constants affecting the concentration rate. The restriction on the minimal eigenvalue is needed to establish coerciveness. If $\Theta$ is compact, then this restriction is not needed. $\triangle$
\end{remark}

	  		\begin{lemma}
	  			\label{lem:LASSO-d}
	  			Assumption \ref{ass:Lip-phi} holds with $\mathbf{d}$ induced by $||.||_{\ell^{2}(b)}$ with $b_{j} = 1$ for all $j$, $r = \pi_{0}$, and $\mathbb{C}_{k,M}(z) = (1+\tau) e_{max}( xx^{T}  )$ for any $k \in \mathbb{N}$, $z$ and $M>0$.
	  			
	  		\end{lemma}

	  		The  following lemma is a simple application of the Mean value theorem and results \cite{Rockafellar70} Ch. 23 (summarized in Lemma \ref{lem:sub-diff} in the Online Appendix \ref{sup:examples}).
	  		
	  		\begin{lemma}
	  			\label{lem:LASSO-delta-bdd}
	  			Suppose that $\theta \mapsto Q(\theta,\mathbf{P})$ is convex and twice continuously differentiable and $Pen$ is convex, $\nu_{k}(\mathbf{P}) \subset int(\Theta_{k})$ and $\Theta_{k} \subseteq \mathbb{R}^{k}$ convex. For all $\theta \in \Theta_{k}$ and $\theta_{0,k} \in \nu_{k}(\mathbf{P})$, 
	  			\begin{align*}
	  			Q_{k}(\theta,\mathbf{P}) - Q_{k}(\nu_{k}(\mathbf{P}),\mathbf{P}) \geq & 0.5 \int_{0}^{1} \frac{d^{2} Q(\theta_{0,k} + t(\theta - \theta_{0,k}),\mathbf{P})}{d\theta^{T} d \theta }[\theta - \theta_{0,k},\theta - \theta_{0,k}] dt\\
	  			& + \lambda_{k} (Pen(\theta) - Pen(\theta_{0,k})- \mu^{T}(\theta - \theta_{0,k}))
	  			\end{align*}
	  			for some $\mu \in \partial{(Pen(\theta_{0,k}))}$ such that $\frac{d Q(\theta_{0,k},\mathbf{P})}{d\theta^{T}}[\theta - \theta_{0,k}] + \lambda_{k} \mu^{T}(\theta - \theta_{0,k}) = 0$.
	  		\end{lemma}
	  		
	  		\begin{proof}
	  			See the Online Appendix \ref{sup:examples}.
	  		\end{proof}

	  		\begin{remark}
	  			(1) Lemma \ref{lem:LASSO-delta-bdd} necessarily implies that for any $\theta' \in \nu_{k}(\mathbf{P})$,  $\frac{d^{2} Q(\theta_{0,k} + t(\theta' - \theta_{0,k}),\mathbf{P})}{d\theta^{T} d \theta }[\theta' - \theta_{0,k},\theta' - \theta_{0,k}] = 0$ for any $t \in [0,1]$ and $Pen(\theta') - Pen(\theta_{0,k})- \mu^{T}(\theta' - \theta_{0,k})=0$ for all $\mu \in \partial{(Pen(\theta_{0,k}))}$ such that $\frac{d Q(\theta_{0,k},\mathbf{P})}{d\theta^{T}}[\theta' - \theta_{0,k}] + \lambda_{k} \mu^{T}(\theta' - \theta_{0,k}) = 0$.
	  			
	  			(2) The assumption $\nu_{k}(\mathbf{P}) \subset int(\Theta_{k})$ is to avoid dealing with boundary issues; it can be relaxed following the standard results. $\triangle$
	  		\end{remark}

	  					           Recall that $I_{n,k}(\omega)(r) \equiv \{ \theta \in \Theta_{k}(M_{n,k}) : r \geq \delta_{k,\mathbf{P}}(\theta,\nu_{k}(\mathbf{P})) \geq 0.5 r   \}$. Let $J_{n,k}(\omega)(r) \equiv \{ \theta \in \Theta_{k}(M_{n,k}) : r \geq \delta_{k,\mathbf{P}}(\theta,\nu_{k}(\mathbf{P}))  \}$. Clearly, $I_{n,k}(\omega)(r) \subseteq J_{n,k}(\omega)(r)$ and $J_{n,k}(\omega)(.)$ is non-decreasing. 
	  					           
	  					           Let $S_{I}(C)\equiv \{ r > 0 \mid r \geq C \max_{x \geq 1} \frac{\Gamma_{k}(I_{n,k}(\omega)(rx)) }{rx}     \}$ and $S_{J}(C)\equiv \{ r > 0 \mid r \geq C \max_{x \geq 1} \frac{\Gamma_{k}(J_{n,k}(\omega)(rx)) }{rx}     \}$, and $r_{I}(C) = \inf_{r \in S_{I}(C)} r$ and $r_{J}(C) = \inf_{r \in S_{J}(C)} r$. 
	  					           
	  					           \begin{lemma}
	  					           	\label{lem:r-fact}
	  					           	The following are true: 
	  					           	
	  					           	(1) $r_{I}(.)$ and $r_{J}(.)$ are non-decreasing.
	  					           	
	  					           	(2) $r_{I} \leq r_{J}$.

	  					           	(3) $r_{J}(C) \leq C r_{J}(1)$ and $r_{I}(C) \leq C r_{I}(1)$ for any $C\geq 1$.
	  					           	
	  					           	(4) For any $(n,k) \in \mathbb{N}^{2}$, if for all $r>0$, $\Gamma_{k}(J_{n,k}(\omega)(r)) \leq \min\{ r \Lambda_{n,k}(\omega), B_{n,k}(\omega)\}$ a.s.-$\mathbf{P}$, for some $\Lambda_{n,k}(\omega)$ and $B_{n,k}(\omega)$. Then $r_{J}(C) \leq \min \{ C \Lambda_{n,k}(\omega) , \sqrt{ C B_{n,k}(\omega)}    \}$.
	  					           \end{lemma} 
	  					           
	  					           \begin{proof}
	  					           	See the Online Appendix \ref{sup:examples}.
	  					           \end{proof}

 			Recall that $M_{n,k} = E_{P_{n}}[\phi(Z,\theta_{\ast})] + \lambda_{k} ||\theta_{\ast}||_{\ell^{1}}$;  and $W_{k} \equiv E[\underline{d}_{k}(X) XX^{T}]$ with $ \underline{d}_{k}(X) \equiv \inf_{\theta \in \{\theta \colon ||\theta||_{\ell^{1}} \leq \max\{M_{n,k},E[M_{n,k}]\}/\lambda_{k} \}} f(X^{T}\theta \mid X) $.
 			
 			\begin{proposition}\label{pro:HD-QR-L1-new}
 				For any $(k,n)$ and any $u >0$, with probability higher than $1-\mathbb{G}_{0}/u$:
 				\begin{align}
 				||W^{1/2}_{k}(\nu_{k}(P_{n})  - \theta_{\ast})   ||_{\ell^{2}} \leq & u \mathbb{K} \max\{ 1, 2^{1/\pi_{0}} \mathbb{E}_{\pi_{0}}\}  \\ \notag
 				& \times \left(  \min \left\{ \sqrt{\frac{tr\{ W^{-1}_{k} \}}{n(\beta)}}  ,  \left( \frac{\log (2 d)} {n(\beta)}\right)^{1/4} \sqrt{M_{n,k}/\lambda_{k}}  \right\}  + \sqrt{\lambda_{k}||\theta_{\ast}||_{\ell^{1}}}    \right),
 				\end{align}
 				where $\mathbb{K} \equiv \max\{ 1, \sqrt{2} 10 L(1+\tau) \}$.
 			\end{proposition}

 			\begin{proof}[Proof of Proposition \ref{pro:HD-QR-L1-new}]
 				Throughout the proof, let $||\cdot||_{2} \equiv	\sqrt{E_{\mathbf{P}}[ \underline{d}_{k}(X) || X^{T}(\cdot)   ||^{2}_{\ell^{2}}]}$. Note that this notion of distance is defined conditional on the data $\omega$ because $\underline{d}_{k}(x)$ depends on $\omega$ through $M_{n,k}$; throughout this dependence is left implicit. Also, let $\mathbb{K}_{\pi_{0}} \equiv 2^{1/2 + 1/\pi_{0}} 5 L(1+\tau) ||e_{max}(XX^{T})||_{L^{\pi_{0}}} $.
 				
 				Lemma \ref{lem:Q-LASSO-1} implies assumption \ref{ass:Q-lsco} under the Euclidean topology, thus, by Lemma \ref{lem:vP-exist}, $\nu(\mathbf{P})$ is non-empty, given by $\nu(\mathbf{P}) = \{ \theta \in \Theta \mid E[X (F(X^{T}\theta|X)- \tau) ] = 0 \}$. Also, by Lemma \ref{lem:reg-min}, $\nu_{k}(\mathbf{P})$ is  non-empty and also convex.

 				By Lemma \ref{lem:LASSO-d}, $\theta \mapsto \phi(z,\theta)$ satisfies assumption \ref{ass:Lip-phi} with $r = \pi_{0}$ and $\mathbb{C}_{k,M}(z) \equiv (1+\tau) e_{max}(xx^{T})$. This fact, Theorem \ref{thm:effe-n} (with $\eta_{n} = 0$) and Proposition \ref{pro:M-rate1}, imply 
 				\begin{align*}
 				V_{n,k} \leq \min \left\{ s > 0 \mid s \geq \frac{\mathbb{K}_{\pi_{0}}}{\sqrt{n(\beta)}} \max_{x \geq 1} \frac{\Gamma_{k}( I_{n,k}(\omega)(sx)  )}{sx}    \right\}.
 				\end{align*}

 				So we now bound $\Gamma_{k}( I_{n,k}(\omega)(sx)  )$. Let  $A_{1} \equiv \{ \tau \in \Theta \colon   || \tau  ||_{\ell^{1}} \leq M_{n,k}/\lambda_{k} \}$ and $A_{0} \equiv \{ \tau \in \Theta \colon   || \tau  ||_{\ell^{1}} \leq Q_{k}(\theta_{\ast},\mathbf{P})/\lambda_{k} \}$.

 				By Lemma \ref{lem:LASSO-delta-bdd} with $\theta_{0,k}$ being the minimum $||.||_{\ell^{1}}$-element of $\nu_{k}(\mathbf{P})$, and the fact that $\frac{d^{2}Q(\theta,\mathbf{P})}{d\theta^{2}} = E[f(X^{T}\theta|X) XX^{T}]$, it follows that for any $\theta \in A_{1} \cup A_{0}$,   $\delta^{2}_{k,\mathbf{P}}(\theta , \nu_{k}(\mathbf{P}))$ is bounded below by $0.5(\theta - \nu_{k}(\mathbf{P}))^{T}\int_{0}^{1} E[f(X^{T}\theta(t)|X) XX^{T}] dt (\theta - \theta_{0,k}) $ with $\theta(t) = \theta_{0,k} + t (\theta - \theta_{0,k})$. Moreover, for any $v \in \Theta$, 
 				\begin{align*}
 				v^{T} \int_{0}^{1} E[f(X^{T}\theta(t)|X) XX^{T}] dt v = &   \int_{0}^{1} E[f(X^{T}\theta(t)|X) v^{T} XX^{T} v] dt \\
 				\geq &   E[\inf_{\theta \in A_{0} \cup A_{1}}  f(X^{T}\theta|X) v^{T} XX^{T} v] \\
 				= & v^{T} E[ \underline{d}_{k}(X)  XX^{T} ] v
 				\end{align*}
 				where the second line because  $\theta,\nu_{k}(\mathbf{P})$ are in $A_{1} \cup A_{0}$ and thus, so is $\theta(t)$. Finally, note that $ A_{1} \cup A_{0} \equiv \{ \theta \in \Theta \colon ||\theta||_{\ell^{1}} \leq \max\{ E_{P_{n}}[\phi(Z,\theta^{\ast})], E_{\mathbf{P}}[\phi(Z,\theta^{\ast})]  \}/\lambda_{k} + ||\theta^{\ast}||_{\ell^{1}}  \}$. Thus
 				\begin{align*}
 				\delta_{k,\mathbf{P}}(\theta,\nu_{k}(\mathbf{P})) \geq \sqrt{0.5} ||\theta - \theta_{0,k}||_{2}.
 				\end{align*}	
 				(this implies that $\nu_{k}(\mathbf{P})$ is in fact a singleton).
 							
 				Therefore, for any $s>0$, $I_{n,k}(\omega)(s) \subseteq A(s) \equiv \{ \theta \in \Theta \colon \sqrt{0.5} ||\theta - \nu_{k}(\mathbf{P})||_{2} \leq s   \}$. This imply that $\Gamma_{k}(I_{n,k}(\omega)(s)) \leq 	E \left[ \sup_{\tau \in A(s)} | \sum_{k=1}^{d} \zeta_{k} \theta_{k}  |   \right] $ (see display \ref{eqn:Gauss} for the definition of $\Gamma_{k}$). Moreover, by H\"{o}lder inequality 
 				\begin{align*}
 				E \left[ \sup_{\tau \in A(s)} | \sum_{k=1}^{d} \zeta_{k} \theta_{k} |   \right] \leq  \sqrt{E \left[\zeta^{T} W^{-1}_{k}\zeta \right] }  \sup_{\tau \in A(s)} \sqrt{\tau^{T} W_{k} \tau  } 	\leq    \sqrt{tr\{ W_{k}^{-1}\} }  \sqrt{2} s
 				\end{align*}
 				where the second line follows because (i) $ \sqrt{\tau^{T} W_{k} \tau  } = ||\tau||_{2}$ and (ii) $\zeta$ are independent with zero mean. Therefore, $\Gamma_{k}(I_{n,k}(\omega)(s)) \leq  \sqrt{tr\{ W_{k}^{-1}\} } \sqrt{2} s$. 
 				
 				On the other hand,
 				\begin{align*}
 				E \left[ \sup_{\tau \in A_{1}} |\sum_{k=1}^{d} \zeta_{k} \theta_{k} |   \right] \leq & E[||\zeta ||_{\ell^{\infty}}] \sup_{\tau \in A_{1}} ||\tau||_{\ell^{1}} \\
 				\leq & \sqrt{ \log (2 d)   } M_{n,k}/\lambda_{k}
 				\end{align*}
 				where the second line follows from Lemma \ref{lem:M-rate1b}. Hence, $\Gamma_{k}(I_{n,k}(\omega)(s)) \leq \sqrt{ \log (2 d)   } M_{n,k}/\lambda_{k}  $. 
 				
 				Since both bounds are valid, for any $s>0$
 				\begin{align*}
 				\Gamma_{k}(I_{n,k}(\omega)(s))	\leq \min \left\{ \sqrt{2 tr\{ W_{k}^{-1}\} } s, \sqrt{ \log (2 d)   } M_{n,k}/\lambda_{k} \right\}.
 				\end{align*} 
 				It is easy to see that the same bound holds for $\Gamma_{k}(J_{n,k}(\omega)(s))$, with $J_{n,k}$ defined as in   Lemma \ref{lem:r-fact}. Therefore, by Lemma \ref{lem:r-fact}(4) with $\Lambda_{n,k}(\omega) \equiv \sqrt{2 tr\{ W_{k}^{-1}\} }$, $B_{n,k}(\omega) \equiv \sqrt{ \log (2 d)   } M_{n,k}/\lambda_{k}$ and $C \equiv \mathbb{K}_{\pi_{0}}/\sqrt{n(\beta)}$, it follows that 
 				\begin{align}\label{eqn:HD-QR-l1-1}
 				V_{n,k}(\omega) \leq \min \left\{ \mathbb{K}_{\pi_{0}} \sqrt{\frac{ 2 tr\{ W_{k}^{-1}\} }{n(\beta)} } , \sqrt{\mathbb{K}_{\pi_{0}}} \left( \frac{\log (2 d)} {n(\beta)}\right)^{1/4} \sqrt{M_{n,k}/\lambda_{k}}  \right\}. 
 				\end{align}
 				
 				Regarding the bias term, it follows that
 				\begin{align*}
 				B_{k}(\mathbf{P})^{2} \leq Q(\theta_{\ast},\mathbf{P}) + \lambda_{k}||\theta_{\ast}||_{\ell^{1}}  - Q(\theta_{\ast},\mathbf{P}) = \lambda_{k}||\theta_{\ast}||_{\ell^{1}}.
 				\end{align*}
 				
 				By analogous arguments to those above, for any $\theta \in A_{0} \cup A_{1}$,
 				\begin{align*}
 				\delta_{\mathbf{P}}(\theta,\nu(\mathbf{P})) = \delta_{\mathbf{P}}(\theta,\theta_{\ast}) \geq \sqrt{0.5} ||\theta - \theta_{\ast}||_{2}.
 				\end{align*}
    (note that $t\theta + (1-t) \theta_{\ast} \in A_{0} \cup A_{1} $).

 			    Hence, display \ref{eqn:HD-QR-l1-1} and Theorem \ref{thm:concen-main} imply that
 				\begin{align*}
 				&\mathbf{P} \left( ||\nu_{k}(P_{n}) - \nu(\mathbf{P})||_{2} \leq Ce \left( \min \left\{ \sqrt{\frac{ tr\{ W_{k}^{-1}\} }{n(\beta)} } , \left( \frac{\log (2 d)} {n(\beta)}\right)^{1/4} \sqrt{M_{n,k}/\lambda_{k}}  \right\} + \sqrt{\lambda_{k}||\theta_{\ast}||_{\ell^{1}}} \right) \right)\\
 				\geq & 1- \mathbb{C}_{0}/u,
 				\end{align*}
 				where $Ce \equiv 2^{1/2}\max\{1, 2 \mathbb{K}_{\pi_{0}}\}$. 
 				
 				It is easy to see that $Ce$ is bounded above by $\max\{1, 2^{1/\pi_{0}} ||e_{max}(XX^{T})||_{L^{\pi_{0}}}   \} \max\{1, \sqrt{2} 10 L (1+\tau) \} $ and thus the result follows with $\mathbb{K} \equiv \max\{1, \sqrt{2} 10  L (1+\tau) \}$.
 				
 			\end{proof}
 			
 \begin{proposition}\label{pro:HD-QR-L2m-CR}
 	Let $p_{j} = j^{m}$ for any $j \in \mathbb{N}$. Then: For any $ u >0$  and any $(n,k)$:
 	
 	(1) For $m > 1$
 	\begin{align*}
 	& || W^{1/2}_{k}(\nu_{k}(P_{n}) - \theta_{\ast}) ||_{\ell^{2}} \leq u \mathbb{K} \max\{1, 2^{1/\pi_{0}} \mathbb{E}_{\pi_{0}}  \}  \left( \sqrt{ \lambda_{k} ||\theta_{\ast}||^{2}_{\ell^{2}(p)} }  +  \sqrt{ \frac{ \min \{  \lambda^{-1/m}_{k} , d \} } {n(\beta)}   \mathbb{B}_{k}  }  \right)
 	\end{align*} 	 		 					
 	with probability higher than $1-g_{0}(u)$, where $\mathbb{B}_{k} \equiv \left( \left( \frac{e_{min}(W_{k})}{2(m-1)}  \right)^{1/m}  +  1 \right) \frac{2}{ e_{min}(W_{k})}$.
 	
 	(2) For $m\in [0,1]$
 	\begin{align*}
 	& || W^{1/2}_{k}(\nu_{k}(P_{n}) - \theta_{\ast}) ||_{\ell^{2}} \leq u \mathbb{K} \max\{1, 2^{1/\pi_{0}} \mathbb{E}_{\pi_{0}}  \}  \left( \sqrt{ \lambda_{k} ||\theta_{\ast}||^{2}_{\ell^{2}(p)} }  +  \sqrt{ \frac{ \min \{  \lambda^{-1}_{k} \frac{d^{1-m}}{1-m} , 2 tr\{ W_{k}^{-1}  \}   \} } {n(\beta)}    }  \right)
 	\end{align*} 	
 	with probability higher than $1-g_{0}(u)$.	
 \end{proposition}
 
 \begin{proof}[Proof of Proposition \ref{pro:HD-QR-L2m-CR}]
 	(1) The proof is analogous to the one for Proposition \ref{pro:HD-QR-L1-new}, and thus we need to provide a bound for $\Gamma_{k}$ which allow us to use Lemma \ref{lem:r-fact}(4). 
 	
 	Let $A\equiv \{ \theta \in \Theta \colon  ||\theta||_{\ell^{2}(p)}^{2} \leq \max\{ M_{n,k} , E[M_{n,k}]  \} /\lambda_{k}     \}  $. 
 	
 	First note that, for any $\theta \in A$ , by Lemma \ref{lem:LASSO-delta-bdd} and analogous calculations to those in proof of Proposition \ref{pro:HD-QR-L1-new}
 	\begin{align*}
 	\delta_{k,\mathbf{P}}(\theta,\nu_{k}(\mathbf{P})) \geq  \sqrt{0.5 ||\theta - \nu_{k}(\mathbf{P})||^{2}_{2}  + \lambda_{k}||\theta - \nu_{k}(\mathbf{P})||_{\ell^{2}(p)}^{2}}  
 	\end{align*}
 	where $||.||^{2}_{2} = E[ ||W^{1/2}_{k} (\cdot) ||^{2}_{\ell^{2}}]$. One can write the RHS as $||\theta - \nu_{k}(\mathbf{P})||_{A_{k}} \equiv \sqrt{ (\theta - \nu_{k}(\mathbf{P}))^{T} A_{k}   (\theta - \nu_{k}(\mathbf{P}))   }$ with $A_{k} \equiv 0.5 W_{k} + \lambda_{k} diag\{ 1,2^{m},...,d^{m}  \}$. 
 	
 	Therefore, $I_{n,k}(\omega)(s) \subseteq A(s) \equiv \{ \theta \in \Theta \colon ||\theta - \nu_{k}(\mathbf{P})||_{A_{k}}  \leq s   \}$. Hence $\Gamma_{k}(I_{n,k}(\omega)(s)) \leq E[\sup_{\tau \in A(s)}  | \sum_{j=1}^{d} \zeta_{j} \theta_{j} | ]$. By letting $a \equiv ( \zeta_{j})_{j=1}^{d}$ and $b(\theta) \equiv ( \theta_{j})_{j=1}^{d}$, then by Cauchy-Swarhz, 
 	\begin{align*}
 	E[\sup_{\tau \in A(s)}  | \sum_{j=1}^{d} \zeta_{j} \theta_{j} | ] = E[\sup_{\tau \in A(s)}  |a^{T} b(\tau) | ] \leq  & E[\sup_{\tau \in A(s)}  ||A^{-1/2}_{k}a||_{\ell^{2}}  ||b(\tau) ||_{A_{k}} ] \\
 	\leq  & s \sqrt{\sum_{j=1}^{d} A^{-1}_{k}[j,j]}
 	\end{align*}
 	where the last line follows from the fact that $a$ are independent normal and $A^{-1}_{k}[j,j]$ denotes the $(j,j)$ element of $A^{-1}_{k}$. For any $a \in \{1,...,d\}$,
 	\begin{align*}
 	\sum_{j=1}^{d} A^{-1}_{k}[j,j]  \leq & \sum_{j=1}^{a}  2 (W_{k}[j,j])^{-1}  + \lambda^{-1}_{k} \sum_{j > a} j^{-m} \\
 	\leq &  \sum_{j=1}^{a}  2 (W_{k}[j,j])^{-1}  + \lambda^{-1}_{k} \int_{a}^{d}  x^{-m} dx \\
 	=  & 2a  \sup_{1 \leq a \leq d}\left(a^{-1} \sum_{j=1}^{a}   (W_{k}[j,j])^{-1} \right) + \lambda^{-1}_{k}  \frac{1}{m-1}(a^{-m+1} - d^{-m+1}) \\
 	\leq & 2a  \sup_{1 \leq a \leq d}\left(a^{-1} \sum_{j=1}^{a}   (W_{k}[j,j])^{-1} \right) + \lambda^{-1}_{k}  \frac{1}{m-1}a^{-m+1}
 	\end{align*}
 	Let $avg_{k} \equiv \sup_{1 \leq a \leq d}\left(a^{-1} \sum_{j=1}^{a}   (W_{k}[j,j])^{-1} \right)$.  By choosing $a = \lambda_{k}^{-1/m} \left(\frac{ \frac{1}{m-1}  }{ 2avg_{k}  } \right)^{1/m}$, the previous display shows that \footnote{More precisely, $a$ is chosen as the integer approximation of $\lambda_{k}^{-1/m} \left(\frac{ \frac{1}{m-1}  }{ 2 avg_{k}   } \right)^{1/m}$, but we omit this for the sake of presentation.}  
 	\begin{align*}
 	E[\sup_{\tau \in A(s)}  | \sum_{j=1}^{d} \zeta_{j} \theta_{j} | ]
 	\leq   s \sqrt{\lambda^{-1/m}_{k} \left(  \left(\frac{1}{m-1}  \right)^{1/m}  \left( 2 avg_{k}   \right)^{1-1/m}\right)  }.
 	\end{align*}
 	
 	If the choice of $a$ exceeds $d$ (i.e., if $d \leq \lambda_{k}^{-1/m} \left(\frac{ \frac{1}{m-1}  }{ 2avg_{k}  } \right)^{1/m}$ ), then an always valid bound is given by  $ E[\sup_{\tau \in A(s)}  | \sum_{j=1}^{d} \zeta_{j} \theta_{j} | ]
 	\leq   s \sqrt{ 2 tr\{ W_{k}^{-1}  \}  } \leq s \sqrt{ 2 d avg_{k} }   $. 
 	
 	Therefore, we proved that
 	\begin{align*}
 	\Gamma_{k}(I_{n,k}(\omega)(s)) \leq   s \sqrt{ \min \{  \lambda^{-1/m}_{k}  \left(\frac{ \frac{1}{m-1}  }{ 2avg_{k}  } \right)^{1/m} , d \}  2 avg_{k}  }.
 	\end{align*}
 	
 	By Theorem \ref{thm:effe-n} (with $\eta_{n} = 0$), the fact that $\theta\mapsto \phi(z,\theta)$ satisfies Assumption \ref{ass:Lip-phi} with $r = \pi_{0}$ and $\mathbb{C}_{k,M}(z) \equiv (1+\tau) e_{max}(xx^{T})$ (see Lemma \ref{lem:LASSO-d}) and Proposition \ref{pro:M-rate1}, it follows that\begin{align*}
 	V_{n,k} \leq \min \left\{ s > 0 \mid s \geq \frac{\mathbb{K}_{\pi_{0}}}{\sqrt{n(\beta)}} \max_{x \geq 1} \frac{\Gamma_{k}( I_{n,k}(\omega)(sx)  )}{sx}    \right\}.
 	\end{align*}
 	So the bound for $\Gamma_{k}$ readily implies that 
 	\begin{align*}
 	V_{n,k} \leq  \frac{\mathbb{K}_{\pi_{0}}}{\sqrt{n(\beta)}} \sqrt{ \min \{  \lambda^{-1/m}_{k}  \left(\frac{ \frac{1}{m-1}  }{ 2avg_{k}  } \right)^{1/m} , d \} 2 avg_{k}  },
 	\end{align*}	
 	where $\mathbb{K}_{\pi_{0}} $ is defined in the proof of Proposition \ref{pro:HD-QR-L1-new}. 
 	
 	Finally, note that, for any $v \in \mathbb{R}^{d}$, $v^{T}W_{k}v \geq v^{T}v e_{min}(W_{k})$. Thus, $avg_{k} \leq e_{min}(W_{k})^{-1}$.  Thus, the desired result follows by simple algebra  (see the last part of the proof of Proposition \ref{pro:HD-QR-L1-new}).
 	
 	\bigskip
 	
 	(2) For the case with $m\in[0,1]$, as before, an always valid bound is given by $ E[\sup_{\tau \in A(s)}  | \sum_{j=1}^{d} \zeta_{j} \theta_{j} | ]
 	\leq   s \sqrt{ 2 tr\{ W_{k}^{-1}  \}  }   $. In addition,	
 	\begin{align*}
 	\sum_{j=1}^{d} A^{-1}_{k}[j,j]  \leq & \lambda_{k}^{-1} \sum_{j=1}^{d} j^{-m} \leq \lambda_{k}^{-1} \frac{d^{1-m}}{1-m}, 
 	\end{align*}
 	and the bound follows. 
\end{proof}

\begin{proposition}\label{pro:OLS-q}
	For any $n \in \mathbb{N}$, $\mu_{0}\leq n$ and $l = 1,...,d$, for any $t \geq 1$
	\begin{align*}
	\mathbf{P} \left( |n^{-1} \sum_{i=1}^{n} x_{i,l} U_{i}  | \geq t \sqrt{\frac{\mu_{0}}{n}}  \right) \leq  \frac{4}{t} \sqrt{E_{\mathbf{P}} \left[ \left(\frac{|\varDelta|}{\sqrt{\mu_{0}}} \right)^{2} \right]}.
	\end{align*}
\end{proposition}

\begin{remark}
	As pointed out in the proof below, one can obtain a better bound of order $t^{-2} \times log(t^{2})$ by refining the arguments in the proof. The goal here is not to obtain the best possible bound but to present a proof that resembles the technique of proof we used in the paper. Also, even when comparing our method to the sharper bound of order $t^{-2} \times log(t^{2})$, it is worth to point out that our method cannot exploit the explicit solution of the OLS estimator, and is designed for general $\beta$-mixing processes.  $\triangle$ 
\end{remark}

     \begin{proof}[Proof of Proposition \ref{pro:OLS-q}]
     	So as to ease the notational burden, we omit the subscript $l$ from $\varDelta_{j,l}$. For any $K_{1}>0$, let $\varDelta^{L}_{j} \equiv \varDelta_{j} 1\{ |\varDelta_{j}| \leq \sqrt{\mu_{0}} K_{1} K_{0}  \}$ and $\varDelta^{U}_{j} \equiv \varDelta_{j} 1\{ |\varDelta_{j}| \geq \sqrt{\mu_{0}} K_{1} K_{0}  \}$. By construction $E[\varDelta_{j}] = E[\varDelta^{L}_{j}]  + E[\varDelta^{U}_{j}] = 0$ and thus 
     	\begin{align*}
	n^{-1} \sum_{i=1}^{n} x_{i,l} U_{i} = \frac{\sqrt{\mu_{0}}}{n} \sum_{j=0}^{J} \bar{\varDelta}^{L}_{j} + \frac{\sqrt{\mu_{0}}}{n} \sum_{j=0}^{J} \bar{\varDelta}^{U}_{j},
     	\end{align*}
     	where $\bar{\varDelta}^{z}_{j} \equiv \varDelta^{z}_{j} - E[\varDelta^{z}_{j}]$ for $z \in \{L,U\}$. Therefore, for any $u>0$,
 	\begin{align*}
 	\mathbf{P} \left( |n^{-1} \sum_{i=1}^{n} x_{i,l} U_{i}  | \geq u   \right) \leq & \mathbf{P} \left( | \frac{\sqrt{\mu_{0}}}{n} \sum_{j=0}^{J} \bar{\varDelta}^{L}_{j} | \geq 0.5 u   \right) + \mathbf{P} \left( | \frac{\sqrt{\mu_{0}}}{n} \sum_{j=0}^{J} \bar{\varDelta}^{U}_{j} | \geq 0.5 u   \right)\\
 	\equiv&  Term_{1}(u) + Term_{2}(u).
 	\end{align*}    	
 	We now bound the $Term_{1}(u)$ (using Bernstein inequality) and $Term_{2}(u)$ (using Markov inequality).
 	
 	By Bernstein inequality and the fact that $|\varDelta^{L}_{j}| \leq K_{1} K_{0} \sqrt{\mu_{0}}$,
 	 	\begin{align*}
 	 	\mathbf{P} \left(  \left|\frac{\sqrt{\mu_{0}}}{n} \sum_{j=0}^{J} \bar{\varDelta}^{L}_{j} \right|   \geq u   \right) \leq \exp \left\{ - \frac{0.5 u^{2} n^{2} \mu_{0}^{-1} }{ \sum_{j=0}^{J} E[(\bar{\varDelta}^{L})^{2}]   + \frac{u}{3} n \mu_{0}^{-1/2} \mu_{0}^{1/2} K_{0}K_{1}  }   \right\}
 	 	\end{align*}
 	 	for all $u > 0$. Hence, since $J+1=n/\mu_{0}$,
 	 	\begin{align*}
 	 	\mathbf{P} \left(  \left|\frac{\sqrt{\mu_{0}}}{n} \sum_{j=0}^{J} \bar{\varDelta}^{L}_{j} \right|   \geq  u   \right) \leq & \exp \left\{ - \frac{0.5 u^{2} n^{2} \mu_{0}^{-1} }{ \sum_{j=0}^{J} E[(\bar{\varDelta}^{L})^{2}]   + \frac{u}{3} n  K_{0}K_{1}  }   \right\}\\
 	 	=& \exp \left\{ - \frac{0.5 u^{2} }{ \mu_{0}\frac{(J+1)}{n^{2}} E[(\bar{\varDelta}^{L})^{2}]   + \frac{u}{3} \frac{\mu_{0}}{n}  K_{0}K_{1}  }   \right\}\\
 	 	= & \exp \left\{ - \frac{0.5 u^{2} }{ \frac{1}{n} E[(\bar{\varDelta}^{L})^{2}]   + \frac{u}{3} \frac{\mu_{0}}{n}  K_{0}K_{1}  }   \right\}\\
 	 	\leq & \exp \left\{ - \frac{0.5 u^{2} }{ \frac{1}{n} E[(\varDelta)^{2}]   + \frac{u}{3} \frac{\mu_{0}}{n}  K_{0}K_{1}  }   \right\}
 	 	\end{align*}
     	where the last line follows from the fact that $E[(\bar{\varDelta}^{L})^{2}] \leq E[(\varDelta^{L})^{2}] \leq E[(\varDelta)^{2}]$.
     	
 	If $E[(\varDelta)^{2}]  \geq  \frac{u}{3} \mu_{0}  K_{0}K_{1}$, then for any $u > 0$,
 	\begin{align*}
 	\mathbf{P} \left(  \left|\frac{\sqrt{\mu_{0}}}{n} \sum_{j=0}^{J} \bar{\varDelta}^{L}_{j} \right|   \geq  u    \right) \leq  \exp \left\{ - \frac{u^{2} }{ 4 \frac{1}{n} E[(\varDelta)^{2}] }   \right\}
 	\end{align*}
 	iff, for all $t > 0$ 
 	\begin{align*}
 	\mathbf{P} \left(  \left|\frac{\sqrt{\mu_{0}}}{n} \sum_{j=0}^{J} \bar{\varDelta}^{L}_{j} \right|   \geq  2 \sqrt{E[(\varDelta)^{2}]} n^{-1/2} \sqrt{t}    \right) \leq  \exp \left\{ - t   \right\},~if~\sqrt{E[(\varDelta)^{2}]}  \geq  \frac{2}{3} t \frac{\mu_{0}}{\sqrt{n}}  K_{0}K_{1}.
 	\end{align*}
 	
 	Else, if $E[(\varDelta)^{2}]  \leq  \frac{u}{3} \mu_{0}  K_{0}K_{1}$, then for any $u > 0$,
 	\begin{align*}
 	\mathbf{P} \left(  \left|\frac{\sqrt{\mu_{0}}}{n} \sum_{j=0}^{J} \bar{\varDelta}^{L}_{j} \right|   \geq  u    \right) \leq  \exp \left\{ - \frac{u }{ \frac{4}{3} \frac{\mu_{0}}{n} K_{0}K_{1} }   \right\}
 	\end{align*}
 	iff, for all $t > 0$  
 	\begin{align*}
 	\mathbf{P} \left(  \left|\frac{\sqrt{\mu_{0}}}{n} \sum_{j=0}^{J} \bar{\varDelta}^{L}_{j} \right|   \geq  \frac{4}{3} \frac{\mu_{0}}{n} K_{0}K_{1} t  \right) \leq  \exp \left\{ - t   \right\},~if~\sqrt{E[(\varDelta)^{2}]}  \leq  \frac{2}{3} t \frac{\mu_{0}}{\sqrt{n}}  K_{0}K_{1} .
 	\end{align*}
 	
 	When $t\geq 1$ it follows that,
 	\begin{align*}
 	& 2 \sqrt{E[(\varDelta)^{2}]} n^{-1/2} \sqrt{t} 1\{  \sqrt{E[(\varDelta)^{2}]}  \geq  \frac{2}{3} t \frac{\mu_{0}}{\sqrt{n}}  K_{0}K_{1}  \} + \frac{4}{3} \frac{\mu_{0}}{n} K_{0}K_{1} t   1\{ \sqrt{E[(\varDelta)^{2}]}  \leq  \frac{2}{3} t \frac{\mu_{0}}{\sqrt{n}}  K_{0}K_{1}  \}\\
 	\leq & \sqrt{\frac{\mu_{0}}{n}} t  \left( 2 \sqrt{\frac{E[(\varDelta)^{2}]}{\mu_{0}}} 1\{  \sqrt{\frac{E[(\varDelta)^{2}]}{\mu_{0}}}  \geq  \frac{2}{3} t \sqrt{\frac{\mu_{0}}{n}}  K_{0}K_{1}  \} + \frac{4}{3} \sqrt{\frac{\mu_{0}}{n}} K_{0}K_{1}   1\{ \sqrt{\frac{E[(\varDelta)^{2}]}{\mu_{0}}}  \leq  \frac{2}{3} t \sqrt{\frac{\mu_{0}}{n}}  K_{0}K_{1}  \}   \right)\\
 	\leq & \sqrt{\frac{\mu_{0}}{n}} t \max \left\{ 4 \sqrt{\frac{E[(\varDelta)^{2}]}{\mu_{0}}} , \frac{8}{3} \sqrt{\frac{\mu_{0}}{n}} K_{0}K_{1}     \right\}.
 	\end{align*}
 	Therefore, for any $t \geq 1$,
 	\begin{align*}
 	Term_{1}(2t) = \mathbf{P} \left( \left|\frac{\sqrt{\mu_{0}}}{n} \sum_{j=0}^{J} \bar{\varDelta}^{L}_{j} \right| \geq t   \right) \leq \exp \left\{- \frac{t}{ \sqrt{\frac{\mu_{0}}{n}} \max \left\{ 4 \sqrt{\frac{E[(\varDelta)^{2}]}{\mu_{0}}} , \frac{8}{3} \sqrt{\frac{\mu_{0}}{n}} K_{0}K_{1}     \right\}} \right\}. 
 	\end{align*}
 	
 	Regarding $Term_{2}(2u) = \mathbf{P} \left( | \frac{\sqrt{\mu_{0}}}{n} \sum_{j=0}^{J} \bar{\varDelta}^{U}_{j} | \geq u   \right)$. By the Markov Inequality, it follows that for any $u>0$
 	\begin{align*}
 	Term_{2}(2u) \leq \frac{\sqrt{\mu_{0}}/n}{u} E_{\mathbf{P}} \left[ \sum_{j=0}^{J} |\bar{\varDelta}^{U}_{j}| \right] \leq \frac{2}{u} E_{\mathbf{P}} \left[ \frac{|\varDelta|}{\sqrt{\mu_{0}}} 1\{ | \varDelta|/\sqrt{\mu_{0}} \geq K_{0} K_{1}   \}\right]
 	\end{align*}
 	and applying Markov inequality again, it follows that $Term_{2}(2u) \leq \frac{2}{uK_{0}K_{1}} E_{\mathbf{P}} \left[ \left(\frac{|\varDelta|}{\sqrt{\mu_{0}}} \right)^{2} \right] $.
 	
 		Therefore, by a simple change of variables 
 		\begin{align*}
 		\mathbf{P} \left( |n^{-1} \sum_{i=1}^{n} x_{i,l} U_{i}  | \geq 2 s  \sqrt{\frac{\mu_{0}}{n}}  \right) \leq & \exp \left\{- \frac{s}{  \max \left\{ 4 \sqrt{\frac{E[(\varDelta)^{2}]}{\mu_{0}}} , \frac{8}{3} \sqrt{\frac{\mu_{0}}{n}} K_{0}K_{1}     \right\}} \right\} \\
 		& + \sqrt{\frac{n}{\mu_{0}}} \frac{2}{s K_{0}K_{1}} E_{\mathbf{P}} \left[ \left(\frac{|\varDelta|}{\sqrt{\mu_{0}}} \right)^{2} \right].
 		\end{align*}  
 	Choosing $K_{1}$ so that $K_{0}K_{1} = \sqrt{\frac{n}{\mu_{0}}} \sqrt{\frac{E[(\varDelta)^{2}]}{\mu_{0}}}$, it follows that
  		\begin{align*}
  		\mathbf{P} \left( |n^{-1} \sum_{i=1}^{n} x_{i,l} U_{i}  | \geq 2 s  \sqrt{\frac{\mu_{0}}{n}}  \right) \leq & \exp \left\{- \frac{s}{4\sqrt{E_{\mathbf{P}} \left[ \left(\frac{|\varDelta|}{\sqrt{\mu_{0}}} \right)^{2} \right]}} \right\} + \frac{2}{s} \sqrt{E_{\mathbf{P}} \left[ \left(\frac{|\varDelta|}{\sqrt{\mu_{0}}} \right)^{2} \right]} \\
  		\leq & \frac{4}{s} \sqrt{E_{\mathbf{P}} \left[ \left(\frac{|\varDelta|}{\sqrt{\mu_{0}}} \right)^{2} \right]}.
  		\end{align*} 
  		
  		\begin{remark}
 			$K_{1}$ was chosen to optimize the RHS bound; if one allows $K_{1}$ to depend on $s$ we obtain a better bound of order $s^{-2} \times log(s^{2})$ as opposed to $s^{-1}$ obtained now. 
$\triangle$ 
  		\end{remark}	 	
     \end{proof}

  	   \section{Proofs for Results in Section \ref{sec:choice}}
  	   \label{app:choice}

  	   \begin{proof}[Proof of Lemma \ref{lem:k-inf}]
  	   	The result follows from Proposition \ref{pro:concen-w}, the definition of $\mathcal{I}(P_{n},\mathbf{P})$ and assumption \ref{ass:V-bdd}.
  	   \end{proof}

  	   		   	   	\begin{proof}[Proof of Proposition \ref{pro:choicek-asym}]
  	   		   	   		We need to show that for any $\epsilon>0$, there exists a $M(\epsilon)$ and $N(\epsilon)$ such that \begin{align*}
  	   		   	   		\mathbf{P}( \varpi (\nu_{k_{s_{n}}^{F}(P_{n})}(P_{n}) , \nu(\mathbf{P})  )    \geq M(\epsilon) s_{n} \tilde{V}_{k^{I}(P_{n},\mathbf{P})} ) \leq \epsilon
  	   		   	   		\end{align*} for all $n \geq N(\epsilon)$. 
  	   		   	   		
  	   		   	   		By theorem \ref{thm:k-fea}, it follows that 
  	   		   	   		\begin{align*}
  	   		   	   		\mathbf{P}( \varpi (\nu_{k_{s_{n}}^{F}(P_{n})}(P_{n}) , \nu(\mathbf{P})  )    \geq s_{n} 6 \mathbb{V} \tilde{V}_{k^{I}(P_{n},\mathbf{P})} ) \leq |\mathcal{K}| g_{0}(s_{n})
  	   		   	   		\end{align*} for all $n \in \mathbb{N}$. By assumption, $|\mathcal{K}| g_{0}(s_{n}) \rightarrow 0$, so for any $\epsilon>0$, there exists a $N(\epsilon)$ such that $|\mathcal{K}| g_{0}(s_{n}) < \epsilon$ for any $n \geq N(\epsilon)$. So the result follows by setting $M(\epsilon) = 6$.    	   		
  	   		   	   	\end{proof}	
  	   			
%
%
%
%
%

\begin{proof}[Proof of Theorem \ref{thm:k-fea}]
	
	Henceforth, we simply write $k^{F}_{s}$  and $k^{I}$ instead of $k^{F}_{s}(P_{n})$ and $k^{I}(P_{n},\mathbf{P})$. Also, we prove the statement for $\mathbb{V}=1$ so as to simplify further the notation; from general $\mathbb{V} \geq 1$ the proof is the same, but simply replacing $\tilde{V}_{k}$ by $\mathbb{V} \tilde{V}_{k} $. Also, recall that $s>0$ is fixed.
	%
	We first define sets we use throughout the proof: 
	\begin{align*}
	A_{n,s} \equiv & \{ \omega \in \Omega \mid  k^{I} \geq k^{F}_{s}   \}\\
	B_{n}(s) \equiv & \{ \omega \in \Omega \mid  \varpi(\nu_{k^{F}_{s}}(P_{n}) , \nu(\mathbf{P})  )  > s6 \tilde{V}_{k^{I}}(P_{n})    \}\\
	C_{n}(s) \equiv & \{ \omega \in \Omega \mid \varpi (\nu_{k^{I}}(P_{n}) , \nu(\mathbf{P})  ) < 2 s \tilde{V}_{k^{I}}(P_{n})   \}\\
	D_{n}(s) \equiv & \{ \omega \in \Omega \mid \varpi (\nu_{k^{F}_{s}}(P_{n}) , \nu(\mathbf{P})  ) \leq \varpi (\nu_{k^{F}_{s}}(P_{n}) , \nu_{k^{I}} (P_{n}) ) \\
	& + 2 s \tilde{V}_{k^{I}}(P_{n})    \}\\
	E_{n}(s,k)  \equiv & \{ \omega \in \Omega \mid  \varpi (\nu_{k}(P_{n}) , \nu(\mathbf{P})  )  < s \left( \tilde{V}_{k}(P_{n}) + \sqrt{B_{k}(\mathbf{P})} \right)   \}
	\end{align*}
	and $E_{n}(s) = \cap_{k \in \mathcal{K}} E_{n}(s,k)$.  	In this notation, we want to show that $\mathbf{P}( B_{n}(s)) \leq 2 | \mathcal{K} | g_{0}(s)$.  
	
	By Proposition \ref{pro:concen-w}, $\mathbf{P}(\Omega \setminus E_{n}(s)) \leq \sum_{k \in \mathcal{K}} \mathbf{P}(\Omega \setminus E_{n}(s,k)) \leq | \mathcal{K}| g_{0}(s)$. Also, by Lemma \ref{lem:k-inf}, $\mathbf{P}(  \Omega \setminus C_{n}(s) )$  is bounded by $ g_{0}(s) \leq  |\mathcal{K}| g_{0}(s)$.

	Thus, it suffices to show
	\begin{align}\label{eqn:k-fea-0}
	\mathbf{P}(  E_{n}(s) \cap C_{n}(s) \cap B_{n}(s) ) = 0.
	\end{align}
	Henceforth, let $F_{n}(s) \equiv E_{n}(s) \cap C_{n}(s) \cap B_{n}(s)$. It suffices to show \begin{align}\label{eqn:k-fea-1}
	\mathbf{P}(  F_{n}(s) \cap D_{n}(s) ) + \mathbf{P}(  \{\Omega \setminus D_{n}(s)\} \cap F_{n}(s) ) = 0.
	\end{align}
	
		We first show that $E_{n}(s) \cap C_{n}(s) \subseteq \{ \omega \in \Omega \mid k^{I} \in \mathcal{F}_{s}(P_{n})   \}$. First note that $k^{I} \in \mathcal{K}$ by construction. If $\omega \in E_{n}(s) \cap C_{n}(s) $, then for all $k \geq k^{I}$,
		\begin{align*}
		\varpi (\nu_{k^{I}}(P_{n}) , \nu_{k}(P_{n})  )  \leq &  \left( \varpi (\nu_{k^{I}}(P_{n}) , \nu (\mathbf{P}) )  + \varpi (\nu_{k}(P_{n}) , \nu (\mathbf{P}) ) \right) \\
		\leq &  s \left\{ 2  \tilde{V}_{k^{I}}(P_{n})  +   \tilde{V}_{k}(P_{n}) + \sqrt{B_{k}(\mathbf{P})} \right\} \\
		\leq &  s \left\{  4 \tilde{V}_{k}(P_{n}) \right\}
		\end{align*}
		where the first line follows from triangle inequality; the third line follows because, since $k \geq k^{I}$, then $\tilde{V}_{k}(P_{n}) \geq  \tilde{V}_{k^{I}}(P_{n}) \geq \sqrt{B_{k^{I}}(\mathbf{P})}  \geq \sqrt{B_{k}(\mathbf{P})}$ 		(note that under assumption \ref{ass:V-bdd}(ii), $k \mapsto B_{k}(\mathbf{P})$ is non-increasing). Thus $k^{I} \in \mathcal{F}_{s}(P_{n})$.

	For any $\omega \in C_{n}(s)$ it follows that
	\begin{align}\label{eqn:k-fea-2}
	\varpi (\nu_{k^{F}_{s}}(P_{n}) , \nu(\mathbf{P})  ) \leq & \varpi (\nu_{k^{F}_{s}}(P_{n}) , \nu_{k^{I}} (P_{n}) ) + \varpi  (\nu_{k^{I}} (P_{n}) , \nu(\mathbf{P}) ) \\
	< & \varpi (\nu_{k^{F}_{s}}(P_{n}) , \nu_{k^{I}} (P_{n}) ) + 2 s \tilde{V}_{k^{I}}(P_{n}).
	\end{align}
	But this means that $\omega \in D_{n}(s)$ and thus $\mathbf{P}( \{ \Omega \setminus D_{n}(s) \} \cap F_{n}(s) ) = 0$. Therefore, from display \ref{eqn:k-fea-1}  it remains to show that $\mathbf{P}(  F_{n}(s) \cap D_{n}(s) ) = 0$. For this, observe that 	
	\begin{align}\label{eqn:k-fea-3}
	\mathbf{P}(  F_{n}(s) \cap D_{n}(s) ) \leq  \mathbf{P}(F_{n}(s) \cap A_{n,s} \cap D_{n}(s) ) +  \mathbf{P}(F_{n}(s) \cap \{\Omega \setminus A_{n,s} \} \cap D_{n}(s) ).
	\end{align}
	
	Regarding the first expression in the RHS, note that when $k^{I} \geq k^{F}_{s} $, by definition of $\mathcal{F}_{s}(P_{n})$ and the fact that $k^{I} \in \mathcal{K}$, it follows that 
	\begin{align*}
	\varpi (\nu_{k^{F}_{s}}(P_{n}) , \nu_{k^{I}} (P_{n}) ) \leq 4 s \tilde{V}_{k^{I}}(P_{n}),
	\end{align*}
	and thus,
	\begin{align*}
	\varpi (\nu_{k^{F}_{s}}(P_{n}) , \nu(\mathbf{P})  )  \leq & \varpi (\nu_{k^{F}_{s}}(P_{n}) , \nu_{k^{I}} (P_{n}) )  + \varpi (\nu_{k^{I}} (P_{n}) , \nu(\mathbf{P})  ) \\
	\leq & 4 s \tilde{V}_{k^{I}}(P_{n})  + \varpi (\nu_{k^{I}} (P_{n}) , \nu(\mathbf{P})  )\\
	= & 6 s \tilde{V}_{k^{I}}(P_{n})
	\end{align*}
	where the second line follows from the fact that $\omega \in C_{n}(s)$.
	But this contradicts the inequality in the definition of $B_{n}(s)$, thus $\mathbf{P}(D_{n}(s) \cap A_{n,s} \cap F_{n}(s) )  = 0$. 
	
	Therefore, it only remains to show that the second term in display \ref{eqn:k-fea-3} is zero. In this case $k^{I} < k^{F}_{s}$. This implies that $\tilde{V}_{k^{I}}(P_{n}) \leq \tilde{V}_{k^{F}_{s}}(P_{n}) $. However, since $k^{I} \in \mathcal{F}_{s}(P_{n})$, it also follows that  $\tilde{V}_{k^{I}}(P_{n}) \geq \tilde{V}_{k^{F}_{s}}(P_{n}) $. Hence  $\tilde{V}_{k^{I}}(P_{n}) = \tilde{V}_{k^{F}_{s}}(P_{n}) $. Therefore,\begin{align*}
	\tilde{V}_{k^{F}_{s}}(P_{n}) + \sqrt{B_{k^{F}_{s}}(\mathbf{P})} = & \tilde{V}_{k^{I}}(P_{n}) + \sqrt{B_{k^{F}_{s}}(\mathbf{P})} \\
	\leq &  \tilde{V}_{k^{I}}(P_{n}) + \sqrt{B_{k^{I}}(\mathbf{P})} ,~because~k^{I} < k^{F}_{s} \\
	\leq & 2 \tilde{V}_{k^{I}}(P_{n}) .
	\end{align*}
	
	Hence, by the fact that $\omega \in E_{n}(s)$, it follows that $\varpi (\nu_{k^{F}_{s}}(P_{n}) , \nu(\mathbf{P})  )  \leq 2 s \tilde{V}_{k^{I}}(P_{n})$. But this contradicts $\omega \in B_{n}(s)$.
\end{proof}

	\begin{proof}[Proof of Proposition \ref{pro:LASSO-choice}]
		Under our assumption $||.||_{2} \geq \underline{\varpi} ||.||_{\ell^{2}}$. So, by our calculations in Example \ref{exa:HD-QR}, it follows, for $Pen = ||.||_{\ell^{1}}$, 
	\begin{align*}
	 ||\nu_{k}(P_{n}) - \theta_{\ast}||_{\ell^{2}} \leq u \mathbb{V}_{1} \left\{   A_{n} \sqrt{\lambda^{-1}_{k} n^{-1} \sum_{i=1}^{n} \phi(Z_{i},0) } + \sqrt{\lambda_{k}}||\theta_{\ast} ||_{\ell^{1}}     \right\}
	\end{align*} 		
with $A_{n} \equiv (\log 2 d /n(\beta))^{1/4}$ with probability higher than $1-g_{0}(u)$. Similarly, for $Pen=||.||_{\ell^{2}(p)}$
	\begin{align*}
	 ||\nu_{k}(P_{n}) - \theta_{\ast}||_{\ell^{2}} \leq u \mathbb{V}_{2} \left\{   \sqrt{\frac{\lambda^{-1/m}_{k} }{n(\beta)}  } + \sqrt{\lambda_{k} ||\theta_{\ast} ||^{2}_{\ell^{2}(p)} }    \right\}.
	\end{align*} 	
	The fact that the RHS has $n^{-1} \sum_{i=1}^{n} \phi(Z_{i},0)$ is valid because $0 \in \Theta$. This simplifies the expressions, in case $0 \notin \Theta$ any other element of $\Theta$ would work.
	
	By choosing $\lambda_{k}$ so as $A_{n}^{2} \lambda^{-1}_{k} n^{-1} \sum_{i=1}^{n} \phi(Z_{i},0)  = \lambda_{k} ||\theta_{\ast} ||^{2}_{\ell^{1}}  $		and $\frac{\lambda^{-1/m}_{k} }{n(\beta)}   = \lambda_{k} ||\theta_{\ast} ||^{2}_{\ell^{2}(p)}  $, resp., one obtains the ``infeasible choice" of tuning parameter. 
	
	This choice imply, for $Pen = ||.||_{\ell^{1}}$, 
		\begin{align*}
		||\nu_{k}(P_{n}) - \theta_{\ast}||_{\ell^{2}} \leq u \mathbb{V}_{1} \left\{   \sqrt{A_{n}} (n^{-1} \sum_{i=1}^{n} \phi(Z_{i},0) )^{1/4} \sqrt{||\theta_{\ast} ||_{\ell^{1}} } \right\}
		\end{align*} 		
	 and, for $Pen=||.||_{\ell^{2}(p)}$
		\begin{align*}
		||\nu_{k}(P_{n}) - \theta_{\ast}||_{\ell^{2}} \leq u \mathbb{V}_{2} \left\{   n^{-\frac{m}{2(m+1)} } ||\theta_{\ast}||_{\ell^{2}(p)}^{\frac{1}{m+1}}  \right\},
		\end{align*} 	
    with probability higher than $1-g_{0}(u)$. Similarly,		
	\end{proof}

    	\subsection{Discussion of Assumption \ref{ass:V-bdd} and $|\mathcal{K}|$}
    	\label{sec:discuss-V-bdd}
    	
    	%

    	The  monotonicity assumption in $k \mapsto \tilde{V}_{k}(P_{n})$ is quite mild. For instance, we need $k \mapsto M_{n,k}$ to be non-decreasing and $k \mapsto \underline{\varpi}_{k}$ to be non-increasing at least for $k \geq k_{0}$ for some $k_{0} \in \mathbb{N}$. Regarding the former, this property follows simply by noting that there exists a $\theta_{0}$ such that: (i) $\theta_{0} \in \Theta_{k_{0}}$ for some $k_{0}$ and (ii) $Pen(\theta_{0})$ is finite; and thus one can set $\theta_{k} = \theta_{0}$ for all $k \geq k_{0}$ (and arbitrarily elsewhere); such that choice exists because $Pen$ is finite in $\Theta_{k}$ for any $k < \infty$. Since $\underline{\varpi}_{k}$ is a lower bound, imposing that is non-increasing can always be achieved by taking the $\underline{\varpi}_{k}$ to be the ``worst" bound among $\underline{\varpi}_{j}$ for $j \leq k$.
    	
    	\medskip

    	The following corollary formalizes the discussion on the $|\mathcal{K}|$ factor of Theorem \ref{thm:k-fea}. It shows that conditioning on the set $E_{n}(s)$ there is no $|\mathcal{K}|$ present on the concentration bound. 
    	\begin{corollary}
    		Under the same assumptions of Theorem \ref{thm:k-fea},
    		\begin{align*}
    		\mathbf{P}( \varpi (\nu_{k^{F}_{s}(P_{n})}(P_{n}) , \nu(\mathbf{P})  )  \geq s 6 \mathbb{V} \tilde{V}_{k^{I}(P_{n},\mathbf{P})}(P_{n})  \mid E_{n}(s)  ) \leq   g_{0}(s) 
    		\end{align*} 
    		with $E_{n}(s) = \cap_{k \in \mathcal{K}} E_{n}(s,k)$ and 	
    		\begin{align*}
    		E_{n}(s,k)  \equiv  \{ \omega \in \Omega \mid  \varpi (\nu_{k}(P_{n}) , \nu(\mathbf{P})  )  < s \left( \tilde{V}_{k}(P_{n}) + \sqrt{B_{k}(\mathbf{P})} \right)   \}.
    		\end{align*}
    	\end{corollary}
    	
    	This corollary also shows that if one has a better bound for $\mathbf{P}(\cap_{k \in \mathcal{K}} E_{n}(.,k))$ then one can improve on the overall bound of Theorem \ref{thm:k-fea}. For instance, this will be the case if there exists a $u \geq 0$ such that $ \mathbf{P} (E_{n}(s,k)) = 1$ for each $k \in \mathcal{K}$ and all $n$ large.

 	   	\newpage
 	   	
\setcounter{section}{0}
\renewcommand*{\theHsection}{chX.\the\value{section}}
 	   	\renewcommand{\thesection}{OA.\arabic{section}}
 	   	\renewcommand{\thesubsection}{\thesection.\arabic{subsection}}
 	   	\renewcommand{\thesubsubsection}{\thesubsection.\arabic{subsubsection}}

 	   	\clearpage
 	   	\setcounter{page}{1}
 	   	
 	   	\begin{center}
 	   			{\huge{Online Appendix}}
 	   	\end{center}
 	   	
 	   	\bigskip

 	   	\section{Some observations about our choices of $n$ and $\mathcal{Q}_{n}$}
 	   	\label{app:domN}
  	   	
 	   	Let $\mathcal{N} \equiv \{ m \colon m = \prod_{i=1}^{\upsilon} p_{i}^{m_{i}},~some~(m_{i})_{i=1}^{\upsilon} \in \mathbb{N}_{0}^{\upsilon} \}$ for some $\upsilon \in \mathbb{N}$ and $(p_{i})_{i=1}^{\upsilon}$ consecutive prime. Our restrictions imply that $n \in \mathcal{N}$. It is well-known that any $n \in \mathbb{N}$ admits a prime factorization $\prod_{i=1}^{e} p_{i}^{e_{i}}$ where $(p_{i},e_{i})_{i=1}^{e}$ and $e$ depend on $n$. Our restriction on $n$ imposes that $e$ and $(p_{i})_{i}$ \emph{do not depend} on $n$. Specifically, we guarantee that $p_{e}$ (the highest prime) is fixed. The assumption that the primes are consecutive is just for simplicity.  
 	   	
 	   	For instance, for $\upsilon = 2$, then $n \in \{ 1, 2, 4, 8, 16, ..., 2^{m} ,... \}$, and for $ \upsilon = 3$, then $n \in \{ 1, 2, 3, 4, 6, 8, 9, 12, ..., 2^{m_{1}} 3^{m_{2}},... \}$. Note that once we fixed $\upsilon$ certain values of $n$ are not admissible, say, if $\upsilon=3$ then neither $n=7$ or $n=11$ (or any prime larger than $3$) are valid choices of $n$. 
 	   	
 	   	These restrictions in $n$ ensure that consecutive elements of $\mathcal{Q}_{n}$ are not ``too far apart". Formally, let $D[n]$ the set of divisors of $n$, then
 	   	\begin{lemma}\label{lem:DivN}
 	   		For any $n \in \mathcal{N}$ and $a \in D[n]$, $a>1$. Then there exists a $a' \in D[n]$ such that $a'<a$ and $p_{\upsilon} a' \geq a$.    
 	   	\end{lemma}
 	   	
 	   	\begin{proof}
 	   		If $a \in D[n]$, then it must be that $a = \prod_{i=1}^{\upsilon} p_{i}^{f_{i}}$  for some integers $(f_{i})_{i}$ such that $f_{i} \leq m_{i}$. Let $l$ be the first $i \in \{1,...,\upsilon\}$ such that $f_{i} \geq 1$ (such $l$ exists because $a>1$). Consider $a' = p_{l}^{f_{l}-1} \prod_{i=l+1}^{e} p_{i}^{f_{i}}$. Clearly $a' < a$ and $p_{v}a' \geq p_{l} a' = a$. Moreover, $\frac{n}{a'} = p_{l} \frac{n}{a}$ which is an integer because $a \in D[n]$. Therefore $a' \in D[n]$. 
 	   	\end{proof}
 	   	
 	   	This lemma shows how the choice of $\upsilon$ and $(p_{i})_{i=1}^{\upsilon}$ affect the ``separation" in $D[n]$. For instance, for $\upsilon = 2$, $2 a' \geq a$ for any $a' \leq a$ in $D[n]$. But if $\upsilon = 7$, then the case $n = 7$ or $n=14$ are such that $7 a' \geq a$.
 	   	
 	   	We note that $\mathcal{Q}_{n} = D[n]$, so the previous lemma readily applies to elements in $\mathcal{Q}_{n}$. In Lemma \ref{lem:L1-bdd} we use this result to establish the bound for the $L^{1}$ norm.

 	   	\section{A Simple Linear Regression Model: MA structure}
 	   	\label{app:OLS-q-MA}

 	   	The setup is the same as that of Example \ref{exa:OLS-q} but with the difference that $U_{i} = \varepsilon_{i} - \theta_{1} \varepsilon_{i-1} - ... - \theta_{q} \varepsilon_{i-q}$ with $(\varepsilon_{i})_{i}$ being an i.i.d. sequence of mean zero bounded random variables, and $q \equiv \mu_{0} \in \mathbb{N}$. It is well known that $(U_{i})_{i}$ is a $q$-dependent sequence.
 	   	
 	   	We now study $n^{-1} \sum_{i=1}^{n} x_{i,l} U_{i}$ and how to decompose it into i.i.d. blocks. Henceforth, we omit the $l$ subscript. Consider $B_{0} = (x_{1} U_{1},....,x_{q} U_{q})$, $B_{1} = (x_{1+q} U_{1+q},....,x_{2q} U_{2q})$,..., $B_{j} = (x_{jq+1} U_{jq+1},....,x_{jq+q} U_{jq+q})$,.... As usual, we assume that $(q,n)$ are such that there exists a $J  \in \mathbb{N}$ such that $J(q+1)=n$. 
 	   	
 	   	We now show that the sequence $(B_{2j})_{j=0,...}$ is an independent sequence. Fix a $k \in \{ 0,... \}$ and take $U_{2kq+q}$ (the last random component of $B_{2k}$) and compare it with $U_{2(k+1)q+1}$ (the first random component of $B_{2(k+1)}$). By construction, $U_{2kq+q}$ is measurable with respect to the $\sigma$-algebra generated by $\varepsilon_{2kq+q}, \varepsilon_{2kq+q-1},...,\varepsilon_{2kq+q-q}$. Similarly, $U_{2(k+1)q+1}$ is measurable with respect to the $\sigma$-algebra generated by $\varepsilon_{2(k+1)q+1}, \varepsilon_{2(k+1)q+1-1},...,\varepsilon_{2(k+1)q+1-q}$. Observe that $\varepsilon_{2(k+1)q+1-q} = \varepsilon_{2kq+2q+1-q} = \varepsilon_{2kq+q+1}$. That is, the ``oldest" innovation of the first element of $B_{2(k+1)}$ is $\varepsilon_{2kq+q+1}$ and is independent from the ``latest" innovation of the last element of $B_{2k}$, which is $\varepsilon_{2kq+q}$. So $U_{2(k+1)q+1}$ and $U_{2kq+q}$ are independent. This implies that any other two elements of $B_{2k}$ and $B_{2(k+1)}$ are independent too. 
 	   	
 	   	Analogously, we can show that $(B_{2j+1})_{j=0,...}$ is an independent sequence. Consequently, we can decomposed $n^{-1} \sum_{i=1}^{n} x_{i} U_{i}$ into two summands $ n^{-1} \sqrt{q} \sum_{j=0}^{[J/2]} \Delta_{2j}$ and $n^{-1} \sqrt{q} \sum_{j=0}^{[(J+1)/2]-1} \Delta_{2j+1}$ with $\Delta_{j} = q^{-1/2} \sum_{i=jq+1}^{jq+q} x_{i} U_{i}$. This is the same decomposition we did in the proof of Lemma \ref{lem:bere} but with the caveat that in this case --- due to the particular MA structure --- there is no need to invoke another dataset given by $Z^{\ast}$. 
 	   	
 	   	Hence, in order to bound $\mathbf{P} \left(  | n^{-1} \sum_{i=1}^{n} x_{i} U_{i} | \geq u  \right)$ it suffices to bound $\mathbf{P} \left(  | n^{-1} \sqrt{q} \sum_{j=0}^{[J/2]} \Delta_{2j} | \geq u  \right)$ and $\mathbf{P} \left(  | n^{-1} \sqrt{q} \sum_{j=0}^{[(J+1)/2]-1} \Delta_{2j+1} | \geq u  \right)$. For each of these terms the proofs proceed in the same fashion as that of Proposition \ref{pro:OLS-q} and it is thus omitted.

   		\section{Bounds for the Suprema of Gaussian Processes}
   		\label{app:sup-Gau}

   		\subsection{Upper Bounds for $\Gamma_{k}(I_{n,k}(\omega)(r))$} 
   		
   		From the Proposition \ref{pro:M-rate1} (and Theorem \ref{thm:effe-n}), we see that the relevant notion of complexity of $\Theta_{k}(M_{n,k})$ (or rather $I_{n,k}(\omega)$) is the expectation of the supremum of Gaussian process of the form presented in display \ref{eqn:Gauss}. In this section we provided some insights so as to further control this quantity when $Pen$ is of the form $|| \cdot ||^{a}_{\ell^{a}(p)}\equiv \sum_{j=1}^{k} p_{j} |\theta_{j}|^{a}$ for $a \geq 1$ and $(p_{j})_{j}$  positive real-valued non-decreasing sequence.\footnote{For $a=\infty$, $|| \theta ||^{a}_{\ell^{a}(p)} = \max_{1 \leq j \leq k} | \theta_{j}| $, for any $\theta \in \Theta_{k}$.}
   		
   		The ideas is as follows: By exploiting Gaussianity and the linearity of the process $\theta \mapsto  \sum_{j=1}^{k} \zeta_{j} \sqrt{b_{j}}  \theta_{j} $ , it turns out that one can get good bounds on $\Gamma_{k}(I_{n,k}(\omega)(r))$ under distances of the form $|| \theta ||^{a}_{\ell^{a}(p)} $; see Lemmas \ref{lem:M-rate1b} - \ref{lem:lp-G}. Since $\delta_{k,\mathbf{P}}(\cdot,\nu_{k}(\mathbf{P}))$ will (in principle) not coincide with Banach norms, e.g., $|| \cdot ||_{\ell^{a}(p)}$, we need to control the relative behavior between these distances. 
   		

   		\begin{assumption}
   			\label{ass:Pen-bound}
   			$Pen(\theta) = \sum_{j=1}^{k} p_{j} |\theta_{j}|^{a}$ for $a \geq 1$ and $(p_{j})_{j}$  positive real-valued non-decreasing sequence. 
   		\end{assumption}
   		
   		\begin{remark}
   			It is straightforward to extend the results below to cases where one "mixes-and-matches" $\ell^{p}$ norms, e.g. $Pen(.) = ||.||^{a}_{\ell^{a}} + ||.||^{b}_{\ell^{b}} $ for some $a,b \geq 1$.   $\triangle$
   		\end{remark}
   		
   		%
   		%

   			For some $v_{k}(\mathbf{P}) \in \nu_{k}(\mathbf{P})$, and for any $A \subseteq \Theta_{k}$ and any $a' \in \mathbb{N}$ and $p' = (p'_{j})_{j=1}^{k}  \in \mathbb{R}^{k}_{+}$, let\begin{align*}
   			\Xi_{k}(a',p',A) \equiv  \sup_{\theta \in A}  \frac{||\theta-v_{k}(\mathbf{P})||_{\ell^{a'}(p')}}{\delta_{k,\mathbf{P}}(\theta,\nu_{k}(\mathbf{P})) }.    
   			\end{align*} 
   			
   			Also, let for any  $p = (p_{j})_{j=1}^{k}  \in \mathbb{R}^{k}_{+}$
   			\begin{align*}
   			C_{k,a}(p) = \left\{ \begin{array}{ll}
   			\sum_{j=1}^{k} \sqrt{b_{j}}  & a = \infty\\
   			(E[|\zeta_{1}|^{\frac{a}{a-1}}])^{\frac{a-1}{a}} \left(  \sum_{j=1}^{k} b_{j}^{0.5 \frac{a}{a-1}}  p^{-\frac{1}{a-1}}_{j} \right)^{\frac{a-1}{a}} & a \in (1,\infty)\\
   			\sqrt{ \max_{1 \leq j \leq k}  b_{j}  p^{-2}_{j} } \sqrt{ \log 2k} & a = 1
   			\end{array}   \right.
   			\end{align*}
   		
   	 	 	 		 	  		\begin{proposition}
   	 	 	 		 	  			\label{pro:Pen-bound}
   	 	 	 		 	  			Suppose Assumption \ref{ass:Pen-bound} holds. Then, for any $(n,k) \in \mathbb{N}^{2}$, $1 \leq a',a \leq \infty $, $(p',p) \in \mathbb{R}_{+}^{2k}$, and $r > 0$,
   	 	 	 		 	  			\begin{align*}
   	 	 	 		 	  			\Gamma_{k}(I_{n,k}(\omega)(r)) \leq  \min \left\{ C_{k,a'}(p')  \Xi_{k}(a',p',I_{n,k}(\omega)(r)) r ,   C_{k,a}(p) \left( \frac{M_{n,k}}{\lambda_{k}} \right)^{1/a_{0}} \right\},
   	 	 	 		 	  			\end{align*}
   	 	 	 		 	  			a.s.-$\mathbf{P}$, with $a_{0} = a$ if $a<\infty$ and $a_{0}=1$ if $a=\infty$.

   	 	 	 		 	  		\end{proposition}

   		
   			To prove the Proposition \ref{pro:Pen-bound} we need the next lemma, which provides an upper bound for $E \left[\sup _{\tau \in A} \sum_{j=1}^{k} \zeta_{j} \sqrt{b_{j}} \tau_{j} \right]$, where $(b_{j})_{j}$ is a positive real-valued sequence (the proof of this Lemma is relegated to the end of the section).

   		\begin{lemma}
   			\label{lem:M-rate1b}
   			For any $A \subseteq \Theta_{k}$, any $(p_{j})_{j}$ real-valued positive sequence, any $k \in \mathbb{N}$, and any $q \in \mathbb{N} \cup \{ \infty \}$.
   			\begin{align}
   			E \left[\sup _{\tau \in A} \sum_{j=1}^{k} \zeta_{j} \sqrt{b_{j}}  \tau_{j} \right] \leq C_{k,q}(p) \sup_{\tau \in A} ||  \tau ||_{\ell^{q}(p)}.
   			\end{align}
   			
   		\end{lemma}

   		\begin{proof}[Proof of Proposition \ref{pro:Pen-bound}]
   			The  proof relies on Lemma \ref{lem:M-rate1b}. 	We first note that 	   	
   			\begin{align*}
   			E \left[\sup _{\theta \in A} \sum_{j=1}^{k} \zeta_{j} \sqrt{b_{j}} \theta_{j} \right] =& E \left[\sup _{\theta \in A} \sum_{j=1}^{k} \zeta_{j} \sqrt{b_{j}}  \theta_{j} \right] - E \left[\sum_{j=1}^{k} \zeta_{j} \sqrt{b_{j}}  v_{k,j}(\mathbf{P}) \right]\\
   			= & E \left[\sup _{\theta \in A} \sum_{j=1}^{k} \zeta_{j} \sqrt{b_{j}}  \left\{ \theta_{j} -  v_{k,j}(\mathbf{P}) \right\} \right]
   			\end{align*}
   			because $E \left[\sum_{j=1}^{k} \zeta_{j} \sqrt{b_{j}}   v_{k,j}(\mathbf{P}) \right] = \sum_{j=1}^{k} E \left[ \zeta_{j} \sqrt{b_{j}}  \right]v_{k,j}(\mathbf{P}) = 0 $.

   			Thus 
   			\begin{align*}
   			\Gamma_{k}(I_{n,k}(\omega)(r)) = E \left[\sup _{\theta \in I_{n,k}(\omega)(r)} \sum_{j=1}^{k} \zeta_{j} \sqrt{b_{j}}  \left\{ \theta_{j} -  v_{k,j}(\mathbf{P}) \right\} \right].
   			\end{align*}
   			
   			We note that $I_{n,k}(\omega)(r) \subseteq \{  \theta \mid \delta_{k,\mathbf{P}} (\theta , v_{k}(\mathbf{P})) \leq r   \} \subseteq \{  \theta \mid   \frac{ ||\theta - v_{k}(\mathbf{P}) ||_{\ell^{a'}(p')}}{\Xi_{k}(a',p',I_{n,k}(\omega)(r))}   \leq r   \} $. Therefore, 
   			\begin{align*}
   			\Gamma_{k}(I_{n,k}(\omega)(r)) \leq  r \Xi_{k}(a',p',I_{n,k}(\omega)(r)) E \left[\sup _{\theta \in ||\theta - v_{k}(\mathbf{P}) ||_{\ell^{a'}(p')} \leq 1 } \sum_{j=1}^{k} \zeta_{j} \sqrt{b_{j}}  \left\{ \theta_{j} -  v_{k,j}(\mathbf{P}) \right\} \right].
   			\end{align*}

   			By invoking Lemma \ref{lem:M-rate1b} with $\tau=\theta - v_{k}(\mathbf{P})$, $p'$ acting as $p$ and $a'=q$ in the Lemma  and $A = \{\theta \in \Theta \mid  ||\theta - v_{k}(\mathbf{P}) ||_{\ell^{a'}(p')} \leq 1 \} $, it follows that one possible upper bound is given by
   			\begin{align*}
   			\Gamma_{k}(I_{n,k}(\omega)(r)) 	\leq  C_{k,a'}(p') \Xi_{k}(a',p',I_{n,k}(\omega)(r)) r.
   			\end{align*}

   			%

   			%
   			%

   			Another possible upper bound is directly given by 
   			\begin{align*}
   			\Gamma_{k}(I_{n,k}(\omega)(r)) \leq & \Gamma_{k}(\Theta_{k}(M_{n,k})) \\
   			\leq & C_{k,a}(p) \sup_{\theta \in \Theta_{k}(M_{n,k})} || \theta  ||_{\ell^{a}(p)}\\
   			\leq & C_{k,a}(p) \left(  M_{n,k} / \lambda_{k}  \right)^{1/a_{0}}
   			\end{align*}
   			where the first line follows from definition of $I_{n,k}(\omega)(r) \subseteq \Theta_{k}(M_{n,k})$; the second follows by Lemma \ref{lem:M-rate1b} with $\tau=\theta$, $q=a$ and $A = \Theta_{k}(M_{n,k})$; and the last follows from the fact that, under Assumption \ref{ass:Pen-bound}, $Pen(\theta) = || \theta ||^{a}_{\ell^{a}(p)} \leq M / \lambda_{k}$ for all $\theta \in \Theta_{k}(M)$.

   			Since both upper bounds are valid, we take the minimum of both and the desired result follows.
   		\end{proof}
   		
   		
   		\begin{remark}\label{rem:Gamma-bound}
   			In cases where there exists a finite constant $M$ such that $\sup_{\theta \in \Theta} ||\theta||_{\ell^{a}(p)} \leq M$, then the proposition can be easily adapted to accommodate this case and yield sharper bounds, of the type: 	\begin{align*}
   			\Gamma_{k}(I_{n,k}(\omega)(r)) \leq \min \left\{  C_{k,a'}(p') \Xi_{k}(a',p',I_{n,k}(\omega)(r)) r  ,    C_{k,a}(p) \left(M \right)^{1/a_{0}} \right\}.
   			\end{align*}
   			$\triangle$
   		\end{remark}

   		We conclude the section with a few remarks regarding $\Xi_{k}(a',p',I_{n,k}(\omega)(r))$. The quantity $\Xi_{k}(a',p',I_{n,k}(\omega)(r))$ measures the discrepancy between $\delta_{k,\mathbf{P}}$, which is the \textquotedblleft natural" distance for the criterion function, and a Banach norm $||.||_{\ell^{a'}(p')}$. The problem is that there could exist elements of $\Theta_{k}(M_{n,k})$ for which the former distance is small, but the latter is not. In a non-/semi-parametric conditional moment context, this quantity is analogous to the \emph{(sieve) measure of ill-posedness} (see \cite{BCK-2007} and \cite{CP-2012}) and in high-dimensional problems with sparsity is related (although not the same) to the so-called \emph{compatibility constant} (see \cite{VdG-Buhlmann11} and references therein). 
   		
   		\subsection{Lower bound for $\delta^{2}_{k,\mathbf{P}}$}
   		\label{app:delta-lower}
   		
   		The next assumption provides a lower bound for $\delta^{2}_{k,\mathbf{P}}$ which is natural in settings where $Q(\cdot,\mathbf{P})$ is convex and twice continuously differentiable, and $Pen$ is convex.
   		\begin{assumption}
   			\label{ass:delta-lower}
   			$Pen$ is convex and there exists a real valued non-negative sequence $(\underline{e}_{n,k})_{n,k}$ such that: For all $(n,k) \in \mathbb{N}^{2}$,
   			\begin{align*}
   			\delta^{2}_{k,\mathbf{P}}(\theta,\nu_{k}(\mathbf{P})) \geq \underline{e}_{n,k} ||\theta - \nu_{k}(\mathbf{P}) ||^{2}_{\ell^{2}} + \lambda_{k} \Delta(\theta,\nu_{k}(\mathbf{P})),~\forall \theta \in \Theta_{k}(M_{n,k}),
   			\end{align*}
   			where $ 0 \leq \Delta(\theta,\nu_{k}(\mathbf{P})) \leq Pen(\theta) - Pen(\nu_{k}(\mathbf{P})) - \mu^{T} (\theta - \nu_{k}(\mathbf{P})) $ for all $\mu \in \partial Pen(\nu_{k}(\mathbf{P}))$.
   		\end{assumption}
   		
   		When $Pen$ is differentiable, $\Delta$ can be taken to be the Bergman divergence (see \cite{1459065} and reference therein). For $Pen = ||.||_{\ell^{1}}$ one can simply take $\Delta = 0$ whereas  for $Pen = ||.||^{2}_{\ell^{2}(p)}$ one can take $\Delta = ||.||^{2}_{\ell^{2}(p)}$.
   		
   		
   		Under this assumption, 
   		\begin{align}\label{eqn:UB-Xi}
   		\Xi_{k}(a',p',I_{n,k}(\omega)(r))  \leq & \Xi_{k}(a',p',\Theta_{k}(M_{n,k})) \\
   		\leq &  \sup_{\theta \in  \Theta_{k}(M_{n,k}) }  \frac{||\theta-\nu_{k}(\mathbf{P})||_{\ell^{a'}(p')}}{\sqrt{\underline{e}_{n,k} ||\theta - \nu_{k}(\mathbf{P}) ||^{2}_{\ell^{2}} + \lambda_{k} \Delta(\theta,\nu_{k}(\mathbf{P}))}}.    
   		\end{align}

   		\subsection{On the Bounds of Lemma \ref{lem:M-rate1b}}
   		
   		The next Lemma shows that for certain sets $A$, the bounds of Lemma \ref{lem:M-rate1b} are sharp for any $q < \infty$, up to universal constants; for $q=\infty$ we loose a $k^{-1}$ which in some cases --- non-/semi-parametric models --- it could diverge. In particular, this lemma extends the results in \cite{talagrand2014} Ch. 2 to more general sets than ellipsoids.
   		
   		\begin{lemma}
   			\label{lem:M-rate1b-low}
   			For $A = \{ \tau \in \mathbb{R}^{k} \mid ||\tau||_{\ell^{q}(p)} \leq M^{1/q}   \}$:
   			
   			(1) For $q=\infty$
   			\begin{align}
   			E \left[\sup _{\tau \in A} \sum_{j=1}^{k} \zeta_{j} \sqrt{b_{j}} \tau_{j} \right] \geq \sup_{\tau \in A} ||  \tau ||_{\ell^{\infty}}  k^{-1} C_{k,q}(p)
   			\end{align}
   			
   			(2) For $q \in (1,\infty)$ 
   			\begin{align}
   			E \left[\sup _{\tau \in A} \sum_{j=1}^{k} \zeta_{j} \sqrt{b_{j}} \tau_{j} \right] \geq K C_{k,q}(p) \sup_{\tau \in A} ||  \tau ||_{\ell^{q}(p)}
   			\end{align}
   			with $K = \left(\frac{1}{L+1} \right)^{\frac{q-1}{q}}<1$ and $L = 2 \frac{q}{q-1} 2^{\frac{q}{2(q-1)}} \Gamma(\frac{q}{2(q-1)}) $.
   			
   			(3) For $q = 1$
   			\begin{align}
   			E \left[\sup _{\tau \in A} \sum_{j=1}^{k} \zeta_{j} \sqrt{b_{j}} \tau_{j} \right] \geq (\max_{1 \leq j \leq k} b_{j}^{-0.5} p_{j})^{-1} (1 - e^{-1}) \sqrt{c^{-1}_{2} \log  c_{1} k  } \sup_{\tau \in A} ||  \tau ||_{\ell^{1}(p)}
   			\end{align}
   			
   		\end{lemma}

   		\subsection{Proofs of Lemmas \ref{lem:M-rate1b} and \ref{lem:M-rate1b-low}}
   		
   		Throughout this section, let $\zeta_{j} \sim iid-N(0,1)$ for any $j=1,2,...$. Recall that $||\cdot ||^{q}_{\ell^{q}(p)} = \sum_{j} p_{j} |\cdot |^{q}$ for any $q \in \mathbb{N} \cup \{ \infty \}$ and $p_{j} \in \mathbb{R}_{+}$ for each $j = 1,2,...$. 	In what follows, fix any $(p_{j})_{j}$ real-valued positive sequence, any $k \in \mathbb{N}$, and any $q \in \mathbb{N} \cup \{ \infty \}$.

   		The proof of Lemma \ref{lem:M-rate1b} requires the following results. The right inequality in the first lemma is a minor modification of Lemma A.2 in \cite{Chatterjee-Jafarov-2015} (their proofs are relegated to the end of the section).

   		\begin{lemma}
   			\label{lem:max-G}
   			Let $\zeta_{j} \sim i.i.d.N(0,1)$ and $(p_{j})_{j}$ a positive real valued sequence. Then
   			\begin{align*}
   			(\max_{1 \leq j \leq k} p_{j})^{-1} (1 - e^{-1}) \sqrt{c^{-1}_{2} \log  c_{1} k  }  \leq E \left[  \max_{1 \leq j \leq k} | p^{-1}_{j} \zeta_{j} |   \right]  \leq (\min_{1 \leq j \leq k}   p_{j})^{-1} \sqrt{ 0.5 \log 2k}.
   			\end{align*}
   			with \begin{align*}
   			c_{1} &\equiv \left(\frac{e^{(\pi+2)^{-1}}}{4} \sqrt{\frac{\pi+2}{\pi}} \right) \\
   			c_{2} &\equiv 1
   			\end{align*}
   		\end{lemma}
   		
   		\begin{lemma}
   			\label{lem:lp-G}
   			Let $\zeta = (\zeta_{1},...,\zeta_{k})$ with $\zeta_{j} \sim i.i.d.N(0,1)$ and $(p_{j})_{j}$ a positive real valued sequence. Then, for any $q \geq 1$,
   			\begin{align*}
   			E[|\zeta_{1}|]  \left( \sum_{j=1}^{k} p_{j} \right)^{1/q} \leq 	E \left[  || \zeta ||_{\ell^{q}(p)}   \right]  \leq  (E[|\zeta_{1}|^{q}])^{1/q} \left( \sum_{j=1}^{k} p_{j} \right)^{1/q}
   			\end{align*}
   		\end{lemma}

   		\begin{proof}[Proof of Lemma \ref{lem:M-rate1b}]   	
   			Let $\xi_{j} \equiv \zeta_{j} \sqrt{b_{j}}$ for all $j$.

   			(1)-(2) By Holder inequality
   			\begin{align*}
   			\Gamma_{k}(A) \leq & E[\sup_{\tau \in A} \sum_{j=1}^{k} | p^{-1/q_{2}}_{j} \xi_{j}  p^{1/q_{2}}_{j} ( \tau_{j}  ) | ] \\
   			\leq &  E \left [\sup_{\tau \in A} \left( \sum_{j=1}^{k} p^{-q_{1}/q_{2}}_{j} |\xi_{j}|^{q_{1}} \right)^{1/q_{1}}  \left( \sum_{j=1}^{k} p_{j} |\tau_{j} |^{q_{2}} \right)^{1/q_{2}} \right]\\
   			= & \sup_{\tau \in A} \left( \sum_{j=1}^{k} p_{j} |\tau_{j} |^{q_{2}} \right)^{1/q_{2}} E \left [ ||\zeta ||_{\ell^{q_{1}}(b^{0.5q_{1}} \cdot p^{-q_{1}/q_{2}})} \right]
   			\end{align*}
   			for any $q_{i} \geq 1$ and $q^{-1}_{1} + q^{-1}_{2} = 1$. By lemma \ref{lem:lp-G} (the $p$ in the lemma is taken to be $b^{0.5 q_{1}} \cdot p^{-q_{1}/q_{2}}$ in this proof), 
   			\begin{align*}
   			\Gamma_{k}(A) \leq \sup_{\tau \in A} ||\tau ||_{\ell^{q_{2}}(p)} (E[|\zeta_{1}|^{q_{1}}])^{1/q_{1}}  \left( \sum_{j=1}^{k} b_{j}^{0.5q_{1}} p^{-q_{1}/q_{2}}_{j} \right)^{1/q_{1}}.
   			\end{align*}
   			The result (2) follows by setting $q_{2} = q$. The result (1) follows from setting $p_{j}=1$ and noting that $q_{1} = 1$ and $||\tau||_{\ell^{q_{2}}(p)} = ||\tau||_{\ell^{\infty}}$, so
   			\begin{align*}
   			\Gamma_{k}(A) \leq \sup_{\tau \in A} ||\tau ||_{\ell^{\infty}} (E[|\zeta_{1}|])  \left( \sum_{j=1} ^{k} b_{j} ^{0.5} \right);
   			\end{align*} 	             	
   			and the fact that $E[|\zeta_{1}|]=1$ for standard Gaussian. 
   			
   			\medskip
   			
   			(3) Note that   	
   			\begin{align*}
   			\Gamma_{k}(A) \leq E[\sup_{\tau \in A} \sum_{j=1}^{k} p^{-1}_{j} \xi_{j}  p_{j} \tau_{j} ] \leq E[\max_{1 \leq j \leq k} | p^{-1}_{j} \xi_{j} | ] \sup_{\theta \in A} ||  \tau ||_{\ell^{1}(p)}.
   			\end{align*}  	
   			By lemma \ref{lem:max-G} (the $p$ in the lemma is taken to be $b^{0.5} \cdot p^{-1}$ in this proof), 
   			\begin{align*}
   			E[\max_{1 \leq j \leq k} | p^{-1}_{j} \xi_{j} | ] \leq \sqrt{ \max_{1 \leq j \leq k}  b_{j} p^{-2}_{j} } \sqrt{ 0.5 \log 2k}.
   			\end{align*}
   		\end{proof}

   		\begin{proof}[Proof of Lemma \ref{lem:M-rate1b-low}]
   			
   			Let $\xi_{j} \equiv \sqrt{b_{j}} \zeta_{j}$.  Let $X_{k}(A) \equiv \sup_{\tau \in A} \sum_{j=1}^{k} \xi_{j} \tau_{j} $ for any realization of $(\zeta_{1},...,\zeta_{k})$.  Therefore, it suffices to show that      
   			\begin{align}
   			E \left[X_{k}(A) \right] \geq K C_{k,q}(p) M^{1/q}.
   			\end{align}  
   			
   			We begin by noting a fact that we use throughout the proof: 	For any realization $\zeta$ and for any $\tau \in A$, by taking $t_{j} = \tau_{j} sign(\xi_{j})$ (which also belongs to $A$) it follows that $X_{k}(A) \geq \sum_{j=1}^{k} \tau_{j} |\xi_{j}| $.

   			\medskip

   			\textsc{The Case $q=\infty$:} From the fact above $X_{k}(A) \geq \sum_{j=1}^{k} \tau_{j} |\xi_{j}|$ for any $\tau \in A$. For any given $\tau' \in A$, choose $\tau$ such that $\tau_{1} = \max_{1 \leq l \leq k} |\tau'_{l}|$ and $\tau_{j} = 0$ for all $j \ne 1$. Clearly (1) $\tau \in A$ and (2) $X_{k}(A) \geq \sum_{j=1}^{k} \tau_{j} |\xi_{j}| \geq |\zeta_{1}| \sqrt{b_{1}} \max_{1 \leq l \leq k} |\tau'_{l}|$. By taking expectations and repeating the process for $j=2$, etc, it follows that $X_{k}(A) \geq k^{-1} \sum_{j=1}^{k} \sqrt{b_{j}} \times \max_{1 \leq l \leq k} |\tau'_{l}| $. The desired result thus follows from taking the supremum over all $\tau' \in A$.

   			\medskip

   			\textsc{The Case $q \in (1,\infty)$:}\\
   			
   			\textsc{Step 1.} We now show that 
   			\begin{align*}
   			E[(X_{k}(A))^{\frac{q}{q-1}}] \geq &  (M)^{\frac{1}{q-1}} E \left[\left( \sum_{j=1}^{k} |\xi_{j}|^{\frac{q}{q-1}} p_{j}^{-\frac{1}{q-1}}   \right) \right] \\
   			= &  (M)^{\frac{1}{q-1}} E \left[ |\zeta_{1}|^{\frac{q}{q-1}} \right] \sum_{j=1}^{k} b_{j}^{0.5 \frac{q}{q-1}}   p_{j}^{-\frac{1}{q-1}}.
   			\end{align*}

   			By the fact at the beginning $X_{k}(A) \geq \sum_{j=1}^{k} \tau_{j} |\xi_{j}| $ for any $\tau \in A$. 
   			
   			Let \begin{align*}
   			\tau_{j} = M^{1/q}\frac{|p_{j}|^{-\frac{1}{q-1}} |\xi_{j}|^{\frac{1}{q-1}}}{(\sum_{j=1}^{k} |p_{j}|^{-\frac{1}{q-1}} |\xi_{j}|^{\frac{q}{q-1}})^{1/q} },~\forall j=1,...,k.
   			\end{align*}
   			
   			Note that $\sum_{j=1}^{k} p_{j} \tau_{j}^{q} = M \sum_{j=1}^{k} p_{j} \frac{p_{j}^{-\frac{q}{q-1}} |\xi_{j}|^{\frac{q}{q-1}}}{\sum_{j=1}^{k} |p_{j}|^{-\frac{1}{q-1}} |\xi_{j}|^{\frac{q}{q-1}} } = M \sum_{j=1}^{k}  \frac{|p_{j}|^{-\frac{1}{q-1}} |\xi_{j}|^{\frac{q}{q-1}}}{\sum_{j=1}^{k} p_{j}^{-\frac{1}{q-1}} |\xi_{j}|^{\frac{q}{q-1}} } = M$ and so $\tau \in A$. Also, note that\begin{align*}
   			\sum_{j=1}^{k} \tau_{j} |\xi_{j}| = & M^{1/q} \sum_{j=1}^{k} \frac{|p_{j}|^{-\frac{1}{q-1}} |\xi_{j}|^{1+\frac{1}{q-1}}}{(\sum_{j=1}^{k} |p_{j}|^{-\frac{1}{q-1}} |\xi_{j}|^{\frac{q}{q-1}})^{1/q} } \\
   			= &  M^{1/q} \sum_{j=1}^{k} \frac{|p_{j}|^{-\frac{1}{q-1}} |\xi_{j}|^{\frac{q}{q-1}}}{(\sum_{j=1}^{k} |p_{j}|^{-\frac{1}{q-1}} |\xi_{j}|^{\frac{q}{q-1}})^{1/q} }\\
   			= & M^{1/q} \left( \sum_{j=1}^{k} |p_{j}|^{-\frac{1}{q-1}} |\xi_{j}|^{\frac{q}{q-1}} \right)^{\frac{q-1}{q}}.
   			\end{align*}
   			
   			Hence\begin{align*}
   			E[(X_{k}(A))^{\frac{q}{q-1}}] \geq E \left[ M^{\frac{1}{q-1}} \left( \sum_{j=1}^{k} |p_{j}|^{-\frac{1}{q-1}} |\xi_{j}|^{\frac{q}{q-1}} \right) \right] 
   			\end{align*}
   			and the results follows.

   			\medskip

   			\textsc{Step 2.} We show that: \begin{align*}
   			E[(X_{k}(A) - E[X_{k}(A)])^{\frac{q}{q-1}}] \leq L \left( E[X_{k}(A)] \right)^{\frac{q}{q-1}},
   			\end{align*}
   			for $L = \frac{q}{q-1} 2^{\frac{q}{2(q-1)}}  \Gamma(\frac{q}{2(q-1)}) >1$.   	               	        	We do this by showing that     	     			
   			\begin{align*}
   			E[(|X_{k}(A) - E[X_{k}(A)]|)^{\frac{q}{q-1}}] \leq L \left( \sigma\right)^{\frac{q}{q-1}},
   			\end{align*}
   			with $\sigma = \sqrt{ \sup_{\tau \in A} (E[\sum_{j=1}^{k} \xi_{j} \tau_{j}])^{2} }$;  and that\begin{align}\label{eqn:lb-Gama-1}
   			\sigma = \sqrt{ \sup_{\tau \in A} (E[\sum_{j=1}^{k} \xi_{j} \tau_{j}])^{2} } \leq \sqrt{ \sup_{\tau \in A} \sum_{j=1}^{k} b_{j} \tau^{2}_{j} } \leq E[X_{k}(A)].
   			\end{align} 
   			
   			To establish the first display, note that
   			\begin{align*}
   			E \left[ \left( \frac{|X_{k}(A) - E[X_{k}(A)]|}{\sigma}   \right)^{\frac{q}{q-1}}   \right] = & \int_{0}^{\infty}  \Pr \left( |X_{k}(A) - E[X_{k}(A)]| \geq \sigma t^{\frac{q-1}{q}}  \right) dt \\
   			\leq  & 2 \int_{0}^{\infty}  \exp \{ - 0.5 \frac{(\sigma t^{\frac{q-1}{q}})^{2}}{\sigma^{2}}   \} dt = 2  \int_{0}^{\infty}  \exp \{ - 0.5  t^{2\frac{q-1}{q}}   \} dt\\
   			= &  2 \frac{q}{q-1} 2^{\frac{q}{2(q-1)}-1} \int_{0}^{\infty}  \exp \{ - w   \} w^{\frac{q}{2(q-1)}-1}dw\\
   			= & 2 \frac{q}{q-1} 2^{\frac{q}{2(q-1)}-1} \Gamma(\frac{q}{2(q-1)})
   			\end{align*}
   			where the second line follow from lemma 2.4.7 in \cite{talagrand2014}; the third follows by change of variables (observe that $\frac{q}{2(q-1)}>0$ ).
   			

   			To establish display \ref{eqn:lb-Gama-1}, we start by showing that: for any $\tau \in A$, there exists a $\tau' \in A$ such that
   			\begin{align*}
   			\sqrt{ \sum_{j=1}^{k} b_{j} \tau^{2}_{j} } \leq \sum_{j=1}^{k} \sqrt{b_{j}}  \tau'_{j}.
   			\end{align*}
   			By taking $\tau' = |\tau|$ (note that since $\tau \in A$, $|\tau| \in A$), it follows that $(\sum_{j=1}^{k} \sqrt{b_{j}} \tau'_{j})^{2} = (\sum_{j=1}^{k} \sqrt{b_{j}} |\tau_{j}|)^{2} \geq \sum_{j=1}^{k} b_{j} |\tau_{j}|^{2}$. Therefore
   			\begin{align*}
   			\sigma = \sqrt{ \sup_{\tau \in A } \sum_{j=1}^{k} b_{j} \tau^{2}_{j} }  \leq \sup_{\tau \in A } \sum_{j=1}^{k} \sqrt{b_{j}} \tau_{j}.
   			\end{align*}
   			
   			By the fact at the beginning of the proof:  $X_{k}(A) \geq \sum_{j=1}^{k} \tau_{j} |\xi_{j}| $ and thus      	     			
   			\begin{align*}
   			E[X_{k}(A)] \geq \sum_{j=1}^{k} \tau_{j} E[|\xi_{1}|] = \sum_{j=1}^{k} \sqrt{b_{j}} \tau_{j}.
   			\end{align*}

   			\medskip

   			\textsc{Step 3.} We now show our desired result. 
   			
   			By step 2, and the fact that $||a-b||_{L^{r}} \geq | ||a||_{L^{r}} - ||b||_{L^{r}}|$ for $r \geq 1$ by Minkowski inequality, it follows that\begin{align*}
   			E[(X_{k}(A))^{\frac{q}{q-1}}] \leq( L + 1) \left( E[X_{k}(A)] \right)^{\frac{q}{q-1}}.
   			\end{align*}
   			By step 1, 
   			\begin{align*}
   			(M)^{\frac{1}{q-1}}  E \left[ |\zeta_{1}|^{\frac{q}{q-1}} \right] \sum_{j=1}^{k} b_{j}^{0.5 \frac{q}{q-1}}  p_{j}^{-\frac{1}{q-1}} \leq( L + 1) \left( E[X_{k}(A)] \right)^{\frac{q}{q-1}},
   			\end{align*}
   			or 
   			\begin{align*}
   			(M)^{\frac{1}{q}} \left( E \left[ |\zeta_{1}|^{\frac{q}{q-1}} \right] \sum_{j=1}^{k}  b_{j}^{0.5 \frac{q}{q-1}}  p_{j}^{-\frac{1}{q-1}} \right) ^{\frac{q-1}{q}}\leq( L + 1)^{\frac{q-1}{q}}   E[X_{k}(A)] ,
   			\end{align*}

   			\medskip

   			\textsc{The Case $q=1$:} As it was shown at the beginning of the proof, $X_{k}(A) \geq \sum_{j=1}^{k} |\xi_{j}| \tau_{j}$ for any $\tau \in A$.  In particular, fix any $j=1,...,k$, let  $\tau_{j} = p^{-1}_{j} M$ and $\tau_{l} = 0 $ for all $ l \ne j$,  it follows that $||\tau ||_{\ell^{1}(p)} = M$ and thus is in $A$. Also $X_{k}(A) \geq |\xi_{j}| p^{-1}_{j} M $. Since this construction holds for any $j=1,..,k$,  by taking the maximum over $j$ and applying expectations it follows that 
   			\begin{align*}
   			E[X_{k}(A)] \geq E[\max_{1 \leq j \leq k }| \sqrt{b_{j}} p^{-1}_{j}\zeta_{j}|] M.
   			\end{align*} 
   			The result thus follows from Lemma \ref{lem:max-G} (the lower bound).\end{proof}

   		\subsubsection{Proofs of Supplementary Lemmas}

   		\begin{proof}[Proof of Lemma \ref{lem:max-G}]
   			We now establish the right inequality. Note that for any $\lambda>0$,
   			\begin{align*}
   			E \left[  \max_{1 \leq j \leq k}  |p^{-1}_{j} \zeta_{j} |  \right]   = & \lambda^{-1} E \left[ \log \exp \{ \lambda \max_{1 \leq j \leq k}  |p^{-1}_{j} \zeta_{j}|   \} \right] \\
   			\leq &   \lambda^{-1} \log  E \left[ \exp \{ \lambda \max_{1 \leq j \leq k} | p^{-1}_{j} \zeta_{j} |  \} \right]~(by~Jensen~inequality) \\
   			\leq &   \lambda^{-1} \log  E \left[ \sum_{1 \leq j \leq k}\exp \{ \lambda | p^{-1}_{j} \zeta_{j} |  \} \right] \\
   			= &   \lambda^{-1} \log \sum_{1 \leq j \leq k} E \left[ \exp \{ \lambda p^{-1}_{j} \zeta_{j}   \} 1\{ \zeta_{j} > 0  \} \right] + E \left[ \exp \{ - \lambda p^{-1}_{j} \zeta_{j}   \} 1\{ \zeta_{j} < 0  \} \right] \\
   			\leq & \lambda^{-1} \log  2 \sum_{1 \leq j \leq k} E \left[ \exp \{ \lambda p^{-1}_{j} \zeta_{j}   \} \right]~(by~symmetry~of~\zeta)\\
   			= & \lambda^{-1} \log  2 \sum_{1 \leq j \leq k}  \exp \{ 0.5 \lambda^{2} p^{-2}_{j}   \}~(by~Gaussianity~of~\zeta) \\
   			= & \lambda^{-1} \log 2 k + \lambda^{-1} \log k^{-1} \sum_{j=1}^{k} \exp\{ 0.5 \lambda^{2} p^{-2}_{j}   \} \\
   			\leq & \lambda^{-1} \log 2 k + \lambda^{-1} \log \max_{1 \leq j \leq k} \exp\{ 0.5 \lambda^{2} p^{-2}_{j}   \} \\
   			= & \lambda^{-1} \log 2 k + 0.5 \lambda  \max_{1 \leq j \leq k}   p^{-2}_{j}.
   			\end{align*}
   			Letting $\lambda = (\max_{1 \leq j \leq k}   p^{-2}_{j})^{-1/2} \sqrt{2 \log 2k} $, it follows that $ E \left[  \max_{1 \leq j \leq k} p^{-1}_{j} \zeta_{j}    \right] \leq \sqrt{\max_{1 \leq j \leq k}   p^{-2}_{j} 0.5 \log 2k} $. Clearly $\sqrt{\max_{1 \leq j \leq k}   p^{-2}_{j} } = \max_{1 \leq j \leq k}   p^{-1}_{j} = (\min_{1 \leq j \leq k}   p_{j})^{-1} $.
   			
   			\medskip
   			
   			We now establish the left inequality. Let $Y_{k} \equiv  \max_{1 \leq j \leq k}  p^{-1}_{j} \zeta_{j}  $. Note that, for any $a \geq 0$, 
   			\begin{align*}
   			E \left[ \max_{1 \leq j \leq k} | p^{-1}_{j} \zeta_{j} |  \right] \geq & E \left[  \max_{1 \leq j \leq k} | p^{-1}_{j} \zeta_{j} | \left \vert \max_{1 \leq j \leq k} | p^{-1}_{j} \zeta_{j} | \geq a  \right. \right] \Pr\{ \max_{1 \leq j \leq k}  | p^{-1}_{j} \zeta_{j} | \geq a  \} \\
   			\geq & a \Pr\{ \max_{1 \leq j \leq k} | p^{-1}_{j} \zeta_{j} | \geq a  \} = a \Pr\{ Y_{k} \geq a \}.
   			\end{align*}
   			Clearly it suffices to (lower) bound $a \Pr\{ Y_{k}  \geq a  \}$. Observe that
   			\begin{align*}
   			\Pr\{ Y_{k} \geq a  \} = & 1 - \prod_{1 \leq j \leq k } \Pr \{  \zeta_{1} < a p_{j} \} \\
   			\geq & 1 - (\Phi(a \max_{1 \leq j \leq k} p_{j}) )^{k},
   			\end{align*}
   			because $\Pr \{  \zeta_{1} < a p_{j} \} \leq \Pr \{  \zeta_{1} < a \max_{1 \leq j \leq k} p_{j} \}$.
   			
   			By Theorem 2.1 in \cite{arXiv1202.6483} it follows that $1-\Phi(x) \geq c_{1} \exp\{ - c_{2} x^{2}  \}$ with $c_{1}= \left(\frac{e^{((\pi(\kappa-1)+2))^{-1}}}{2\kappa} \sqrt{\frac{\kappa-1}{\pi}(\pi(\kappa-1)+2)} \right)$ and $c_{2}= \frac{\kappa}{2}$ for any $\kappa \geq 1$ (in particular, we use $\kappa =2$). Thus \begin{align*}
   			\Phi(a \max_{1 \leq j \leq k} p_{j}) \leq 1-c_{1} \exp\{ - c_{2} (a \max_{1 \leq j \leq k} p_{j})^{2}  \}.
   			\end{align*}
   			
   			Thus\begin{align*}
   			E \left[  Y_{k}  \right] \geq a \left[ 1 - (1-c_{1} \exp\{ - c_{2} (a \max_{1 \leq j \leq k} p_{j})^{2}  \})^{k} \right]
   			\end{align*}	 	  
   			taking $a = (\max_{1 \leq j \leq k} p_{j})^{-1} \sqrt{\frac{ \log  c_{1} k }{c_{2}} }$, it follows that
   			\begin{align*}
   			E \left[  Y_{k}  \right] \geq (\max_{1 \leq j \leq k} p_{j})^{-1} \sqrt{\frac{ \log  c_{1} k }{c_{2}} } (1 - (1-k^{-1})^{k}).
   			\end{align*}
   			The function $k \mapsto (1-k^{-1})^{k}$ is increasing an bounded by $e^{-1}$ and thus the desired result follows.

   		\end{proof}

   		\begin{proof}[Proof of Lemma \ref{lem:lp-G}]

   			\begin{align*}
   			E \left[  || \zeta ||_{\ell^{q}(p)}   \right] = & E \left[  \left( \sum_{j=1}^{k} p_{j} | \zeta_{j}|^{q} \right)^{1/q}  \right] \\
   			\leq & \left( \sum_{j=1}^{k} p_{j} E \left[   | \zeta_{j}|^{q}  \right]  \right)^{1/q}~(by~Jensen~inequality) \\
   			= &  (E \left[   | \zeta_{1}|^{q}  \right])^{1/q}  \left( \sum_{j=1}^{k} p_{j} \right)^{1/q}.
   			\end{align*}  
   			
   			Also, 
   			\begin{align*}
   			E \left[  || \zeta ||_{\ell^{q}(p)}   \right] = & E \left[  \left( \sum_{j=1}^{k} p_{j} | \zeta_{j}|^{q} \right)^{1/q}  \right] \\
   			\geq & \left( \sum_{j=1}^{k} p_{j} \right)^{1/q} E \left[ \sum_{j=1}^{k} \frac{p_{j}}{ \sum_{j=1}^{k} p_{j}}    | \zeta_{j}|    \right] ~(by~Jensen~inequality) \\
   			= &  E \left[   | \zeta_{1}|  \right] \left( \sum_{j=1}^{k} p_{j} \right)^{1/q}.
   			\end{align*}  
   		\end{proof}

   				\section{Proofs of Supplementary Lemmas in Appendix \ref{app:sec-main-prop}}
   				\label{app:supp-lem-bound-H}
   				
   				Recall that $\mathbb{B}_{r}(q) = \sqrt{2} \left( \int_{0}^{1} |\mu_{q}(u)|^{\frac{r}{r-2}}    \right)^{(r-2)/(2r)}$, for any $ r> 2$. 
   				
   				\begin{lemma}
   					\label{lem:bound-normq}
   					For any $ r> 2$, any $f \in \mathcal{F}(\Theta)$, $||f||_{q} \leq \mathbb{B}_{r}(q) ||f||_{L^{r}(\mathbf{P})}$ for all $q \in \mathbb{N}$, with \footnote{Recall that the function $\beta^{-1} : [0,1] \rightarrow \mathbb{R}_{+}$ is defined as $\beta^{-1}(u) = \min \{ s \mid \beta(s) \leq u   \}$ --- since $\beta$ is right-continuous, the min exists ---. }
   					\begin{align*}
   					\mathbb{B}_{r}(q) \leq \sqrt{2}  \left( 0.5 \left( \int_{\beta(q)}^{1} (\beta^{-1}(u)+1)^{\frac{r}{r-2}} du + (q+1)^{\frac{r}{r-2}} \beta(q)  \right)  \right)^{(r-2)/(2r)}
   					\end{align*}
   					and 
   					\begin{align*}
   					\mathbb{B}_{r}(q) \geq \sqrt{2} \left( 0.5 \left( \int_{\beta(1+q)}^{1} (\beta^{-1}(u) )^{\frac{r}{r-2}} du + (q+1)^{\frac{r}{r-2}} \beta(1+q)  \right)  \right)^{(r-2)/(2r)}
   					\end{align*}
   				\end{lemma}
   				

   				\begin{proof}[Proof of Lemma \ref{lem:bound-normq}]
   					By definition $||f||_{q}^{2} = 2\int_{0}^{1} \mu_{q}(u) Q_{f}(u)^{2} du$. By Cauchy-Swarchz inequality \begin{align*}
   					\int_{0}^{1} \mu_{q}(u) Q_{f}^{2}(u) du \leq \left( \int_{0}^{1} |\mu_{q}(u)|^{q}    \right)^{1/q} \left( \int_{0}^{1} |Q_{f}(u)|^{2p}    \right)^{1/p}  
   					\end{align*}
   					with $1/q + 1/p = 1$. Letting $p = r/2$, it follows that 
   					\begin{align*}
   					2\int_{0}^{1} \mu_{q}(u) Q_{f}^{2}(u) du \leq & 2\left( \int_{0}^{1} |\mu_{q}(u)|^{\frac{r}{r-2}}    \right)^{(r-2)/r} \left( \int_{0}^{1} |Q_{f}(u)|^{r}    \right)^{2/r}  \\
   					= & 2 \left( \int_{0}^{1} |\mu_{q}(u)|^{\frac{r}{r-2}}    \right)^{(r-2)/r} ||f||^{2}_{L^{r}(\mathbf{P})}.
   					\end{align*}

   					The final bounds follow from Lemma \ref{lem:q-norm}(3)(4) with $a = \frac{r}{r-2}$. 
   				\end{proof}	
   				
   				\begin{proof}[Proof of Lemma \ref{lem:bound-B}]
   					
   					Throughout the proof, we use the upper and lower bounds of $\mathbb{B}_{r}(q)$ provided by Lemma \ref{lem:bound-normq}. \\
   					
   					(1) For this case, $\beta^{-1}(u) = m^{-1}_{0}$. And thus \begin{align*}
   					\mathbb{B}_{r}(q) \leq  0.5^{\frac{r-2}{2r}} \sqrt{2} \left( (m_{0}^{-1} + 1)^{\frac{r}{r-2}}(1 - 1\{ q < m^{-1}_{0} \}) + (q+1)^{\frac{r}{r-2}} 1\{ q < m^{-1}_{0} \}   \right)^{\frac{r-2}{2r}}.
   					\end{align*}
   					So, after simple algebra it follows that $\mathbb{B}_{r}(q) \leq (0.5)^{(r-2)/(2r)} \sqrt{2} \sqrt{\min \{ 1+q,1+m^{-1}_{0} \}}$.
   					
   					Also, 
   					\begin{align*}
   					\mathbb{B}_{r}(q)	 \geq  0.5^{\frac{r-2}{2r}} \sqrt{2} \left( (m_{0}^{-1})^{\frac{r}{r-2}}(1 - 1\{ 1+q < m^{-1}_{0} \}) + (1+q)^{\frac{r}{r-2}} 1\{ 1+q < m^{-1}_{0} \}   \right)^{\frac{r-2}{2r}}.
   					\end{align*}
   					So $\mathbb{B}_{r}(q) \leq (0.5)^{(r-2)/(2r)} \sqrt{2} \sqrt{\min \{ 1+q, m^{-1}_{0} \}}$. Note also that $0.5^{\frac{r-2}{2r}} \sqrt{2} = \sqrt{2^{1-(r-2)/r}} = 2^{1/r}$.

   					%
   					%
   					%
   					
   					\medskip 
   					
   					\underline{The upper bound for cases (2) and (4):} For all these cases $\beta^{-1}(u) = (u)^{-1/m_{0}} - 1$, so
   					\begin{align*}
   					\mathbb{B}_{r}(q) \leq  & 2^{1/r} \left( \int_{\beta(q)}^{1} (u)^{-\frac{r}{m_{0}(r-2)}} du + (1+q)^{\frac{r}{r-2}} (1+q)^{-m_{0}}    \right)^{(r-2)/(2r)}\\
   					= & 2^{1/r}\left( \frac{1 - (1+q)^{\frac{r}{(r-2)}-m_{0}}}{1-\frac{r}{m_{0}(r-2)}} + (1+q)^{\frac{r}{r-2} - m_{0}}    \right)^{(r-2)/(2r)}\\ 
   					= & 2^{1/r} \left( \frac{1}{m_{0}-\frac{r}{r-2}} \left( m_{0} - \frac{r}{r-2} (1+q)^{\frac{r}{r-2} - m_{0}}  \right)    \right)^{(r-2)/(2r)}.
   					\end{align*}
   					For case (2) $m_{0} > \frac{r}{r-2}$, so an upper bound for the RHS in the display is $2^{1/r} \left( \frac{m_{0}}{m_{0}-\frac{r}{r-2}} \right)^{(r-2)/(2r)}$. For the case (4), 				$m_{0} < \frac{r}{r-2}$, so an upper bound for the RHS in the display is $2^{1/r} \left( \frac{\frac{r}{r-2}}{\frac{r}{r-2}-m_{0}} \right)^{(r-2)/(2r)}  (1+q)^{\frac{r -(r-2)m_{0}}{2r}}$.
   					
   					\medskip 
   					
   					\underline{The upper bound for case (3):} For this case
   					\begin{align*}
   					\mathbb{B}_{r}(q) \leq   2^{1/r} \left( \int_{\beta(q)}^{1} (u)^{-1} du + 1    \right)^{(r-2)/(2r)}	=  2^{1/r}\left( m_{0} \log (1+q) + 1  \right)^{(r-2)/(2r)}.
   					\end{align*}
   					
   					\medskip 
   					
   					\underline{The lower bound for cases (2) and (4):} We first observe that for any $0 \leq u \leq \beta(1) = 2^{-m_{0}}$, then $u^{-1/m_{0}} \geq 2$. Therefore $((u)^{-1/m_{0}} - 1) \geq 0.5 (u)^{-1/m_{0}} $ for all $0 \leq u \leq \beta(1) = 2^{-m_{0}}$. Hence
   					\begin{align*}
   					\mathbb{B}_{r}(q) \geq  & 2^{1/r} \left(  \int_{\beta(1+q)}^{1} (u^{-1/m_{0}} - 1)^{\frac{r}{r-2}} du + (1+q)^{\frac{r}{r-2}} (2+q)^{-m_{0}}    \right)^{\frac{r-2}{2r}} \\
   					\geq & 2^{1/r} \left(  \int_{\beta(1+q)}^{\beta(1)} (u^{-1/m_{0}} - 1)^{\frac{r}{r-2}} du + (1+q)^{\frac{r}{r-2}} (2+q)^{-m_{0}}    \right)^{\frac{r-2}{2r}} \\
   					\geq & 2^{1/r} \left( 0.5^{\frac{r}{r-2}} \int_{\beta(1+q)}^{\beta(1)} u^{-\frac{r}{m_{0}(r-2)}} du + (1+q)^{\frac{r}{r-2}} (2+q)^{-m_{0}}    \right)^{\frac{r-2}{2r}} \\	
   					= & 2^{1/r} \left( 0.5^{\frac{r}{r-2}} \frac{m_{0}}{m_{0} - \frac{r}{r-2}}\left( 2^{\frac{r}{r-2} - m_{0}} - (2+q)^{\frac{r}{r-2} - m_{0}}  \right) + (1+q)^{\frac{r}{r-2}} (2+q)^{-m_{0}}    \right)^{\frac{r-2}{2r}}\\
   					= & 2^{1/r} \left( \frac{2^{- m_{0}}  m_{0}}{m_{0} - \frac{r}{r-2}} + (2+q)^{\frac{r}{r-2}- m_{0}}  \left( \left( \frac{1+q}{2+q}\right)^{\frac{r}{r-2}} -  0.5^{\frac{r}{r-2}} \frac{m_{0}}{m_{0} - \frac{r}{r-2}}  \right) \right)^{\frac{r-2}{2r}}\\					 							
   					\end{align*}
   					
   					For case (2), $m_{0} > \frac{r}{r-2}$, then $\frac{1+q}{2+q} \geq 0.5$ and thus the RHS is bounded below by \begin{align*}
   					& 2^{1/r} \left( \frac{2^{- m_{0}}  m_{0}}{m_{0} - \frac{r}{r-2}} + (2+q)^{\frac{r}{r-2}-m_{0}} 0.5^{\frac{r}{r-2}}  \left( 1 -   \frac{m_{0}}{m_{0} - \frac{r}{r-2}}  \right) \right)^{\frac{r-2}{2r}}  \\
   					= & 2^{1/r} \left( \frac{2^{- m_{0}}  m_{0}}{m_{0} - \frac{r}{r-2}} - (2+q)^{\frac{r}{r-2}-m_{0}}   \left(  \frac{2^{-\frac{r}{r-2}}\frac{r}{r-2}}{m_{0} - \frac{r}{r-2}}  \right) \right)^{\frac{r-2}{2r}}\\
   					\geq &     2^{1/r} \left( \frac{2^{- m_{0}}  m_{0}}{m_{0} - \frac{r}{r-2}} - (2)^{\frac{r}{r-2}-m_{0}}   \left(  \frac{2^{-\frac{r}{r-2}}\frac{r}{r-2}}{m_{0} - \frac{r}{r-2}}  \right) \right)^{\frac{r-2}{2r}}\\
   					= &  2^{1/r} \left( 2^{- m_{0}}\right)^{\frac{r-2}{2r}}.
   					\end{align*} 
   					For case (4), $m_{0} < \frac{r}{r-2}$, by similar arguments, a lower bound is given by
   					\begin{align*}
   					& 2^{1/r} \left( \frac{2^{- m_{0}}  m_{0}}{m_{0} - \frac{r}{r-2}} + (2+q)^{\frac{r}{r-2}- m_{0}}  \left(  \frac{2^{-\frac{r}{r-2}}\frac{r}{r-2} }{- m_{0} + \frac{r}{r-2}} \right)  \right)^{\frac{r-2}{2r}}\\
   					= & 2^{1/r} \left( \frac{2^{- m_{0}}  m_{0}}{m_{0} - \frac{r}{r-2}} + (1+0.5q)^{\frac{r}{r-2}- m_{0}}  \left(  \frac{2^{m_{0}}\frac{r}{r-2} }{- m_{0} + \frac{r}{r-2}} \right)  \right)^{\frac{r-2}{2r}} \\
   					= & 2^{1/r} \left( \frac{2^{- m_{0}}  m_{0}}{\frac{r}{r-2}-m_{0}} ((1+0.5q)^{\frac{r}{r-2}- m_{0}} - 1)   \right)^{\frac{r-2}{2r}} \\
   					\geq & 2^{1/r} \left( \frac{2^{- m_{0}}  m_{0}}{\frac{r}{r-2}-m_{0}} ((1+0.5q)^{\frac{r}{r-2}- m_{0}} - 1)   \right)^{\frac{r-2}{2r}}.
   					\end{align*}
   					
   					\medskip 
   					
   					\underline{The lower bound for cases (3):} By analogous calculations 
   					\begin{align*}
   					\mathbb{B}_{r}(q) \geq & 2^{1/r} \left( 0.5^{\frac{r}{r-2}} \int_{\beta(1+q)}^{\beta(1)} u^{-1} du + (1+q)^{\frac{r}{r-2}} (2+q)^{-m_{0}}    \right)^{\frac{r-2}{2r}} \\	
   					= & 2^{1/r} \left( 0.5^{\frac{r}{r-2}} m_{0} (\log (2+q) - \log (2))  +1   \right)^{\frac{r-2}{2r}} \\		
   					= & 2^{1/r}  \sqrt{0.5} \left( m_{0} \log (1+0.5q) + 2^{\frac{r-2}{r}}    \right)^{\frac{r-2}{2r}}.		 							
   					\end{align*}
   				\end{proof}

 	\section{Proofs of Supplementary Lemmas in Appendix \ref{app:gbrack}}   
 	\label{app:supp-gbrack}
 	
	\begin{proof}[Proof of Lemma \ref{lem:f-decom}]
		Recall that $\Delta_{k} f  \equiv  \pi_{k}(f) - \pi_{k-1}(f)$, $\Xi_{k} = -f + \pi_{k}(f)$, and 
		\begin{align}
		m(f)(z) = \min \{ k \mid  |v_{k} (f) (z)| > a_{k}    \},~\forall z \in \mathbb{Z}.
		\end{align}
		
		Observe that $\Delta_{k} f = \Xi_{k}(f) - \Xi_{k-1}(f)$ and by definition $\pi_{0}(f) = 0$. Thus, for all $z$, $f(z) = -\pi_{0}(f)(z) + \pi_{1}(f)(z) + f(z) - \pi_{1}(f)(z) = \Delta_{1}(f)(z) + f(z) - \pi_{1}(f)(z)$. Continuing in this fashion,
		\begin{align*}
		f(z) = & \sum_{k=1}^{\infty} \Delta_{k} f (z) \\
		= & \sum_{k=1}^{\infty} \Delta_{k} f (z) 1_{m(f)(z) \geq k} + \sum_{k=1}^{\infty} \Delta_{k} f (z) 1_{m(f)(z) < k} \\
		= & \sum_{k=1}^{\infty} \Delta_{k} f (z) 1_{m(f)(z) \geq k} + \sum_{k=1}^{\infty} (\Xi_{k}(f) -  \Xi_{k-1}(f))) (z) 1_{m(f)(z) < k}\\
		= & \sum_{k=1}^{\infty} \Delta_{k} f (z) 1_{m(f)(z) \geq k} - \Xi_{0} (f)(z)\{ 1_{m(f)(z) < 1}    \} - \Xi_{1}(f)(z)\{ 1_{m(f)(z) < 2} - 1_{m(f)(z) < 1}   \} \\
		& - \Xi_{2}(f)(z)\{ 1_{m(f)(z) < 3} - 1_{m(f)(z) < 2}   \} - ... \\
		= & \sum_{k=1}^{\infty} \Delta_{k} f (z) 1_{m(f)(z) \geq k} - \Xi_{0}(f)(z)\{ 1_{m(f)(z) < 1}    \} - \Xi_{1}(f)(z)\{ 1_{m(f)(z) = 1}   \} \\
		& - \Xi_{2}(f)(z)\{ 1_{m(f)(z) = 2}   \} - ... \\
		= & \sum_{k=1}^{\infty} \Delta_{k} f (z) 1_{m(f)(z) \geq k} - \sum_{k=0}^{\infty} \Xi_{k}(f)(z)1_{m(f)(z) = k}\\
		= &  \sum_{k=1}^{\infty} \Delta_{k} f (z) 1_{m(f)(z) \geq k} - \sum_{k=1}^{\infty} \Xi_{k}(f)(z)1_{m(f)(z) = k} - \Xi_{0}(f)(z) 1_{m(f)(z) = 0},
		\end{align*}
		where the third line follows because $1_{m(f)(z) < k+1} - 1_{m(f)(z) < k} =  1_{m(f)(z) = k}$ and $1_{m(f)(z) < 1} = 1_{m(f)(z)= 0}$.
		
		Observe that $1_{\{m(f)(z) \geq k \cap |v_{k} f(z)| > a_{k}\}} = 1_{\{m(f)(z) = k \cap |v_{k} f(z)| > a_{k}\}}$. And thus $\sum_{k=1}^{\infty} \Delta_{k} f (z) 1_{m(f)(z) \geq k} = \sum_{k=1}^{\infty} \Delta_{k} f (z) 1_{\{m(f)(z) \geq k \cap |v_{k} f(z)| \leq a_{k}\}} + \sum_{k=1}^{\infty} \Delta_{k} f (z) 1_{\{m(f)(z) = k \cap |v_{k} f(z)| > a_{k}\}}$. Therefore
		\begin{align*}
		f(z) = & \sum_{k=1}^{\infty} \Delta_{k} f (z) 1_{\{m(f)(z) \geq k \cap |v_{k} f(z)| \leq a_{k}\}} + \sum_{k=1}^{\infty} \Delta_{k} f (z) 1_{\{m(f)(z) = k \cap |v_{k} f(z)| > a_{k}\}} \\
		& - \sum_{k=1}^{\infty} \Xi_{k}(f)(z)1_{\{m(f)(z) = k \cap |v_{k} f(z)| > a_{k}\}} - \sum_{k=1}^{\infty} \Xi_{k}(f)(z)1_{\{m(f)(z) = k \cap |v_{k} f(z)| \leq a_{k}\}} \\
		& - \Xi_{0}(f)(z) 1_{\{m(f)(z) = 0\}},
		\end{align*}
		
		Since\begin{align*}
		\sum_{k=1}^{\infty} \Delta_{k} f (z) 1_{\{m(f)(z) = k \cap |v_{k} (f)(z)| > a_{k}\}} = \sum_{k=1}^{\infty} \Xi_{k} f (z) 1_{\{m(f)(z) = k \cap |v_{k} f(z)| > a_{k}\}} - \sum_{k=1}^{\infty} \Xi_{k-1} f (z) 1_{\{m(f)(z) = k \cap |v_{k} f(z)| > a_{k}\}},
		\end{align*}
		it follows that 
		\begin{align*}
		& \sum_{k=1}^{\infty} \Delta_{k} f (z) 1_{\{m(f)(z) = k \cap |v_{k} (f)(z)|  > a_{k}\}} - \sum_{k=1}^{\infty} \Xi_{k}(f)(z)1_{\{m(f)(z) = k \cap |v_{k} (f)(z)| > a_{k}\}} \\
		= &  - \sum_{k=1}^{\infty} \Xi_{k-1} f (z) 1_{\{m(f)(z) = k \cap |v_{k} (f)(z)| > a_{k}\}}.
		\end{align*}
		
		Therefore
		\begin{align*}
		f(z) = & \sum_{k=1}^{\infty} \Delta_{k} f (z) 1_{\{m(f)(z) \geq k \cap |v_{k} (f)(z)| \leq a_{k}\}} - \sum_{k=1}^{\infty} \Xi_{k-1}(f) (z) 1_{\{m(f)(z) = k \cap |v_{k} (f)(z)| > a_{k}\}} \\
		&  - \sum_{k=0}^{\infty} \Xi_{k}(f)(z)1_{\{m(f)(z) = k \cap |v_{k} (f)(z)| \leq a_{k}\}} - \Xi_{0}(f)(z) 1_{\{m(f)(z) = 0\}}.
		\end{align*}
		
		Finally, $m(f)(z) = k$ implies that $|v_{k}(f)(z)| > a_{k}$ and thus $1_{\{m(f)(z) = k \cap |v_{k} (f)(z)| \leq a_{k}\}} = 0$.		
	\end{proof}
	
	\begin{proof}[Proof of Lemma \ref{lem:decom-L}]
		
		Let $g_{k}^{1}(f)(z) \equiv \Delta_{k} f (z) 1_{\{m(f,z) \geq k \cap |v_{k} (f)(z)| \leq a_{k}\}}$ and\\ $g_{k}^{2}(f)(z) \equiv (I-\pi_{k-1})(f) (z) 1_{\{m(f,z) = k \cap |v_{k} (f)(z)| > a_{k}\}} $. By lemma \ref{lem:f-decom}, 
		\begin{align}
		\mathbf{L}_{n}(f) = \sum_{k=1}^{\infty} \mathbf{L}_{n} \left( g_{k}^{1}(f) \right) + \sum_{k=1}^{\infty} \mathbf{L}_{n} \left( g_{k}^{2}(f) \right)  -  \mathbf{L}_{n} \left(\Xi_{0}(f)1_{\{ m(f) = 0 \}}\right).  
		\end{align}
		
		Let for all $l=1,2$, 
		\begin{align}
		\Omega_{n}^{l}(u) = \{ \omega \in \Omega \mid \forall f \in \mathcal{F}(A),~ |\sum_{k=1}^{\infty} \mathbf{L}_{n} \left( g_{k}^{l}(f) \right)| \leq u \Gamma_{l} \}.
		\end{align}
		and 
		\begin{align}
		\Omega_{n}^{3}(u) = \{ \omega \in \Omega \mid \forall f \in \mathcal{F}(A),~ |\sum_{k=1}^{\infty} \mathbf{L}_{n} \left( \Xi_{0}(f)1_{\{ m(f) = 0 \}} \right)| \leq u \Gamma_{3} \}.
		\end{align}
		
		It follows that 
		\begin{align*}
		P \left( \sup_{f \in \mathcal{F}(A)} |\mathbf{L}_{n}(f)| \geq u (\Gamma_{1} + \Gamma_{2} + \Gamma_{3})   \right) \leq \sum_{l=1}^{3}\mathbf{P} \left( \Omega \setminus \Omega_{n}^{l}(u) \right).
		\end{align*}
	\end{proof}

 	\begin{proof}[Proof of Lemma \ref{lem:beta-bdd}]
 		Note that $|\mathbf{L}_{n}(g) - \mathbf{L}^{\ast}_{n}(g) | \leq a n^{-1} \sum_{i=1}^{n} \left(  1_{Z_{i} \ne Z^{\ast}_{i}} + \mathbb{P}(Z_{i} \ne Z^{\ast}_{i}) \right)$ for any $g \in B$. Therefore $E_{\mathbb{P}} \left[ \sup_{g \in B}  \sqrt{n} |\mathbf{L}_{n}(g) - \mathbf{L}^{\ast}_{n}(g)|          \right] \leq 2 a n^{-1/2}\sum_{i=1}^{n} \mathbb{P}(Z_{i} \ne Z^{\ast}_{i}) $.
 		
 		By our definition of $Z^{\ast}_{i}$ and $Z_{i}$, $\mathbb{P}(Z_{i} \ne Z^{\ast}_{i}) \leq \mathbb{P}(U_{i}(q) \ne U^{\ast}_{i}(q)) \leq \beta(q)$. 
 			\end{proof}
 	
 	%

 	\begin{proof}[Proof of Lemma \ref{lem:q-norm}]
 		(1) Since $|f|\leq |g|$, it is clear that $H_{f} \leq H_{g}$ and by construction of $Q_{.}$ the result follows.\\
 		
 		(2) Trivial.\\
 		
 		(3) Recall that $t \mapsto \beta(t) \in [0,1]$ is cadlag and non-increasing. Let $\beta^{-1}(u) = \min \{ s \mid \beta (s) \leq u  \}$ for $u \in [0,1]$ and $\beta^{-1}(u) = 0$ for $u > 1$. It is easy to see that $\mu_{q}(u) = \min\{ j(2u), 1+q\}$ where $j(u) = \min \{ s \in \mathbb{N}_{0} \mid u > \beta(s)  \}$.
 		
 		We claim that $\beta^{-1}(u) \geq j(u)-1$. By definition of $j(u)$, $u \leq \beta(j(u)-1)$ and by definition of $\beta^{-1}(u)$, $\beta(\beta^{-1}(u)) \leq u$; thus $\beta(\beta^{-1}(u)) \leq \beta(j(u)-1)$. Since $\beta(.)$ is non-increasing, this implies that $\beta^{-1}(u) \geq j(u)-1$.	
 		
 		Therefore, $\mu_{q}(u) \leq \min\{ \beta^{-1}(2u) + 1 , q+1  \}$. And thus, for any $a \geq 0$
 		\begin{align*}
 		\int_{0}^{1} (\mu_{q}(u))^{a} du \leq& \int_{0}^{1} \min \{ (\beta^{-1}(2u)+1)^{a} , (q+1)^{a} \} du\\
 		\leq &  \int_{0.5\beta(q)}^{1} (\beta^{-1}(2u)+1)^{a}  du + (q+1)^{a} 0.5\beta(q)\\
 		= & 0.5  \int_{\beta(q)}^{2} (\beta^{-1}(v)+1)^{a}  dv + (q+1)^{a} 0.5\beta(q)\\
 		= & 0.5 \left(\int_{\beta(q)}^{1} (\beta^{-1}(v)+1)^{a}  dv + (q+1)^{a} \beta(q) \right) ~(since~\beta^{-1}(u) =0,~for~u>1).	
 		\end{align*}
 		
 		\medskip 
 		
 		(4) We first show that $\beta^{-1}(u) \leq j(u)$. Suppose not, i.e., $\beta^{-1}(u) > j(u)$. By definition of $j(u)$, $u > \beta(j(u))$. But this violates the property that $\beta^{-1}(u)$ is the minimum over all $s$ such that $\beta (s) \leq u$. Therefore, $\mu_{q}(u) \geq \min\{   \beta^{-1}(2u) , q+1  \}$ and thus
 		\begin{align*}
 		\int_{0}^{1} (\mu_{q}(u))^{a} du \geq & \int_{0}^{1} (\min\{ \beta^{-1}(2u)    , q+1  \})^{a} du\\
 		= &  \int_{0.5\beta(q+1)}^{1} (  \beta^{-1}(2u) , 1  )^{a}  du + (q+1)^{a} 0.5\beta(1+q)\\
 		= & 0.5  \int_{\beta(q+1)}^{2} (  \beta^{-1}(v))^{a}  dv + (q+1)^{a} 0.5\beta(1+q)\\
 		= & 0.5 \left(\int_{\beta(q+1)}^{1} (  \beta^{-1}(v) )^{a}  dv + (q+1)^{a} \beta(1+q) \right).
 		\end{align*}
 		
 	\end{proof}

 	\begin{proof}[Proof of Lemma \ref{lem:L1-bdd}]
 		We divide the proof into several steps. \\
 		
 		\textsc{Step 1.} Note that $|| f 1_{\{ f > Q_{f}(\epsilon)  \}}||_{L^{1}} \leq \int_{0}^{\epsilon} Q_{f}(t) dt$ for any $\epsilon \in [0,1)$.\footnote{The case for $\epsilon=0$ comes from the fact that for $f \in \mathcal{F} \subseteq L^{1}$, the RHS is zero since $Q_{f}(0) = \infty$.} For any $q$, by definition of $||.||_{q}$ and the fact that $u \mapsto Q_{f}(u)$ is non-increasing, $||f||_{q} \geq \sqrt{2} \sup_{u \leq t } Q_{f}(u) \sqrt{ \int_{0}^{u}  \mu_{q}(v) dv  }$ for all $ t \in [0,1)$. Therefore $Q_{f}(u) \leq \frac{ ||f||_{q} }{ \sqrt{2} \sqrt{ \int_{0}^{u}  \mu_{q}(v) dv  } } $ for all $u < 1$. Hence, for any $\epsilon \in [0,1)$,
 		\begin{align*}
 		|| f 1_{\{ f > Q_{f}(\epsilon)  \}}||_{L^{1}} \leq ||f||_{q} \int _{0}^{\epsilon} \frac{1}{ \sqrt{2} \sqrt{ \int_{0}^{t}  \mu_{q}(v) dv  } } dt.
 		\end{align*}
 		
 		Let $\epsilon_{k} = n^{-1}\underline{q}_{n,k} 2^{k+1} $ where $\underline{q}_{n,k}$ is maximal element of $\mathcal{Q}_{n}$ lower than $q_{n,k}$. We note that $\epsilon_{k}$ is less than 1 by the definition of $q_{n,k}$ and the fact that $\beta(.) \leq 1$. Clearly, if $\frac{b_{k}}{q_{n,k}} \geq Q_{f}(\epsilon_{k})$, then  
 		\begin{align}
 		|| f 1_{\{ f > b_{k}/q_{n,k}    \}} ||_{L^{1}} \leq  ||f ||_{q_{n,k}} \int _{0}^{\epsilon_{k}} \frac{1}{ \sqrt{2} \sqrt{ \int_{0}^{t}  \mu_{q_{n,k}}(v) dv  } } dt.
 		\end{align}

 		\medskip
 		
 		\textsc{Step 2.} We now establish that in fact $\frac{b_{k}}{q_{n,k}} \geq Q_{f}(\epsilon_{k})$. Since $\mu_{q}$ is non-increasing, if suffices to show that $\frac{b_{k}}{q_{n,k}} \geq \frac{ ||f||_{q} }{ \sqrt{ \mu_{q_{n,k}}(\epsilon_{k}) \epsilon_{k}  } } $.
 		
 		Since $\epsilon_{k} = n^{-1}\underline{q}_{n,k} 2^{k+1} $,
 		\begin{align*}
 		\mu_{q_{n,k}}(\epsilon_{k}) = \mu_{q_{n,k}}(n^{-1}(\underline{q}_{n,k}) 2^{k+1} )  = \sum_{i=0}^{q_{n,k}} 1_{\{ n^{-1}(\underline{q}_{n,k}) 2^{k+1} \leq 0.5 \beta(i)     \}}. 
 		\end{align*}
 		By definition of $q_{n,k}$, $n^{-1}\underline{q}_{n,k} 2^{k+1} < 0.5 \beta(\underline{q}_{n,k}) $. Since $\beta(.)$ is non-increasing, this implies that $n^{-1}\underline{q}_{n,k} 2^{k+1} < 0.5 \beta(i)$ for all $i=0,...,\underline{q}_{n,k}$ and thus $\mu_{q_{n,k}}(\epsilon_{k})  =  \sum_{i=0}^{q_{n,k}} 1_{\{ n^{-1}\underline{q}_{n,k} 2^{k+1} \leq 0.5 \beta(i)     \}} \geq \underline{q}_{n,k}$. This implies that 
 		\begin{align*}
 		\sqrt{ \mu_{q_{n,k}}(\epsilon_{k})  \epsilon_{k}}  \geq  \underline{q}_{n,k} n^{-1/2} \sqrt{2^{k+1}} = \epsilon_{k} \sqrt{\frac{n}{2^{k+1}}}.
 		\end{align*}

 		Since $q_{n,k} \in \mathcal{Q}_{n}$ and $n$ is restricted to be of the form $\prod_{i=1}^{\upsilon} p_{i}^{m_{i}}$ for some integers $(m_{i})_{i}$, it follows by Lemma \ref{lem:DivN} in Appendix \ref{app:domN}, that $q_{n,k} \leq p_{\upsilon} \underline{q}_{n,k}$.

 		Hence $q_{n,k} \leq p_{\upsilon} n \epsilon_{k} /2^{k+1}$, and thus, from the previous display it follows that 
 		
 		\begin{align*}
 		\frac{1}{q_{n,k} }\sqrt{ \mu_{q_{n,k}}(\epsilon_{k})  \epsilon_{k}}  \geq  \frac{2^{k+1}}{ p_{\upsilon} n \epsilon_{k} }\epsilon_{k} \sqrt{\frac{n}{2^{k+1}}} = \frac{1}{p_{\upsilon}} \sqrt{\frac{2^{k+1}}{  n  } },
 		\end{align*}
 		since $b_{k} = B_{0} \sqrt{n} ||f||_{q_{n,k}}  /\sqrt{ 2^{k+1}}$ and $B_{0} = p_{\upsilon}$, it readily follows  that $\frac{b_{k} }{q_{n,k} }\sqrt{ \mu_{q_{n,k}}(\epsilon_{k})  \epsilon_{k}}  \geq ||f||_{q_{n,k}} $.

 		\medskip 
 		
 		\textsc{Step 3.} By step 1-2 we have
 		
 		\begin{align}
 		|| f 1_{\{ f > b_{k}/q_{n,k}    \}} ||_{L^{1}} \leq  \frac{||f ||_{q_{n,k}}}{\sqrt{2}} \int _{0}^{\epsilon_{k}} \frac{1}{ \sqrt{ \int_{0}^{t}  \mu_{q_{n,k}}(v) dv  } } dt.
 		\end{align}
 		
 		Since $ u \mapsto \mu_{q}(u)$ is non-increasing, $t \mapsto t^{-1}\int_{0}^{t}  \mu_{q_{n,k}}(v) dv$ is also non-increasing and thus $\int _{0}^{\epsilon_{k}} \frac{1}{ \sqrt{ \int_{0}^{t}  \mu_{q_{n,k}}(v) dv  } } dt = \int _{0}^{\epsilon_{k}} \sqrt{  \frac{t}{ \int_{0}^{t}  \mu_{q_{n,k}}(v) dv  } } t^{-1/2}dt \leq \sqrt{  \frac{\epsilon_{k}}{ \int_{0}^{\epsilon_{k}}  \mu_{q_{n,k}}(v) dv  } } \int _{0}^{\epsilon_{k}}  t^{-1/2}dt$. Thus
 		the RHS is bounded above by $\frac{||f ||_{q_{n,k}}}{\sqrt{2}} \sqrt{\frac{\epsilon_{k}}{\int_{0}^{\epsilon_{k}}  \mu_{q_{n,k}}(v) dv }   }  \int_{0}^{\epsilon_{k}} \frac{1}{ \sqrt{ t } } dt = \sqrt{2} ||f ||_{q_{n,k}} \frac{\epsilon_{k} }{  \sqrt{  \int_{0}^{\epsilon_{k}}  \mu_{q_{n,k}}(v) dv} } $. By Step 2, $\frac{\epsilon_{k} }{  \sqrt{  \int_{0}^{\epsilon_{k}}  \mu_{q_{n,k}}(v) dv} }  \leq \frac{\epsilon_{k} }{  \sqrt{  \epsilon_{k} \mu_{q_{n,k}}(\epsilon_{k})} } $ . Therefore,
 		\begin{align*}
 		|| f 1_{\{ f > b_{k}/q_{n,k}    \}} ||_{L^{1}} \leq \sqrt{2} ||f ||_{q_{n,k}} \sqrt{ \frac{\epsilon_{k} }{    \mu_{q_{n,k}}(\epsilon_{k})}  }.
 		\end{align*}
 		
 		By the calculations in Step 2, $\sqrt{ \frac{\epsilon_{k} }{    \mu_{q_{n,k}}(\epsilon_{k})}  } \leq \sqrt{ \frac{ n^{-1} \underline{q}_{n,k} 2^{k+1}  }{ \underline{q}_{n,k} } }= n^{-1/2} \sqrt{ 2^{k+1} }$. We thus conclude that $|| f 1_{\{ f > b_{k}/q_{n,k}    \}} ||_{L^{1}} \leq \sqrt{2} ||f ||_{q_{n,k}} n^{-1/2} \sqrt{ 2^{k+1} }$ as desired.
 	\end{proof}

 	\begin{proof}[Proof of Lemma \ref{lem:bere}]
 		We note that since $q \in \mathcal{Q}_{n}$, $J$ is an integer equal to $n/q-1$.  Throughout, we use $J = J(q,n) = n/q-1$.					Let $f \mapsto \delta_{j}(\omega^{\ast}) (f) = q^{-1/2} \sum_{i=qj+1}^{qj+q} \{f(Z^{\ast}_{i}) - E_{\mathbf{P}^{\ast}}[f(Z^{\ast})]\}$ for any $j = 0,1,2,...$. Then, it follows that for any $f \in \mathcal{F}$, 
 		\begin{align*}
 		\mathbf{L}^{\ast}_{n}(f) =& n^{-1} \sum_{i=1}^{n} \bar{f}(Z^{\ast}_{i}) =  n^{-1} \sum_{j=0}^{J}  \sum_{i=jq+1}^{jq+q} \bar{f}(Z^{\ast}_{i})  	= n^{-1} \sqrt{q} \sum_{j=0}^{J} \delta_{j}(\omega^{\ast}) (f)
 		\end{align*}
 		where $\delta_{j}(\omega^{\ast}) (f)$ is measurable with respect to $U^{\ast}_{j}(q)=(Z^{\ast}_{qj+1},...,Z^{\ast}_{qj+q})$. 
 		
 		Note that, for any sequence $(a_{j})_{j} $, $\sum_{j=0}^{J} a_{j} = \sum_{j=0}^{[J/2]}  a_{2j} + \sum_{j=0}^{[(J+1)/2]-1} a_{2j+1}$, then \footnote{Here, $[a] = \max\{ n \in \mathbb{N} : n \leq a  \}$. } 					
 		\begin{align}\label{eqn:bere-1}
 		\sqrt{n}\mathbf{L}^{\ast}_{n}(f) = n^{-1/2} \sqrt{q} \left(   \sum_{j=0}^{[J/2]} \left\{ \delta_{2j}(\omega^{\ast}) (f)   \right\}  + \sum_{j=0}^{[(J+1)/2]-1} \left\{ \delta_{2j+1}(\omega^{\ast}) (f)    \right\}  \right).
 		\end{align}
 		
 		Observe that for any $t>0$, \begin{align}\notag
 		& \mathbf{P} \left(	\sqrt{n}\mathbf{L}^{\ast}_{n}(f) \geq t \right) \\ \notag
 		\leq & \mathbf{P} \left(	n^{-1/2} \sqrt{q}   \sum_{j=0}^{[J/2]} \left\{ \delta_{2j}(\omega^{\ast}) (f)   \right\}  \geq 0.5 t \right) \\ \label{eqn:bere-2}
 		& + \mathbf{P} \left(	n^{-1/2} \sqrt{q}  \sum_{j=0}^{[(J+1)/2]-1} \left\{ \delta_{2j+1}(\omega^{\ast}) (f)   \right\}  \geq 0.5 t \right) .
 		\end{align}
 		
 		Thus, it suffices to bound each term in the RHS separately.

 		By the conditions in Section \ref{sec:prem}, $(U^{\ast}_{2j}(q))_{j \geq 0}$ are i.i.d. (so are $(U^{\ast}_{2j+1}(q))_{j \geq 0}$). Moreover, they  have the same distribution as $(U_{2j}(q))_{j\geq 0 }  $ (also $(U^{\ast}_{2j+1}(q))_{j \geq 0}$ has the same distribution as $(U_{2j+1}(q))_{j \geq 0}$). Thus, we can replace $E_{\mathbf{P}^{\ast}}[.]$ by $E_{\mathbf{P}}[.]$ in the definition of $\delta_{j}(\omega^{\ast})(f)$.
 		
 		It is easy to see that for any $f \in \mathcal{F}$, $||\delta_{j}(\omega^{\ast}) (f)||_{L^{\infty}} \leq \sqrt{q} 2 ||f||_{L^{\infty}}$. Moreover, by \cite{Rio1993} Corollary 1.2 and Lemma 1.2, for any $f \in \mathcal{F}$,		\begin{align*}
 		E_{\mathbf{P}^{\ast}}[(\delta_{0}(\omega^{\ast}) (f))^{2}] = & E_{\mathbf{P}}[(\delta_{0}(\omega) (f) )^{2}] \\
 		= & Var_{\mathbf{P}}(f(Z_{0})) + 2 \sum_{k=1}^{q} (1-k/q) Cov_{\mathbf{P}}(f(Z_{0}),f(Z_{k}))\\
 		\leq & 4 \sum_{k=0}^{q} (1-k/q) \int_{0}^{\alpha_{k}} Q^{2}_{f}(u) du \\
 		\leq & 4 \int_{0}^{1} \left(\sum_{k=0}^{q} (1-k/q)1 \{ u \leq 0.5 \beta(k)   \}    \right) Q^{2}_{f}(u) du\\
 		\leq & 4 \int_{0}^{1} \left(\sum_{k=0}^{q} 1 \{ u \leq 0.5 \beta(k)   \}    \right) Q^{2}_{f}(u) du\\
 		= & 4 \int_{0}^{1} \mu_{q}(u) Q^{2}_{f}(u) du = 2 ||f||_{q}^{2}
 		\end{align*}
 		where $(\alpha_{k})_{k}$ are the strong mixing coefficients and the second line follows because $\alpha_{k} \leq 0.5 \beta(k)$ (e.g. see \cite{Rio1993}).
 		
 		\begin{remark}
 			By setting $\beta(0) = 2$, for the i.i.d. case --- wherein $\beta(q) = 0$ for any $q>0$ --- the display implies that $E_{\mathbf{P}^{\ast}}[(\delta_{0}(\omega^{\ast}) (f))^{2}] \leq 4 \int_{0}^{1} 1 \{ u \leq 1   \}   Q^{2}_{f}(u) du = 4||f||_{L^{2}}^{2}$. Thus, our approach looses a factor of $4$ for the simple i.i.d. case. Hence, in situations where i.i.d.-ness is imposed one could obtain better constants in this lemma.  $\triangle$
 		\end{remark}

 		By Bernstein inequality (see \cite{VdV-W1996} Lemma 2.2.9), it follows that
 		\begin{align*}
 		\mathbf{P}^{\ast} \left(  \sqrt{\frac{q}{n}} \sum_{j=0}^{[J/2]} \left\{ \delta_{2j}(\omega^{\ast}) (g)   \right\} \geq u         \right) \leq & \exp \left\{ - \frac{0.5 q^{-1}n u ^{2}}{\sum_{j=0}^{[J/2]} E_{\mathbf{P}^{\ast}}[(\delta_{2j}(\omega^{\ast}) (g) )^{2}]  + \frac{1}{3} \sqrt{nq^{-1}} u 2 \sqrt{q} ||g||_{L^{\infty}}    }  \right\} \\
 		= & \exp \left\{ - \frac{0.5 q^{-1}n u ^{2}}{\sum_{j=0}^{[J/2]} E_{\mathbf{P}^{\ast}}[(\delta_{2j}(\omega^{\ast}) (g) )^{2}]  + \frac{2}{3} \sqrt{n} u ||g||_{L^{\infty}}    }  \right\} \\
 		= & \exp \left\{ - \frac{0.5 n q^{-1} u ^{2}}{(0.5J+1) 2||g||_{q}^{2}  + \frac{2}{3} \sqrt{n} u   ||g||_{L^{\infty}}    }  \right\}\\
 		\leq & \exp \left\{ - \frac{0.5  u ^{2}}{2||g||_{q}^{2}  + \frac{2}{3} u \frac{q}{\sqrt{n}} ||g||_{L^{\infty}}    }  \right\}\\
 		\leq & \exp \left\{ - \frac{u ^{2}}{4 \left( \mathbf{b}^{2}  + \frac{1}{3} u \frac{q}{\sqrt{n}} ||g||_{L^{\infty}}  \right)  }  \right\}
 		\end{align*}
 		where the second line follows from the fact that $E_{\mathbf{P}^{\ast}}[(\delta_{2j}(\omega^{\ast}) (g) )^{2}] = E_{\mathbf{P}^{\ast}}[(\delta_{0}(\omega^{\ast}) (g) )^{2}]$ (stationarity) and our previous observations; the third line follows from simple algebra and the fact that $J-1 = \frac{n}{q}$ and $n/q \geq 1$; the last one from $||g||_{q} \leq \mathbf{b}$.
 		
 		If $\mathbf{b}^{2} \geq \frac{1}{3} u \frac{q}{\sqrt{n}} ||g||_{L^{\infty}} $, then 
 		\begin{align*}
 		\mathbf{P}^{\ast} \left(  \sqrt{\frac{q}{n}} \sum_{j=0}^{[J/2]} \left\{ \delta_{2j}(\omega^{\ast}) (g)   \right\} \geq u         \right) \leq  \exp \left\{ - \frac{ u ^{2}}{8 \mathbf{b}^{2}     }  \right\}
 		\end{align*}
 		so 
 		\begin{align*}
 		\mathbf{P}^{\ast} \left(  \sqrt{\frac{q}{n}} \sum_{j=0}^{[J/2]} \left\{ \delta_{2j}(\omega^{\ast}) (g)   \right\} \geq t \sqrt{8} \mathbf{b}      \right) \leq  \exp \left\{ - t^{2}      \right\}
 		\end{align*}
 		Otherwise, if  $\mathbf{b}^{2} < \frac{1}{3} u \frac{q}{\sqrt{n}} ||g||_{L^{\infty}} $, then 
 		\begin{align*}
 		\mathbf{P}^{\ast} \left(  \sqrt{\frac{q}{n}} \sum_{j=0}^{[J/2]} \left\{ \delta_{2j}(\omega^{\ast}) (g)   \right\} \geq u         \right) \leq  \exp \left\{ - \frac{ u }{\frac{8}{3} \frac{q}{\sqrt{n}} ||g||_{L^{\infty}} }  \right\}
 		\end{align*}
 		so 
 		\begin{align*}
 		\mathbf{P}^{\ast} \left(  \sqrt{\frac{q}{n}} \sum_{j=0}^{[J/2]} \left\{ \delta_{2j}(\omega^{\ast}) (g)   \right\} \geq t^{2}\frac{8}{3} \frac{q}{\sqrt{n}} ||g||_{L^{\infty}}     \right) \leq  \exp \left\{ - t^{2}      \right\}.
 		\end{align*}
 		
 		We are interested in $t = u 2^{k/2}$ for any $u \geq 1$. Note also that $||g||_{L^{\infty}} \leq B_{0} q^{-1} \sqrt{n} \mathbf{b} \frac{1}{\sqrt{2^{k}}} $. Hence 
 		\begin{align*}
 		&t^{2}\frac{8}{3} \frac{q}{\sqrt{n}} ||g||_{L^{\infty}} 1\{ t\frac{1}{3} \frac{q}{\sqrt{n}} ||g||_{L^{\infty}} > \mathbf{b}^{2}  \} + \sqrt{8}t \mathbf{b} 1\{ t\frac{1}{3} \frac{q}{\sqrt{n}} ||g||_{L^{\infty}} < \mathbf{b}^{2}  \} \\
 		\leq & u^{2} \frac{8 B_{0}}{3} 2^{k/2} \mathbf{b}  1\{ t\frac{1}{3} \frac{q}{\sqrt{n}} ||g||_{L^{\infty}} > \mathbf{b}^{2}  \} + \sqrt{8} u 2^{k/2} \mathbf{b} 1\{ t\frac{1}{3} \frac{q}{\sqrt{n}} ||g||_{L^{\infty}} < \mathbf{b}^{2}  \}\\
 		\leq & u ^{2} \mathbf{b} 2^{k/2} \left( \frac{8 B_{0}}{3}  1\{ t\frac{1}{3} \frac{q}{\sqrt{n}} ||g||_{L^{\infty}} > \mathbf{b}^{2}  \} + \sqrt{8}  1\{ t\frac{1}{3} \frac{q}{\sqrt{n}} ||g||_{L^{\infty}} < \mathbf{b}^{2}  \} \right) \\
 		\leq & u ^{2} \mathbf{b} 2^{k/2} \frac{8 B_{0}}{3},
 		\end{align*}
 		where the third line follows since $u \geq 1$ and since $B_{0} \geq 2$, $8 B_{0}/3 \geq 16/3 \geq \sqrt{8}$.

 		Therefore,
 		\begin{align*}
 		\mathbf{P}^{\ast} \left(  \sqrt{\frac{q}{n}} \sum_{j=0}^{[J/2]} \left\{ \delta_{2j}(\omega^{\ast}) (g)   \right\} \geq  u^{2} 2^{k/2}   \mathbf{b}   \frac{8 B_{0}}{3}  \right) \leq  \exp \left\{ - u^{2} 2^{k}      \right\}.
 		\end{align*}		
 		
 		The same result holds for $\sum_{j=0}^{[(J+1)/2]-1} \left\{ \delta_{2j+1}(\omega^{\ast}) (g)   \right\}$. And thus, by equation \ref{eqn:bere-2}, the desired result follows. 		
 	\end{proof}

 			\section{Proofs of Supplementary Lemmas in Appendix \ref{app:examples} }
 			\label{sup:examples}
 			
 				\begin{proof}[Proof of Lemma \ref{lem:Q-LASSO-1}]
 					
 					\emph{Convexity:} It is well-known (e.g. \cite{Koenker01}) that $\theta \mapsto Q(\theta,\mathbf{P}) = E_{\mathbf{P}}[\phi(Z,\theta)]$ is convex.\\
 					
 					\emph{Continuity:} Note that   $Q(\theta,\mathbf{P}) = \tau E_{\mathbf{P}}[Y - X^{T} \theta] - E_{\mathbf{P}}[ \int_{-\infty}^{X^{T}\theta} (y - X^{T}\theta) F(dy|X) ] = \tau E_{\mathbf{P}}[Y - X^{T} \theta] - E_{\mathbf{P}}[ \int_{-\infty}^{X^{T}\theta} y f(y|X) dy] + E_{\mathbf{P}}[X ^{T}\theta F(X^{T}\theta|X)]$.
 					
 					Clearly $\theta \mapsto E_{\mathbf{P}}[X]^{T}\theta$ is continuous. Also $\theta \mapsto  F(X^{T}\theta|X)$ is continuous a.s. and by the dominated convergence theorem (DCT), $\theta \mapsto  E_{\mathbf{P}}[F(X^{T}\theta|X) X]$ is too. Since $y \mapsto y f(y|X)$ is integrable a.s.; $\theta \mapsto \int_{-\infty}^{X^{T}\theta} y f(y|X) dy$ is continuous a.s. Clearly, $|\int_{-\infty}^{X^{T}\theta} y f(y|X) dy| \leq \int_{-\infty}^{\infty} |y| f(y|X) dy \equiv H(X)$ and $E_{\mathbf{P}}[H(X)] = E_{\mathbf{P}}[|Y|] < \infty$ and thus by the DCT, $\theta \mapsto E_{\mathbf{P}}[ \int_{-\infty}^{X^{T}\theta} y f(y|X) dy]$ is continuous.\\

 					\emph{Coercivity:} We now show that: For all $M > 0 $, $\{ \theta \in \Theta \mid E_{\mathbf{P}}[\phi(Z,\theta)] \leq M  \}$ is bounded. For this, suppose that : For all $(\theta_{n})_{n}$ such that $\theta_{n} \in \Theta$ and $\lim_{n \rightarrow \infty} || \theta_{n} ||_{\ell^{2}} = \infty$, then
 					\begin{align}\label{eqn:Q-LASSO-1}
 					\mathbf{P}( \{ X  : \lim_{n \rightarrow \infty} |X^{T} \theta_{n} | = \infty \} ) > 0.
 					\end{align}
 					(We show this at the end of the proof).
 					
 					We do the proof by contradiction, that is: Suppose there exists a $M_{0} > 0$ such that  $S(M_{0}) \equiv \{ \theta \in \Theta \mid E_{\mathbf{P}}[\phi(Z,\theta)] \leq M_{0}  \}$ is unbounded. That is, there exists a sequence $(\theta_{n})_{n}$ such that $\theta_{n} \in S(M_{0}) $ and $\lim_{n \rightarrow \infty} || \theta_{n} ||_{\ell^{2}} = \infty$.  Take $A$ the set of $X$'s such that $\lim_{n \rightarrow \infty} |X^{T} \theta_{n} | = \infty $; this set has positive probability by \ref{eqn:Q-LASSO-1}. Moreover, for any $x$ in such set $(y - x^{T}\theta_{n}) \rightarrow \pm \infty$ a.s., and thus $\phi(z,\theta)\rightarrow \infty$ a.s. That is, the set $\{ Z \mid \phi(Z,\theta_{n})\rightarrow \infty \}$ has positive measure, thus implying by Fatou's lemma that $E_{\mathbf{P}}[\phi(Z,\theta_{n})] \rightarrow \infty$. But this contradicts the fact that $\theta_{n} \in S(M_{0}) $.

 					We now show display \ref{eqn:Q-LASSO-1}. The idea is that if it does not hold, then the two or more components of $X$ are linearly dependent. Take any $(\theta_{n})_{n}$ in $\Theta$, note that 
 					\begin{align*}
 					E_{\mathbf{P}}[\lim_{n \rightarrow \infty } \frac{|X^{T} \theta_{n} |}{||\theta_{n}||_{\ell^{2}}}] = \lim_{n \rightarrow \infty }	E[ \frac{|X^{T} \theta_{n} |}{||\theta_{n}||_{\ell^{2}}}] \leq e_{max}(E[XX^{T}])
 					\end{align*}
 					because $\frac{|X^{T} \theta_{n} |}{||\theta_{n}||_{\ell^{2}}} \leq e_{max}(XX^{T})$ is integrable. 
 					
 					Also note that \begin{align*}
 					E_{\mathbf{P}}[\lim_{n \rightarrow \infty } \frac{|X^{T} \theta_{n} |}{||\theta_{n}||_{\ell^{2}}}] \geq e_{min}(E[XX^{T}]) > 0.
 					\end{align*}
 					Letting $e \equiv e_{min}(E[XX^{T}])$ and by the fact that $E[X] \leq a  + E[X| X>a] \Pr (X > a)$ for any random variable $X$ and any $a>0$, it follows that  by setting $a = 0.5 e$,
 					\begin{align*}
 					0.5 e \leq&   E_{\mathbf{P}}[\lim_{n \rightarrow \infty } \frac{|X^{T} \theta_{n} |}{||\theta_{n}||_{\ell^{2}}} \mid \lim_{n \rightarrow \infty } \frac{|X^{T} \theta_{n} |}{||\theta_{n}||_{\ell^{2}}} > 0.5e ] \mathbf{P}(\lim_{n \rightarrow \infty } \frac{|X^{T} \theta_{n} |}{||\theta_{n}||_{\ell^{2}}} > 0.5e)\\
 					\leq & (0.5e)^{-1} E_{\mathbf{P}}[ \left( \lim_{n \rightarrow \infty } \frac{|X^{T} \theta_{n} |}{||\theta_{n}||_{\ell^{2}}} \right)^{2}  ] \mathbf{P}(\lim_{n \rightarrow \infty } \frac{|X^{T} \theta_{n} |}{||\theta_{n}||_{\ell^{2}}} > 0.5e)\\
 					=& (0.5e)^{-1} E_{\mathbf{P}}[  \lim_{n \rightarrow \infty } \frac{|X^{T} \theta_{n} |^{2}}{||\theta_{n}||^{2}_{\ell^{2}}}  ] \mathbf{P}(\lim_{n \rightarrow \infty } \frac{|X^{T} \theta_{n} |}{||\theta_{n}||_{\ell^{2}}} > 0.5e)\\
 					=& (0.5e)^{-1} \lim_{n \rightarrow \infty } E_{\mathbf{P}}[   \frac{|X^{T} \theta_{n} |^{2}}{||\theta_{n}||^{2}_{\ell^{2}}}  ] \mathbf{P}(\lim_{n \rightarrow \infty } \frac{|X^{T} \theta_{n} |}{||\theta_{n}||_{\ell^{2}}} > 0.5e),
 					\end{align*} 
 					where the last line follows from the DCT and the fact that $\frac{|X^{T} \theta_{n} |^{2}}{||\theta_{n}||^{2}_{\ell^{2}}} \leq e_{max}(XX^{T})$, which it is integrable as argued above. Since $\lim_{n \rightarrow \infty } E[   \frac{|X^{T} \theta_{n} |^{2}}{||\theta_{n}||^{2}_{\ell^{2}}}  ]  \leq e_{max}(E[XX^{T}]) < \infty$ the display above implies that $\mathbf{P}(\lim_{n \rightarrow \infty } \frac{|X^{T} \theta_{n} |}{||\theta_{n}||_{\ell^{2}}} > 0.5e)>0$.
 					
 					Since this result holds for any sequence and $e>0$, this implies display \ref{eqn:Q-LASSO-1}.\\
 					
 					\emph{Continuous Differentiability:} We now show that $\theta \mapsto E_{\mathbf{P}}[\phi(Z,\theta)]$ is twice-continuously differentiable. Clearly $\theta \mapsto E_{\mathbf{P}}[X]^{T}\theta$ is twice-continuously differentiable. Also, since $\theta \mapsto  \int_{-\infty}^{x^{T}\theta} y f(y|x) dy $ is continuously differentiable (since $\theta \mapsto x^{T}\theta$ is continuously differentiable and $y \mapsto y f(y|x)$ are continuous), then $\theta \mapsto \int  \int_{-\infty}^{x^{T}\theta} y f(y|x) dy \mathbf{P}(dx)$ is continuously differentiable with derivative equal to $x^{T}\theta f(x^{T}\theta|x)$ which is also continuously differentiable under our assumptions. Finally, $\theta \mapsto \theta^{T} \int x \int_{-\infty}^{x^{T}\theta}  f(y|x) dy \mathbf{P}(dx) $ is too by similar arguments.  	
 				\end{proof}

 	 	 \begin{proof}[Proof of Lemma \ref{lem:LASSO-d}]
 	 	 	For part (1), observe that for any $u_{1} \leq u_{2}$ 
 	 	 	\begin{align*}
 	 	 	|u_{1} 1\{ u_{1} \leq 0  \} - u_{2} 1\{ u_{2} \leq 0  \} |= \left\{ \begin{array}{cc}
 	 	 	|u_{1} - u_{2}| & if~u_{1} \leq u_{2} \leq 0 \\
 	 	 	|u_{1} | \leq - u_{1} + u_{2} = |u_{1} - u_{2}| & if~u_{1} \leq 0 \leq u_{2} \\
 	 	 	0 & if~0 \leq u_{1} \leq u_{2}
 	 	 	\end{array}
 	 	 	\right.
 	 	 	\end{align*}
 	 	 	
 	 	 	Thus, 
 	 	 	\begin{align}
 	 	 	|\phi(z,\theta_{1})-\phi(z,\theta_{2})|\leq &	|x^{T}(\theta_{1}-\theta_{2})|(1+\tau) \\
 	 	 	\leq & (1+\tau) \sqrt{  (\theta_{1}-\theta_{2})^{T}xx^{T}(\theta_{1}-\theta_{2})         }\\
 	 	 	\leq & (1+\tau) e_{max}( xx^{T}  ) ||\theta_{1}-\theta_{2}||_{\ell^{2}}.
 	 	 	\end{align}
 	 	 	
 	 	 	Assumption \ref{ass:Lip-phi}(2) holds with $r = \pi_{0}$ by our assumptions in Example \ref{exa:HD-QR}.  			
 	 	 \end{proof}

 			Let $\partial f(x_{0})$ be the set of sub-gradients of function $f$ at $x_{0}$.
 			
 			\begin{lemma}
 				\label{lem:sub-diff}
 				Let $f: \mathbb{R}^{d} \rightarrow \mathbb{R}$ and $g: \mathbb{R}^{d} \rightarrow \mathbb{R}$ where $f$ is convex and twice continuously differentiable and $g$ is convex. Let $h = f + a g$ for $a \geq 0$. Then:
 				
 				(1) \begin{align*}
 				\partial h(x_{0}) = \{ u \in \mathbb{R}^{d} \mid u = \frac{d f(x_{0})}{d x} + a v,~with~v \in \partial g (x_{0})   \} 
 				\end{align*}
 				for any $x_{0} \in \mathbb{R}^{d}$.
 				
 				(2) $x_{0}$ is a minimizer of $h$ iff $ 0 \in \partial h(x_{0})$.
 			\end{lemma}
 			
 			\begin{proof}
 				Both $f$ and $a g$ are proper functions (see \cite{Rockafellar70} p. 24). The domain of $f$ and $g$ coincide (and so does its relative interior see \cite{Rockafellar70} p. 44). By Theorem 23.8 in \cite{Rockafellar70}, $\partial h(x_{0}) = \partial f(x_{0}) + \partial  a g(x_{0}) = \partial f(x_{0}) + a \partial   g(x_{0})$. Since $f$ is differentiable, $\partial f(x_{0}) = \{ \frac{d f(x_{0})}{d x} \}$. 
 				
 				By Theorem 23.5 in \cite{Rockafellar70} $0 \in \partial h(x_{0})$ iff $-h(x_{0})$ achieves its supremum.
 			\end{proof}

 			\begin{proof}[Proof of Lemma \ref{lem:LASSO-delta-bdd}]
 				By lemma \ref{lem:Q-LASSO-1} $Q_{k}(\cdot,\mathbf{P})$ is twice continuously differentiable, and thus
 				\begin{align*}
 				Q_{k}(\theta,\mathbf{P}) - Q_{k}(\theta_{0,k},\mathbf{P}) = & \frac{d Q(\theta_{0,k},\mathbf{P})}{d\theta^{T}}[\theta - \theta_{0,k}] \\
 				& + 0.5 \int_{0}^{1} \frac{d^{2} Q(\theta_{0,k} + t(\theta - \theta_{0,k}),\mathbf{P})}{d\theta^{T} d \theta }[\theta - \theta_{0,k},\theta - \theta_{0,k}]\\
 				& + \lambda_{k} (Pen(\theta) - Pen(\theta_{0,k}))\\
 				= & \frac{d Q(\theta_{0,k},\mathbf{P})}{d\theta^{T}}[\theta - \theta_{0,k}] + \lambda_{k} \mu^{T}(\theta - \theta_{0,k})  \\
 				& + 0.5 \int_{0}^{1} \frac{d^{2} Q(\theta_{0,k} + t(\theta - \theta_{0,k}),\mathbf{P})}{d\theta^{T} d \theta }[\theta - \theta_{0,k},\theta - \theta_{0,k}] dt\\
 				& + \lambda_{k} (Pen(\theta) - Pen(\theta_{0,k})- \mu^{T}(\theta - \theta_{0,k}))
 				\end{align*}
 				where  $\mu \in \partial (Pen(\theta_{0,k}))$ (is well defined since $Pen$ is convex; see Lemma \ref{lem:sub-diff}).
 				Hence $Pen(\theta) - Pen(\theta_{0,k}) - \mu^{T}(\theta - \theta_{0,k}) \geq 0$. Also, $\theta_{0,k}$ is a minimizer of $Q_{k}(\cdot,\mathbf{P})$  and $ \theta \mapsto Q(\theta,\mathbf{P})$ is convex and continuously differentiable (see Lemma \ref{lem:Q-LASSO-1}) and $Pen$ is convex. Thus by lemma \ref{lem:sub-diff}, one can choose $\mu$ such that $\frac{d Q(\theta_{0,k},\mathbf{P})}{d\theta^{T}}[\theta - \theta_{0,k}] + \lambda_{k} \mu^{T}(\theta - \theta_{0,k}) = 0$.
 			\end{proof}

 			\begin{proof}[Proof of Lemma \ref{lem:r-fact}]
 				(1) Suppose not, then there exists a $C_{1} > C_{2}$ but $r_{I}(C_{1}) < r_{I}(C_{2})$. By definition of $r_{I}(C_{1})$ there exists a sequence $(r_{I,m}(C_{1}))_{m}$ such that $r_{I,m}(C_{1}) \in S_{I}(C_{1})$ and $r_{I,m}(C_{1}) \rightarrow r_{I}(C_{1})$. Hence, for sufficiently large $m$, $r_{I,m}(C_{1}) < r_{I}(C_{2})$. 
 				
 				Thus, it must hold that $r_{I,m}(C_{1}) < C_{2} \max_{x \geq 1} \frac{\Gamma_{k}(I_{n,k}(\omega)(r_{I,m}(C_{1})x)) }{r_{I,m}(C_{1})x} < C_{1} \max_{x \geq 1} \frac{\Gamma_{k}(I_{n,k}(\omega)(r_{I}(C_{1})x)) }{r_{I}(C_{1})x}$, but this is a contradiction to the fact that $r_{I,m}(C_{1}) \in S(C_{1})$. An analogous results holds for $r_{J}$.
 				
 				\medskip 
 				
 				(2) Suppose not. Then there exists a $C$ such that $r_{I}(C) > r_{J}(C)$, but then, it must hold that $r_{J}(C) < C \max_{x \geq 1} \frac{\Gamma_{k}(I_{n,k}(\omega)(r_{J}(C)x)) }{r_{J}(C)x}$ (or for a subsequence converging to $r_{J}(C)$). Since $I_{n,k}(\omega)(r) \subseteq J_{n,k}(\omega)(r)$, by definition of $\Gamma_{k}$, $r_{J}(C) < C \max_{x \geq 1} \frac{\Gamma_{k}(J_{n,k}(\omega)(r_{J}(C)x)) }{r_{J}(C)x}$; a contradiction to the fact that $r_{J}(C)$ (or a subsequence) belongs to $S_{J}(C)$.
 				
 				\medskip	
 				
 				(3) Suppose not. Then there exists a $C$ such that $r_{J}(C) > C r_{J}(1)$. By definition of $r_{J}(1)$ there exists a sequence $(r_{J,m}(1))_{m}$ such that $r_{J,m}(1) \in S_{J}(1)$ and $r_{J,m}(1) \rightarrow r_{J}(1)$. That is, for sufficiently large $m$, $r_{J}(C) > C r_{J,m}(1)$. Thus, it must hold that $C r_{J,m}(1) < C \max_{x \geq 1} \frac{\Gamma_{k}(J_{n,k}(\omega)(Cr_{J,m}(1)x)) }{Cr_{J,m}(1)x}$. However 
 				\begin{align*}
 				r_{J,m}(1) \geq  \max_{x \geq 1} \frac{\Gamma_{k}(J_{n,k}(\omega)(r_{J,m}(1)x)) }{r_{J,m}(1)x} \geq   \frac{\Gamma_{k}(J_{n,k}(\omega)(r_{J,m}(1)x')) }{r_{J,m}(1)x'}
 				\end{align*}
 				for any $x' \geq 1$. In particular $x' = z C$ where $z\geq 1$ is such that $\frac{\Gamma_{k}(J_{n,k}(\omega)(r_{J,m}(1)Cz)) }{r_{J,m}(1)Cz}= \max_{x \geq 1} \frac{\Gamma_{k}(J_{n,k}(\omega)(r_{J,m}(1)Cx)) }{r_{J,m}(1)Cx}$. Since $C \geq 1$, $x'\geq 1$ is a valid choice. Thus\begin{align*}
 				C r_{J,m}(1) \geq  C \max_{x \geq 1} \frac{\Gamma_{k}(J_{n,k}(\omega)(r_{J,m}(1)C x)) }{r_{J,m}(1)C x};
 				\end{align*}
 				a contradiction to the fact that $r_{J}(C)$ is minimal. An analogous results holds for $r_{I}$.
 				
 				\medskip
 				
 				(4) By using the same arguments as in (2), $r_{J}(C) \leq \inf_{r \in T_{J}(C)} r $  with $T_{J}(C) \equiv \{ r > 0 \mid  r \geq C \max_{ x \geq 1} \frac{\max\{ xr \Lambda_{n,k}(\omega), B_{n,k}(\omega)\}}{xr}  \}$. Clearly, $\frac{\min\{ xr \Lambda_{n,k}(\omega), B_{n,k}(\omega)\}}{xr} \leq \min \{ \Lambda_{n,k}(\omega), \frac{ B_{n,k}(\omega)}{r} \} $. By applying the arguments in (2) again, it follows that $r_{J}(C)$ is less or equal than
 				\begin{align*}
 				\min \{ r>0  \mid  r \geq C  \min \{ \Lambda_{n,k}(\omega), \frac{ B_{n,k}(\omega)}{r} \}   \}.
 				\end{align*}
 				
 				Consider $t^{\ast} =  \min \{ C \Lambda_{n,k}(\omega) , \sqrt{ C B_{n,k}(\omega)}    \}$. If $C \Lambda_{n,k}(\omega) > \sqrt{ C B_{n,k}(\omega)} $, then $t^{\ast} = \sqrt{ C B_{n,k}(\omega)}  $ and 
 				\begin{align*}
 				t^{\ast} = & \min \{ C \Lambda_{n,k}(\omega), \sqrt{ C B_{n,k}(\omega)} \} \\ 
 				= & C \min \{  \Lambda_{n,k}(\omega), \frac{B_{n,k}(\omega)}{\sqrt{ C B_{n,k}(\omega)}} \} = C \min \{  \Lambda_{n,k}(\omega), \frac{B_{n,k}(\omega)}{t^{\ast}} \}.
 				\end{align*}
 				
 				If $C \Lambda_{n,k}(\omega) < \sqrt{ C B_{n,k}(\omega)} $, $t^{\ast} = C \Lambda_{n,k}(\omega) $, and 
 				\begin{align*}
 				C\min \{ \Lambda_{n,k}(\omega), \frac{ B_{n,k}(\omega)}{t^\ast} \}  = & \min \{ C \Lambda_{n,k}(\omega), \frac{ B_{n,k}(\omega)}{ \Lambda_{n,k}(\omega) } \}\\
 				= &  C \Lambda_{n,k}(\omega) \min \{ 1 , \frac{ B_{n,k}(\omega)}{ C (\Lambda_{n,k}(\omega))^{2} } \}\\
 				= & C \Lambda_{n,k}(\omega) = t^{\ast},
 				\end{align*}
 				because $\frac{ B_{n,k}(\omega)}{ C (\Lambda_{n,k}(\omega))^{2} }  > 1$. Thus $t^{\ast}$ belongs to $r \geq C  \min \{ \Lambda_{n,k}(\omega), \frac{ B_{n,k}(\omega)}{r} \}$, and thus $r_{J}(C) \leq t^{\ast}$.

 			\end{proof}

\end{document}